\documentclass[10pt,reqno]{amsart}
\usepackage{amssymb,mathrsfs,color}

\usepackage{enumerate}
\usepackage{graphicx}
\usepackage{xcolor} 
\usepackage[letterpaper, margin=1in]{geometry}
\usepackage{amsfonts,amssymb,amsmath}
\usepackage{slashed}
\usepackage{ mathrsfs }
\usepackage[titletoc,title]{appendix}
\usepackage{wrapfig}

\usepackage{cancel}

\usepackage{cite}

\usepackage{nth}

\usepackage{marginnote}

\usepackage{pgfplots}

\usepackage{verbatim}

\definecolor{green}{rgb}{0,0.8,0} 

\definecolor{deepgreen}{cmyk}{1,0,1,0.5}

\newcommand{\Del}[1]{}

\numberwithin{equation}{section}

\newtheorem{theorem}{Theorem}[section]
\newtheorem{lemma}[theorem]{Lemma}
\newtheorem{proposition}[theorem]{Proposition}
\newtheorem{remark}[theorem]{Remark}

\newcommand{\jap}[1]{\langle #1\rangle}

\renewcommand{\Re}{\mathrm{Re}}

\newcommand{\sign}{\operatorname{sign}}

\renewcommand{\hbar}{{\underline h}}

\newcommand{\bbC}{\mathbb C}

\newcommand{\bbR}{\mathbb R}

\newcommand{\bbZ}{\mathbb Z}

\newcommand{\calC}{\mathcal C}

\newcommand{\calF}{\mathcal F}

\newcommand{\calI}{\mathcal I}
\newcommand{\calJ}{\mathcal J}

\newcommand{\calO}{\mathcal O}

\newcommand{\calQ}{\mathcal Q}
\newcommand{\calR}{\mathcal R}

\newcommand{\calV}{\mathcal V}

\newcommand{\px}{\partial_x}

\newcommand{\pt}{\partial_t}

\newcommand{\jn}{\jap{\nabla}}
\newcommand{\jt}{\jap{t}}
\newcommand{\js}{\jap{s}}
\newcommand{\jx}{\jap{x}}
\newcommand{\jxi}{\jap{\xi}}
\newcommand{\hf}{\frac{1}{2}}
\newcommand{\thf}{\frac{3}{2}}
\newcommand{\Hess}{\mathrm{Hess}}

\newcommand{\ud}{\mathrm{d}}

\DeclareMathOperator{\sech}{sech}

\begin{document}

\title[Asymptotics for 1D KG equations with variable coefficient quadratic nonlinearities]{Asymptotics for 1D Klein-Gordon equations with \\ variable coefficient quadratic nonlinearities}

\author[H. Lindblad]{Hans Lindblad}
\address{Department of Mathematics \\ Johns Hopkins University \\ Baltimore, MD 21218, USA}
\email{lindblad@math.jhu.edu}

\author[J. L\"uhrmann]{Jonas L\"uhrmann}
\address{Department of Mathematics \\ Texas A\&M University \\ College Station, TX 77843, USA}
\email{luhrmann@math.tamu.edu}

\author[A. Soffer]{Avy Soffer}
\address{Mathematics Department, Rutgers University, New Brunswick, NJ 08903, USA}
\email{soffer@math.rutgers.edu}

\thanks{
H. Lindblad was supported in part by NSF grant DMS-1500925 and by Simons Foundation Collaboration grant 638955. 
J. L\"uhrmann was supported in part by NSF grant DMS-1954707 during the completion of this work. 
A. Soffer was supported in part by NSF grant DMS-1600749 and by NSFC11671163. 
Part of this work was conducted while the last two authors were visiting Central China Normal University, Wuhan, China.}

\begin{abstract}
 We initiate the study of the asymptotic behavior of small solutions to one-dimensional Klein-Gordon equations with variable coefficient quadratic nonlinearities. The main discovery in this work is a striking resonant interaction between specific spatial frequencies of the variable coefficient and the temporal oscillations of the solutions. In the resonant case a novel type of modified scattering behavior occurs that exhibits a logarithmic slow-down of the decay rate along certain rays. In the non-resonant case we introduce a new variable coefficient quadratic normal form and establish sharp decay estimates and asymptotics in the presence of a critically dispersing constant coefficient cubic nonlinearity.
 The Klein-Gordon models considered in this paper are motivated by the study of the asymptotic stability of kink solutions to classical nonlinear scalar field equations on the real line.
\end{abstract}

\maketitle 
\tableofcontents

\section{Introduction}

We consider the Cauchy problem for Klein-Gordon equations in one space dimension with variable coefficient quadratic nonlinearities of the form
\begin{equation} \label{equ:nlkg}
 \left\{ \begin{aligned}
  (\partial_t^2 - \partial_x^2 + 1) u &= \alpha(x) u^2 + \beta_0 u^3 + \beta(x) u^3 \text{ on } \bbR^{1+1}, \\
  (u, \partial_t u)|_{t=0} &= (u_0, u_1),
 \end{aligned} \right.
\end{equation}
where $\alpha(x)$ and $\beta(x)$ are smooth and decaying functions, $\beta_0 \in \bbR$, and the initial data $(u_0, u_1)$ are assumed to be real-valued, smooth, and sufficiently decaying. Global existence of solutions to~\eqref{equ:nlkg} follows readily from energy conservation for small smooth initial data. The goal of our investigation is to establish sharp decay estimates and asymptotics for small solutions to~\eqref{equ:nlkg}. A striking discovery in this work is a delicate resonant interaction between specific spatial frequencies of the variable coefficient and the temporal oscillations of the solutions.

\medskip 

\subsection{Motivation}
The main motivation for our investigation stems from the asymptotic stability problem for kink solutions occurring in classical nonlinear scalar field equations on the real line.
The two most well-known examples are the $\phi^4$ model
\begin{equation} \label{equ:phi4_equation}
 (\pt^2 - \px^2) \phi = \phi - \phi^3 \text{ on } \bbR^{1+1},
\end{equation}
and the sine-Gordon equation
\begin{equation} \label{equ:sineGordon_equation}
 (\pt^2 - \px^2) \psi = -\sin(\psi) \text{ on } \bbR^{1+1}.
\end{equation}
These admit special static solutions, called \emph{kinks}, that are explicitly given by
\begin{equation} \label{equ:kinks}
 \phi_0(x) = \tanh( {\textstyle \frac{x}{\sqrt{2}} } ), \quad \text{respectively} \quad \psi_0(x) = 4 \arctan(e^x).
\end{equation}
Kinks are the simplest (one-dimensional) examples of topological solitons, we refer to~\cite{MantSut04, Vachaspati06} for more background.
Note that the Lorentz invariance and the translation invariance of the field equations~\eqref{equ:phi4_equation} and~\eqref{equ:sineGordon_equation} also give rise to moving and shifted versions of the kinks~\eqref{equ:kinks}.

It is a classical problem to understand the asymptotic stability of the kinks $\phi_0(x)$ and $\psi_0(x)$ under general smooth perturbations.
The basic strategy of a perturbative approach to this problem is to decompose the evolution of a perturbation into a modulated kink, which takes into account the symmetries of the equation, and a small residue term. 
The asymptotic stability problem then largely amounts to studying the long-time behavior of the modulation parameters and of the residue term.
This strategy is generally easier to implement in higher space dimensions due to stronger decay properties.
In the following discussion we ignore for simplicity the issue of modulation.
Then for the $\phi^4$ model, the remainder term $u(t,x) = \phi(t,x) - \phi_0(x)$ satisfies the equation
\begin{equation} \label{equ:phi4_perturbation}
 \bigl( \partial_t^2 - \partial_x^2 + 2 - 3 \sech^2( {\textstyle \frac{x}{\sqrt{2}} }) \bigr) u = - 3 \tanh( {\textstyle \frac{x}{\sqrt{2}} } ) u^2 - u^3,
\end{equation}
while for the sine-Gordon equation the residue term $w(t,x) = \psi(t,x) - \psi_0(x)$ is a solution to the equation
\begin{equation} \label{equ:sineGordon_perturbation}
 \begin{aligned}
  \bigl( \partial_t^2 - \partial_x^2 + 1 - 2 \sech^2(x) \bigr) w &= -\sech(x) \tanh(x) w^2 + \frac{1}{6} w^3 - \frac{1}{3} \sech^2(x) w^3 + \bigl\{\text{higher order}\bigr\}.   
 \end{aligned}
\end{equation}
A salient feature of these nonlinear Klein-Gordon equations is the presence of {\it variable coefficient quadratic nonlinearities}.

The analysis of the asymptotic behavior of small solutions to~\eqref{equ:phi4_perturbation} and to~\eqref{equ:sineGordon_perturbation} poses several significant difficulties. 
Due to the slow decay of Klein-Gordon waves in one space dimension, constant coefficient quadratic and cubic nonlinearities are known to exhibit long-range effects leading to a modified scattering behavior of the solutions. The variable coefficient quadratic nonlinearities compound the treatment of these long-range effects and at the same time introduce new resonance phenomena. Moreover, the linearized operators on the left-hand sides of~\eqref{equ:phi4_perturbation} and~\eqref{equ:sineGordon_perturbation} have threshold resonances, i.e. resonances at the bottom of their continuous spectra, and in the case~\eqref{equ:phi4_perturbation} of the $\phi^4$ model an internal oscillation mode, i.e. a single, positive, discrete eigenvalue below the continuous spectrum.

In this paper we initiate the study of decay and asymptotics of small solutions to Klein-Gordon equations with variable coefficient quadratic nonlinearities of the form~\eqref{equ:nlkg}.  
We consider this investigation an important step towards obtaining a better understanding of the long-time behavior of perturbations of the kinks $\phi_0(x)$ and $\psi_0(x)$. 
While the equation~\eqref{equ:nlkg} does not yet include a linear potential, we contend that~\eqref{equ:nlkg} retains some of the key difficulties of the linearized operators in~\eqref{equ:phi4_perturbation} and in~\eqref{equ:sineGordon_perturbation}, because the free Klein-Gordon operator in one space dimension also has a threshold resonance.
We emphasize that the leading order nonlinearities on the right-hand side of the equation~\eqref{equ:sineGordon_perturbation} for perturbations of the sine-Gordon kink fall exactly into the class of nonlinearities of~\eqref{equ:nlkg} that are considered in this paper, see also Remark~\ref{rem:relevance_sineGordon}. In a future investigation we will consider non-localized quadratic variable coefficients $\alpha(x)$ that can assume different limits as $x \to \pm \infty$ like the coefficient $\tanh(x)$ in~\eqref{equ:phi4_perturbation} for perturbations of the kink of the $\phi^4$ model.

To conclude this subsection, we briefly summarize orbital and asymptotic stability results for kinks.
The orbital stability of kinks with respect to small perturbations in the energy space was obtained by Henry-Perez-Wreszinski~\cite{HPW82} for general one-dimensional scalar field theories. 
For a class of scalar field models with a certain flatness assumption on the potential near the wells and under suitable spectral assumptions (no resonances, possible presence of an internal mode), Komech-Kopylova~\cite{KK11_1, KK11_2} proved the asymptotic stability of kinks with respect to a weighted energy norm (pointwise in time).
In a remarkable work~\cite{KMM17}, Kowalczyk-Martel-Mu\~{n}oz established the asymptotic stability of the kink $\phi_0(x)$ of the $\phi^4$ model with respect to a local energy norm in the special case of odd, finite energy perturbations. 
The asymptotic stability of the kink $\psi_0(x)$ of the sine-Gordon equation is not possible in the energy space due to the existence of a wobbling kink solution around $\psi_0(x)$, see for instance~\cite[Remark 1.3]{KMM17}. However, Alejo-Mu\~{n}oz-Palacios~\cite{AMP20}  constructed a smooth infinite co-dimensional manifold of initial data close to the kink $\psi_0(x)$, for which there is asymptotic stability in the energy space. After completion of this work, Delort-Masmoudi~\cite{DelMas20} very recently proved long-time dispersive estimates for odd weighted perturbations of the kink of the $\phi^4$ model up to times $T \sim \varepsilon^{-4+c}$, for arbitrary $c>0$, where $\varepsilon$ measures the size of the initial data in a weighted Sobolev space. 
A sufficient condition for the asymptotic stability locally in the energy space of (moving) kinks in general $(1+1)$-scalar field models under arbitrary small finite energy perturbations has been introduced by Kowalczyk-Martel-Mu\~{n}oz-Van den Bosch~\cite{KMMV20}. 
Moreover, Chen-Liu-Lu~\cite{CLL20} very recently showed that the sine-Gordon kink is asymptotically stable under sufficiently strong weighted perturbations, relying on the complete integrability of the sine-Gordon model and using the nonlinear steepest descent method. 
For related asymptotic stability results we refer to the surveys~\cite{S06, Tao09, KMM17_1} and to the references therein.

\subsection{Main results}

Let $u(t)$ be the solution to~\eqref{equ:nlkg}. In the remainder of this paper we work at the level of the variable
\begin{equation*}
 v(t) := \frac{1}{2} \bigl( u(t) - i \jn^{-1} \pt u \bigr)
\end{equation*}
that satisfies the first-order equation 
\begin{equation} \label{equ:intro_first_order_kg}
  (\pt - i \jn) v = \frac{1}{2i} \jn^{-1} \bigl( \alpha(\cdot) u^2 + \beta_0 u^3 + \beta(\cdot) u^3 \bigr) \text{ on } \bbR^{1+1}
\end{equation}
with initial datum $v(0) = v_0 := \frac{1}{2} (u_0 - i \jn^{-1} u_1)$. It suffices to derive decay estimates and asymptotics for~$v(t)$ since we have that
\begin{equation} \label{equ:u_equ_v_plus_vbar}
 u(t) = v(t) + \bar{v}(t).
\end{equation}
We emphasize that we will frequently use~\eqref{equ:u_equ_v_plus_vbar} as a convenient short-hand notation. 

\medskip 

A key discovery in this work is a delicate resonant interaction that can occur in the variable coefficient quadratic nonlinearity between the spatial frequencies $\xi = \pm \sqrt{3}$ of the variable coefficient and the temporal oscillations of the solution. Correspondingly, we distinguish between the resonant case ($\widehat{\alpha}(+\sqrt{3}) \neq 0$ or $\widehat{\alpha}(-\sqrt{3}) \neq 0$) and the non-resonant case ($\widehat{\alpha}(+\sqrt{3}) = 0$ and $\widehat{\alpha}(-\sqrt{3}) = 0$).

\medskip 

Our first theorem pertains to the resonant case and uncovers a novel modified scattering behavior of small solutions to~\eqref{equ:intro_first_order_kg} when $\beta_0 = \beta(x) = 0$.
\begin{theorem}[Resonant Case] \label{thm:resonant}
 Let $\alpha(x)$ be a smooth function satisfying $\| \jx^8 \alpha(x) \|_{H^3_x} < \infty$. Suppose that 
 \begin{equation*}
  \widehat{\alpha}(+\sqrt{3}) \neq 0 \quad \text{ or } \quad \widehat{\alpha}(-\sqrt{3}) \neq 0.
 \end{equation*} 
 Then there exists a small absolute constant $\varepsilon_0 > 0$ so that for any initial datum $v_0$ with 
 \begin{equation*}
  \varepsilon := \| \jx^2 v_0 \|_{H^4_x} \leq \varepsilon_0,
 \end{equation*}
 there exists a unique global solution $v(t)$ to
 \begin{equation} \label{equ:resonant_nlkg}
  \left \{ \begin{aligned}
   (\pt - i \jn) v &= \frac{1}{2i} \jn^{-1} \bigl( \alpha(\cdot) u^2 \bigr) \text{ on } \bbR^{1+1}, \\
   v(0) &= v_0,
  \end{aligned} \right.
 \end{equation}
 satisfying the decay estimate
 \begin{equation} \label{equ:resonant_thm_sharp_decay}
  \|v(t)\|_{L^\infty_x(\bbR)} \leq C \frac{\log(1 + \jt)}{\jt^{\hf}} \varepsilon.
 \end{equation}
 Moreover, the solution $v(t)$ admits a decomposition 
 \begin{equation} \label{equ:resonant_thm_solution_decomp}
  v(t) = v_{free}(t) + v_{mod}(t), \quad t \geq 1,
 \end{equation}
 with the following properties: 
 \begin{itemize}
  \item There exists $\widehat{V} \in L^\infty$ with $\|\widehat{V}\|_{L^\infty} \lesssim \varepsilon$ such that 
 \begin{equation} \label{equ:resonant_thm_asymptotics_vfree}
  v_{free}(t,x) = \frac{1}{t^\hf} e^{i\frac{\pi}{4}} e^{i\rho} \widehat{V} \Bigl(-\frac{x}{\rho}\Bigr) \theta\Bigl( \frac{x}{t} \Bigr) + \calO \Bigl( \frac{\varepsilon}{t^{\frac{5}{8}-}} \Bigr), \quad t \geq 1,
 \end{equation} 
 where $\rho := (t^2-x^2)^{\hf}$, $\theta(z) = 1$ for $|z| < 1$, and $\theta(z) = 0$ for $|z| \geq 1$. 
 \item There exists $a_0 \in \bbC$ with $|a_0| \lesssim \varepsilon$ so that $v_{mod}(t)$ is given by
  \begin{equation}
   v_{mod}(t,x) = \frac{a_0^2}{2} \int_1^t \bigl( e^{i(t-s) \jn} \jn^{-1} \alpha \bigr)(x) \frac{e^{2is}}{s} \, \ud s.
  \end{equation}
  For any given $\delta > 0$, there exists a constant $C_\delta > 0$ such that we have uniformly
  \begin{equation} \label{equ:resonant_thm_decay_off_rays}
   \bigl| v_{mod}(t,x) \bigr| \leq C_\delta \frac{\varepsilon^2}{\jap{t}^{\frac{1}{2}}} \qquad \text{whenever} \quad |x| < \Bigl( \frac{\sqrt{3}}{2} - \delta \Bigr) t \quad \text{ or } \quad |x| > \Bigl( \frac{\sqrt{3}}{2} + \delta \Bigr) t
  \end{equation}
 and along the rays $x = \pm \frac{\sqrt{3}}{2} t$ the asymptotics of $v_{mod}(t)$ are given by
 \begin{equation} \label{equ:resonant_thm_asymptotics_along_special_rays}
  v_{mod}\Bigl(t, \pm \frac{\sqrt{3}}{2} t \Bigr) = \frac{a_0^2}{\sqrt{8}}  e^{i\frac{\pi}{4}} e^{i \frac{t}{2}} \widehat{\alpha}(\mp \sqrt{3}) \frac{\log(t)}{t^{\frac{1}{2}}} + \calO \Bigl( \frac{\varepsilon^2}{t^\hf} \Bigr), \quad t \gg 1.
 \end{equation} 
 In particular, the decay estimate~\eqref{equ:resonant_thm_sharp_decay} is sharp. 
 \end{itemize}
 A decomposition analogous to~\eqref{equ:resonant_thm_solution_decomp} holds for negative times $t \leq -1$.
\end{theorem}

\begin{remark}
 The amplitude $a_0 \in \bbC$ in the statement of Theorem~\ref{thm:resonant} is explicitly given by
 \begin{equation} \label{equ:intro_formula_a0}
 \begin{aligned}
  a_0 &= \hat{v}_0(0) + \frac{1}{\sqrt{2\pi}} \biggl( \frac{1}{2} \int_{\bbR} \alpha(x) v(0,x)^2 \, \ud x - \int_{\bbR} \alpha(x) |v(0,x)|^2 \, \ud x - \frac{1}{6} \int_{\bbR} \alpha(x) \bar{v}(0,x)^2 \, \ud x \biggr) \\
  &\quad \quad \quad \, \, \, + \frac{1}{\sqrt{2\pi}} \biggl( \int_0^\infty e^{+is} \int_{\bbR} \alpha(x) \partial_s \bigl( e^{-is} v(s,x) \bigr) \bigl( e^{-is} v(s,x) \bigr) \, \ud x \, \ud s \\
  &\qquad \qquad \qquad \qquad - \int_0^\infty e^{-is} \int_{\bbR} \alpha(x) \partial_s \Bigl( \bigl( e^{-is} v(s,x) \bigr) \bigl( e^{+is} \bar{v}(s,x) \bigr) \Bigr) \, \ud x \, \ud s \\
  &\qquad \qquad \qquad \qquad - \frac{1}{3} \int_0^\infty e^{-3is} \int_{\bbR} \alpha(x) \partial_s \bigl( e^{+is} \bar{v}(s,x) \bigr) \bigl( e^{+is} \bar{v}(s,x) \bigr) \, \ud x \, \ud s \biggr).
 \end{aligned}
 \end{equation}
\end{remark}

\begin{remark}
 In the special case where the initial datum $v_0$ is odd and the variable coefficient $\alpha(x)$ is odd, which implies that the solutions are odd, it is evident from~\eqref{equ:intro_formula_a0} that $a_0 = 0$. Thus, in that case no resonance occurs. At a more technical level this follows from the fact that $v(t,0) = 0$ for odd solutions and that therefore the entire variable coefficient quadratic term
 \begin{equation*}
  \alpha(x) v(t,x)^2  = \alpha(x) v(t,x)^2 - \alpha(x) v(t,0)^2 \sim x \alpha(x) (\px v)(t) v(t)
 \end{equation*}
 has stronger time decay due to the improved local decay of spatial derivatives of the solution. See also the discussion of the main ideas of the proof of Theorem~\ref{thm:resonant} in Subsection~\ref{subsubsec:resonant_ideas} below.
\end{remark}

\begin{remark}
 An inspection of the proof of Theorem~\ref{thm:resonant} shows that the limit profile $\widehat{V}$ in~\eqref{equ:resonant_thm_asymptotics_vfree} in fact satisfies $\widehat{V} \in L^\infty \cap L^2$.
\end{remark}

\begin{remark} \label{rem:intro_mention_LLSS}
 After completion of this work, the authors in collaboration with Schlag~\cite{LLSS} generalized Theorem~\ref{thm:resonant} to the study of the asymptotic behavior of small global solutions to the following one-dimensional quadratic Klein-Gordon model with a linear potential
 \begin{equation} \label{equ:intro_more_general_KG_model}
  (\pt^2 - \px^2 + 1 + V(x)) u = P_c ( \alpha(\cdot) u^2 ) \text{ on } \bbR^{1+1}
 \end{equation}
 for a spatially localized variable coefficient $\alpha(x)$. The core assumption in~\cite{LLSS} is that the linear potential $V(x)$ is non-generic, in other words that the Schr\"odinger operator $H = - \px^2 + V(x)$ has a zero energy resonance, i.e., that there exists a bounded non-trivial solution $\varphi(x)$ to $H \varphi = 0$ such that $\varphi(x) \to 1$ as $x \to \infty$ and $\varphi(x) \to c \neq 0$ as $x \to -\infty$. It is worth to emphasize that the Laplacian $-\px^2$ in one space dimension exhibits a zero energy resonance, namely the constant function $1$. 
 Denoting by $\widetilde{\mathcal{F}}$ the distorted Fourier transform associated with the Schr\"odinger operator $H$, \cite[Theorem 1.1]{LLSS} establishes the same type of modified scattering behavior as exhibited in Theorem~\ref{thm:resonant} in this paper for the more general Klein-Gordon model~\eqref{equ:intro_more_general_KG_model} under the corresponding resonance assumption $\widetilde{\mathcal{F}}[\alpha \varphi^2](+\sqrt{3}) \neq 0$ or $\widetilde{\mathcal{F}}[\alpha \varphi^2](-\sqrt{3}) \neq 0$. The work~\cite{LLSS} further clarifies the role that the threshold resonance of the linear Klein-Gordon propagator and the local decay properties of the associated Klein-Gordon waves play for the occurrence of the type of modified scattering behavior uncovered in Theorem~\ref{thm:resonant} for the special flat case $V(x) = 0$.
\end{remark}

\begin{remark}
 It is straightforward to extend the proof of Theorem~\ref{thm:resonant} to also include a variable coefficient cubic nonlinearity $\beta(x) u^3$ for a smooth and sufficiently decaying coefficient $\beta(x)$. A very interesting next step is to generalize the result of Theorem~\ref{thm:resonant} to include a constant coefficient quadratic or cubic nonlinearity in the right-hand side of the Klein-Gordon model~\eqref{equ:resonant_nlkg}. In order to capture the long-range effects of the constant coefficient quadratic or cubic nonlinearity, one would have to propagate a slow growth estimate for a quantity like $\|\jxi^2 \partial_\xi \hat{f}(t)\|_{L^2_\xi}$, where $f(t) := e^{-it\jap{\nabla}} v(t)$ denotes the profile of the solution $v(t)$. This is achieved in the second result of this paper in Theorem~\ref{thm:nonresonant} below under the additional non-resonance assumption $\widehat{\alpha}(\pm \sqrt{3}) = 0$. However, in the general case a quick computation shows that such a quantity is, e.g., strongly divergent for the profile of $v_{mod}(t)$. This step therefore needs significant further ideas. 
 
 In this regard we emphasize that after completion of this paper the remarkable recent work of Germain-Pusateri~\cite{GP20} considered the long-time behavior of small solutions to the general one-dimensional Klein-Gordon equation 
 \begin{equation} \label{equ:intro_resonant_thm_GP}
  (\pt^2 - \px^2 + 1 + V(x)) u = a(x) u^2 \text{ on } \bbR^{1+1},
 \end{equation}
 where the variable coefficient $a(x)$ is assumed to satisfy $a(x) \to \ell_{\pm \infty}$ as $x \to \pm \infty$ for arbitrary fixed $\ell_{\pm \infty} \in \bbR$ (and is thus not necessarily spatially localized) and where $V(x)$ is a sufficiently regular and decaying linear potential with no bound states. Under the key spectral assumption that the distorted Fourier transform $\tilde{u}(t,0) = 0$ of the solution vanishes at zero frequency at all times, \cite[Theorem 1.1]{GP20} establishes modified scattering in the sense that small solutions to~\eqref{equ:intro_resonant_thm_GP} decay in $L^\infty_x$ at the rate $t^{-\hf}$ of free Klein-Gordon waves and that their asymptotics involve logarithmic phase corrections (that are ``caused by the non-zero limits'' $\ell_{\pm\infty}$ at spatial infinity of the coefficient $a(x)$). The condition $\tilde{u}(t,0) = 0$ holds automatically for generic potentials, but only under additional assumptions for non-generic potentials, for instance by imposing suitable parity conditions. We note that the flat potential $V(x) = 0$ in one space dimension is non-generic and that~\cite[Theorem 1.1]{GP20} would in effect only pertain to odd solutions in the flat case. The work~\cite{GP20} further clarifies that the special frequencies $\pm \sqrt{3}$ in Theorem~\ref{thm:resonant} are the (distorted) output frequencies of a nonlinear space-time resonance. Such a fully coherent phenomenon is expected to generally occur for quadratic interactions in one space dimension in the presence of a linear potential $V(x)$ and already in the presence of a variable coefficient quadratic nonlinearity $\alpha(x) u^2$ in the flat case $V(x) = 0$, because the latter cause a decorrelation between the input and the output frequencies. The previously mentioned difficulties with slow $L^2_\xi$ growth estimates for a derivative of the distorted Fourier transform of the profile of the solution are overcome in~\cite{GP20} by introducing an adapted functional framework, which penalizes around the special (distorted) frequencies $\pm \sqrt{3}$, and by crucially exploiting the stronger local decay properties of solutions that are available if one imposes the vanishing condition $\tilde{u}(t,0) = 0$ at zero frequency.  
\end{remark}

Our second theorem establishes asymptotics for small solutions to~\eqref{equ:intro_first_order_kg} in the non-resonant case.

\begin{theorem}[Non-Resonant Case] \label{thm:nonresonant}
 Let $\alpha(x)$ and $\beta(x)$ be smooth and decaying coefficients satisfying $\| \jx^4 \alpha(x) \|_{H^2_x} < \infty$, respectively $\|\jx^3 \beta(x)\|_{H^1_x} < \infty$, and let $\beta_0 \in \bbR$. Assume that 
 \begin{equation*}
  \widehat{\alpha}(+\sqrt{3}) = 0 \quad \text{ and } \quad \widehat{\alpha}(-\sqrt{3}) = 0.
 \end{equation*}
 Then there exists a small absolute constant $\varepsilon_0 > 0$ with the following property: For any initial datum $v_0$ with
 \begin{equation*}
  \varepsilon := \bigl\| \jx v_0 \bigr\|_{H^2_x} \leq \varepsilon_0,
 \end{equation*}
 there exists a unique global solution $v(t)$ to
 \begin{equation} \label{equ:first_order_kg}
  \left \{ \begin{aligned}
   (\pt - i \jn) v &= \frac{1}{2i} \jn^{-1} \bigl( \alpha(\cdot) u^2 + \beta_0 u^3 + \beta(\cdot) u^3 \bigr) \text{ on } \bbR^{1+1}, \\
   v(0) &= v_0,
  \end{aligned} \right.
 \end{equation}
 satisfying the decay estimate
 \begin{equation} \label{equ:thm_sharp_decay}
  \|v(t)\|_{L^\infty_x(\bbR)} \lesssim \frac{\varepsilon}{\jt^{\hf}}.
 \end{equation}
 Moreover, there exists a unique final state $\widehat{W} \in L^\infty$ with $\|\widehat{W}\|_{L^\infty} \lesssim \varepsilon$ such that 
 \begin{equation} \label{equ:thm_asymptotics}
  v(t,x) = \frac{1}{t^\hf} e^{i\frac{\pi}{4}} e^{i \rho} e^{-i \frac{3 \beta_0}{2} \jap{\frac{x}{\rho}}^{-1} | \widehat{W} (-\frac{x}{\rho}) |^2 \log(t) } \widehat{W} \Bigl(-\frac{x}{\rho}\Bigr) \theta\Bigl( \frac{x}{t} \Bigr) + \calO\Bigl( \frac{\varepsilon}{t^{\hf + \nu}} \Bigr), \quad t \geq 1,
 \end{equation} 
 where $\rho := (t^2-x^2)^{\hf}$, $\theta(z) = 1$ for $|z| < 1$, $\theta(z) = 0$ for $|z| \geq 1$, and $0 < \nu \ll 1$ is a small constant.
\end{theorem}

\begin{remark}
 An inspection of the proof of Theorem~\ref{thm:nonresonant} shows that the limit profile $\widehat{W}$ in~\eqref{equ:thm_asymptotics} in fact satisfies $\widehat{W} \in L^\infty \cap L^2$.
\end{remark}

\begin{remark}
 The method of proof of Theorem~\ref{thm:nonresonant} easily allows to also include a constant coefficient quadratic nonlinearity $\alpha_0 u^2$, $\alpha_0 \in \bbR$, and to establish asymptotics for small solutions to
 \begin{equation*}
  (\pt - i \jn) v = \frac{1}{2i} \jn^{-1} \bigl( \alpha_0 u^2 + \alpha(\cdot) u^2 + \beta_0 u^3 + \beta(\cdot) u^3 \bigr) \text{ on } \bbR^{1+1}
 \end{equation*}
 in the non-resonant case $\widehat{\alpha}(\pm \sqrt{3}) = 0$. It is a classical observation by Shatah~\cite{Sh85} that the constant coefficient quadratic nonlinearity for the Klein-Gordon equation has a nice non-resonance property and can be transformed by the normal form method into non-local constant coefficient cubic nonlinearities. These can then be treated analogously to the (local) constant coefficient cubic term $\beta_0 u^3$ on the right-hand side of~\eqref{equ:first_order_kg}. See for instance~\cite{Del01, HN12} for asymptotics of small solutions to 1D Klein-Gordon equations with a constant coefficient quadratic nonlinearity and~\cite{LS15, Sterb16} in the presence of an additional variable coefficient cubic nonlinearity. 
\end{remark}

\begin{remark} \label{rem:relevance_sineGordon}
 As already discussed in Remark~\ref{rem:intro_mention_LLSS}, after completion of this work the authors in collaboration with Schlag~\cite{LLSS}  generalized Theorem~\ref{thm:resonant} of this paper to the more general one-dimensional Klein-Gordon equation 
 \begin{equation*} 
  (\pt^2 - \px^2 + 1 + V(x)) u = P_c ( \alpha(\cdot) u^2 ) \text{ on } \bbR^{1+1}
 \end{equation*}
 with a spatially localized variable coefficient $\alpha(x)$ and a non-generic potential $V(x)$. In particular, \cite{LLSS} identifies the non-resonance condition for this more general setting to be $\widetilde{\mathcal{F}}[\alpha \varphi^2](\pm \sqrt{3}) = 0$, where $\widetilde{\mathcal{F}}$ denotes the distorted Fourier transform associated with the Schr\"odinger operator $H = - \px^2 + V(x)$ and where $\varphi(x)$ denotes the (normalized) zero energy resonance of $H$. In particular, \cite[Remark 1.2]{LLSS} observes that this non-resonance condition turns out to be satisfied by the sine-Gordon model! For this reason, the Klein-Gordon equation~\eqref{equ:first_order_kg} considered in Theorem~\ref{thm:nonresonant} in this paper is an important model problem for studying perturbations of the sine-Gordon kink since the nonlinearities on the right-hand side of~\eqref{equ:first_order_kg} have exactly the same structure as those that occur in the equation~\eqref{equ:sineGordon_perturbation} for perturbations of the sine-Gordon kink and since both the flat potential $V(x) = 0$ as well as the potential $V(x) = - 2 \sech^2(x)$ in~\eqref{equ:sineGordon_perturbation} are non-generic. 
\end{remark}

\subsection{Related works}

Over the past decades there has been enormous progress in the study of modified scattering for dispersive and hyperbolic equations. The rich and vast literature on this subject cannot be reviewed here in its entirety. We primarily focus on those papers that are most relevant to our results.

The investigation of the long-time behavior of small solutions to Klein-Gordon equations with constant coefficient nonlinearities originates in the pioneering works of Klainerman~\cite{Kl80, Kl85} and Shatah~\cite{Sh85}. The long-range effects of quadratic and cubic nonlinearities for the one-dimensional Klein-Gordon equation were unveiled in the work of Delort~\cite{Del01, Del06}, which established that the asymptotic behavior of small solutions to such equations differs from that of linear Klein-Gordon waves by a logarithmic phase correction. A simpler approach was later developed by Lindblad-Soffer~\cite{LS05_1, LS05_2} in the cubic case, providing a detailed asymptotic expansion of the solution for large times. Subsequently, Hayashi-Naumkin~\cite{HN08, HN12} removed the compact support assumptions about the initial data required in~\cite{Del01, LS05_1, LS05_2}, see also Stingo~\cite{Stingo18} and Candy-Lindblad~\cite{CL18}.

In contrast, the study of one-dimensional nonlinear Klein-Gordon equations with variable coefficient nonlinearities was only recently initiated by Lindblad-Soffer~\cite{LS15} and by Sterbenz~\cite{Sterb16}. Specifically, \cite{LS15, Sterb16} prove dispersive decay of small solutions for smooth, compactly supported initial data in the case of a variable coefficient cubic nonlinearity 
coupled to constant coefficient cubic and quadratic nonlinearities. Recently, a much simpler and robust approach was introduced by the authors in~\cite{LLS19} establishing sharp decay estimates and asymptotics for one-dimensional Klein-Gordon equations with constant and variable coefficient cubic nonlinearities. In contrast to~\cite{LS15, Sterb16}, the approach in~\cite{LLS19} avoids the use of any variable coefficient cubic normal forms and instead uses local decay estimates for the Klein-Gordon propagator as a key tool to deal with difficulties caused by the variable coefficient cubic nonlinearity.

After completion of this work, Germain-Pusateri~\cite{GP20} established modified scattering for the general one-dimensional quadratic Klein-Gordon equation $(\pt^2 - \px^2 + 1 + V(x)) u = a(x) u^2$ with a linear potential $V(x)$, 
where $a(x)$ is a smooth coefficient satisfying $a(x) \to \ell_{\pm \infty}$ as $x \to \pm \infty$ for arbitrary fixed $\ell_{\pm \infty} \in \bbR$ (and is thus not necessarily spatially localized) and where $H = - \px^2 + V(x)$ has no bound states. Under the key spectral assumption that the distorted Fourier transform of the solution $\tilde{u}(t,0) = 0$ vanishes at zero frequency at all times $t \in \bbR$, \cite[Theorem 1.1]{GP20} shows that such small solutions to that Klein-Gordon equation decay in $L^\infty_x$ at the rate $t^{-\hf}$ of free Klein-Gordon waves and that their asymptotics feature logarithmic phase corrections (which are related to the ``non-zero limits'' $\ell_{\pm\infty}$ of the coefficient $a(x)$). 
The condition $\tilde{u}(t,0) = 0$ holds automatically for generic potentials $V(x)$, while for non-generic potentials it is a special case that can often be enforced by imposing suitably parity conditions. It is worth to record that a peculiar feature of the flat Klein-Gordon operator in one space dimension is that it exhibits a zero energy resonance, i.e., that the flat potential $V(x) = 0$ is non-generic.
As an application, \cite[Corollary 1.4]{GP20} yields the full asymptotic stability of kinks with respect to odd perturbations for the double sine-Gordon model in an appropriate range of the deformation parameter.

We emphasize that the occurrence of a logarithmic slow-down of the pointwise decay rate due to the presence of a space-time resonance was exhibited by Bernicot-Germain~\cite{BerGerm13} in a simpler setting of proving bilinear dispersive estimates for quadratic interactions of 1D free dispersive waves. We also refer to~\cite{DIP17, DIPP17} for higher-dimensional instances, where the optimal pointwise decay cannot be propagated by the nonlinear flow (but where it appears that the obtained decay rate is not asserted to be sharp).

For the one-dimensional Schr\"odinger equation with a constant coefficient cubic nonlinearity we refer to Hayashi-Naumkin~\cite{HN98}, Lindblad-Soffer~\cite{LS06}, Kato-Pusateri~\cite{KatPus11}, and Ifrim-Tataru~\cite{IT15} for closely related results on modfied scattering of small solutions. Deift-Zhou~\cite{DZ03} obtained asymptotics even for large initial data in the defocusing case, utilizing the complete integrability of the equation and studying the problem via inverse scattering techniques.

Asymptotics for small solutions to one-dimensional Schr\"odinger equations with a linear potential and with constant as well as variable coefficient cubic nonlinearities were obtained by Delort~\cite{Del16}, Germain-Pusateri-Rousset~\cite{GermPusRou18}, and Chen-Pusateri~\cite{ChenPus19}. In the case of the zero potential, these results only pertain to odd solutions. We remark that the restriction to odd solutions avoids to deal with the threshold resonance of the free Schr\"odinger operator in one space dimension and therefore constitutes a simplification of the problem. 

We conclude by emphasizing that many other nonlinear dispersive and hyperbolic equations exhibit modified scattering of small solutions. Without being exhaustive we mention, for example, the mKdV equation~\cite{HarropGriffiths16, GPR16, HN99}, the boson star equation~\cite{Pus14}, fractional Schr\"odinger equations~\cite{IP14}, water waves equations~\cite{AD15, IP15, DIPP17, IP18, IT16}, the Maxwell-Dirac equation~\cite{FlatoSimonTaflin87}, and the Vlaslov-Poisson system~\cite{MouhotVillani11, BedMasMou16, GrenNguRod20, ChoiKwon16, IPWW20}. See also the corresponding problem for the Einstein field equations of general relativity~\cite{ChristodKlainer93, LR10, Lindblad17, IP19}.

\subsection{Proof ideas} \label{subsec:intro_proof_ideas}

\subsubsection{Resonant Case} \label{subsubsec:resonant_ideas}

We begin with a heuristic discussion of the delicate resonant interaction that occurs in the variable coefficient quadratic nonlinearity between certain spatial frequencies of the variable coefficient and the temporal oscillations of the solution to
\begin{equation} \label{equ:intro_ideas_res_kg}
 (\pt - i \jn) v = \frac{1}{2i} \jn^{-1} \bigl( \alpha(\cdot) u^2 \bigr) \text{ on } \bbR^{1+1}.
\end{equation}
Due to the strong spatial localization of the coefficient $\alpha(x)$, we expect that the leading order contribution of the variable coefficient quadratic nonlinearity $\alpha(x) u(t,x)^2$ is governed by the behavior of $u(t)$ close to the origin $x=0$, in fact by the contribution of $\alpha(x) u(t,0)^2$. At a technical level this intuition can be made rigorous by observing that the difference $\alpha(x) u(t,x)^2 - \alpha(x) u(t,0)^2$ is by the fundamental theorem of calculus schematically of the form $x \alpha(x) (\px u)(t) u(t)$. Then the stronger time decay of this term stems from the improved local decay of spatial derivatives of the solution, which is in effect due to the spatial localization of the coefficient $\alpha(x)$.
Recalling that $u(t,0) = v(t,0) + \bar{v}(t,0)$, we can thus think of~\eqref{equ:intro_ideas_res_kg} as
\begin{equation} \label{equ:intro_ideas_res1}
 \begin{aligned}
  (\pt - i \jn) v &= \frac{1}{2i} \bigl( \jn^{-1} \alpha \bigr) \Bigl( v(t,0)^2 + 2 v(t,0) \bar{v}(t,0) + \bar{v}(t,0)^2 \Bigr) + \bigl\{ \text{better terms} \bigr\}.
 \end{aligned}
\end{equation}
Now suppose for the moment that $v(t,0)$ asymptotically behaves like a linear Klein-Gordon wave, i.e. 
\begin{equation} \label{equ:intro_ideas_res_v_origin}
 v(t,0) \sim \frac{e^{it}}{t^\hf} \quad \text{ for } t \gg 1.
\end{equation}
Inserting this into Duhamel's formula for the main nonlinear terms on the right-hand side of~\eqref{equ:intro_ideas_res1} gives that to leading order the long-time behavior of $v(t,x)$ is determined by an expression of the form 
\begin{equation} \label{equ:intro_ideas_res2}
 \frac{1}{2i} \int_1^t \bigl( e^{i(t-s)\jn} \jn^{-1} \alpha \bigr)(x) \biggl( \frac{e^{2is}}{s} + \frac{2}{s} + \frac{e^{-2is}}{s} \biggr) \, \ud s.
\end{equation}
Clearly, the asymptotics of~\eqref{equ:intro_ideas_res2} hinge on the temporal oscillations in $s$ of the integrand. Filtering by the linear evolution $e^{it\jn}$ and taking the Fourier transform gives
\begin{equation} \label{equ:intro_ideas_res3}
 \frac{1}{2i} \int_1^t \jxi^{-1} \widehat{\alpha}(\xi) \Bigl( e^{is(2-\jxi)} + 2 e^{is\jxi} + e^{-is(2+\jxi)} \Bigr) \frac{1}{s} \, \ud s.
\end{equation}
Then it is evident that the first term in the parentheses in~\eqref{equ:intro_ideas_res3} causes a resonance because it is stationary in~$s$ when
\begin{equation*}
 2-\jxi = 0 \quad \Leftrightarrow \quad \xi = \pm \sqrt{3},
\end{equation*}
where we use the Japanese bracket notation $\jxi = (1+\xi^2)^\hf$. The other two terms in the parentheses in~\eqref{equ:intro_ideas_res3} are non-stationary in~$s$ and therefore better behaved.
In view of the dispersion relation for the Klein-Gordon propagator $e^{it\jn}$, the frequencies $\xi = \pm \sqrt{3}$ are associated with the rays $\frac{x}{t} = \mp \frac{\sqrt{3}}{2}$. 
This computation suggests that the corresponding part of the integral~\eqref{equ:intro_ideas_res2} along the rays $\frac{x}{t} = \mp \frac{\sqrt{3}}{2}$ is (partially) monotone, which should in particular cause a logarithmic slow-down of the decay rate along those rays. 

Becoming more rigorously now, the first step in the proof of Theorem~\ref{thm:resonant} consists in establishing local decay estimates for the solution $v(t)$ to~\eqref{equ:intro_ideas_res1} as well as for its spatial derivatives $\px v$ and for the time derivative of its ``phase-filtered'' component $\pt ( e^{-it} v(t) )$, see Proposition~\ref{prop:resonant_bootstrap_bounds}. Their derivation crucially exploits the spatial localization provided by the variable coefficient $\alpha(x)$. 
Next, we conclude in Proposition~\ref{prop:resonant_asymptotics_origin} that the asymptotic behavior of the nonlinear solution $v(t)$ to~\eqref{equ:intro_ideas_res1} at the origin $x=0$ is indeed that of linear Klein-Gordon waves! The idea is to just start off from Duhamel's formula for the nonlinear term and to simply insert the asymptotics for the retarded Klein-Gordon propagator. Having the powerful local decay bounds for $v(t)$ at our disposal along with the spatial localization of the coefficient $\alpha(x)$, we are able to control all remainder terms and to isolate the leading order long-time behavior at the origin. Specifically, we obtain that there exists $a_0 \in \bbC$, explicitly defined in~\eqref{equ:resonant_amplitude_a0} in terms of the solution $v(t)$, such that
\begin{equation*}
 v(t,0) = \frac{1}{t^\hf} e^{i\frac{\pi}{4}} e^{it} a_0 + \calO \Bigl( \frac{\varepsilon^2}{t^{1-}} \Bigr), \quad t \gg 1.
\end{equation*}
Here, $0 < \varepsilon \ll 1$ measures the small size of a weighted Sobolev norm of the initial datum.
Finally, in the proof of Theorem~\ref{thm:resonant} we first use the local decay bounds for $v(t)$ to infer the decay estimate
\begin{equation} \label{equ:intro_ideas_res4}
 \|v(t)\|_{L^\infty_x} \lesssim \varepsilon \frac{1 + \log(\jt)}{\jt^\hf}.
\end{equation}
Then the majority of the work goes into uncovering more details of the asymptotics of $v(t)$, which in particular shows that the decay estimate~\eqref{equ:intro_ideas_res4} is sharp. Proceeding along the lines of the reasoning in the heuristic discussion above, we use the local decay bounds for $v(t)$ and the knowledge of the asymptotics of $v(t)$ at $x=0$ to peel off all parts of $v(t)$ that asymptotically behave like linear Klein-Gordon waves. This leaves us with the component
\begin{equation} \label{equ:intro_res_3}
 v_{mod}(t,x) := \frac{a_0^2}{2} \int_1^t \bigl( e^{i(t-s) \jn} \jn^{-1} \alpha \bigr)(x) \frac{e^{2is}}{s} \, \ud s.
\end{equation}
Then a stationary (and non-stationary) phase analysis of~\eqref{equ:intro_res_3} reveals the modified scattering behavior~\eqref{equ:resonant_thm_decay_off_rays}--\eqref{equ:resonant_thm_asymptotics_along_special_rays} with the punch line being the logarithmic slow-down along the rays $\frac{x}{t} = \pm \frac{\sqrt{3}}{2} t$ with asymptotics given by
\begin{equation*} 
 v_{mod}\Bigl(t, \pm \frac{\sqrt{3}}{2} t \Bigr) = \frac{a_0^2}{\sqrt{8}}  e^{i\frac{\pi}{4}} e^{i \frac{t}{2}} \widehat{\alpha}(\mp \sqrt{3}) \frac{\log(t)}{t^{\frac{1}{2}}} + \calO \Bigl( \frac{\varepsilon^2}{t^\hf} \Bigr), \quad t \gg 1.
\end{equation*}

\subsubsection{Non-Resonant Case} \label{subsubsec:nonresonant_ideas}

We now outline the main ideas of the proof of Theorem~\ref{thm:nonresonant}, which establishes sharp decay estimates and asymptotics for small solutions to
\begin{equation} \label{equ:intro_ideas_nonres1}
 (\pt - i \jn) v = \frac{1}{2i} \jn^{-1} \bigl( \alpha(\cdot) u^2 + \beta_0 u^3 + \beta(\cdot) u^3 \bigr) \text{ on } \bbR^{1+1}
\end{equation}
under the non-resonance assumption
\begin{equation} \label{equ:intro_ideas_nonres2}
 \widehat{\alpha}(+\sqrt{3}) = 0 \quad \text{ and } \quad \widehat{\alpha}(-\sqrt{3}) = 0.
\end{equation}
Recall that we use the short-hand notation $u(t) = v(t) + \bar{v}(t)$. 
In view of the slow decay rate $t^{-\hf}$ of linear Klein-Gordon waves in one space dimension, the constant coefficient cubic term $\beta_0 u^3$ has critical dispersive decay. We therefore expect it to cause a modified scattering behavior of the solution $v(t)$ to~\eqref{equ:intro_ideas_nonres1}. Oversimplifying a little bit here, all current techniques to capture asymptotic corrections in the scattering behavior of small solutions to dispersive equations usually combine some version of an ODE argument with slow growth estimates for energies of weighted vector fields of the solution. In the context of Klein-Gordon equations, the Lorentz boost $Z = t \px + x \pt$ and the closely related operator $L = \jn x - i t \px$ play a crucial role.
However, in the presence of a variable coefficient nonlinearity, it becomes problematic to obtain such slow growth bounds, because the vector field $Z$ and the operator $L$ produce badly divergent factors of $t$ when they fall onto a variable coefficient. 
This issue becomes particularly severe in the case of variable coefficient {\it quadratic} nonlinearities as in~\eqref{equ:intro_ideas_nonres1} that can only provide little time decay to compensate. 
In our previous work~\cite{LLS19} we introduced a simple and robust method based on local decay estimates for the Klein-Gordon propagator to overcome this problem in the context of variable coefficient {\it cubic} nonlinearities.

The first and key step in the proof of Theorem~\ref{thm:nonresonant} is therefore to transform the variable coefficient quadratic nonlinearity on the right-hand side of~\eqref{equ:intro_ideas_nonres1} into a more favorable form that is of ``variable coefficient cubic type''. As already observed in the discussion of the proof of Theorem~\ref{thm:resonant} above, we expect the difference $\alpha(x) u(t,x)^2 - \alpha(x) u(t,0)^2 \sim x \alpha(x) (\px u)(t) u(t)$ to have stronger time decay due to the improved local decay of spatial derivatives of the solution, see Lemma~\ref{lem:improved_decay_localized_derivatives} in the context of the study of~\eqref{equ:intro_ideas_nonres1}. Hence, it suffices to recast the part $\alpha(x) u(t,0)^2$ of the variable coefficient quadratic nonlinearity in~\eqref{equ:intro_ideas_nonres1} into a better form. To this end we insert the decomposition of $u(t,0)$ into its ``phase-filtered components''
\begin{equation*}
 u(t,0) = e^{+it} \bigl( e^{-it} v(t,0) \bigr) + e^{-it} \bigl( e^{+it} \bar{v}(t,0) \bigr)
\end{equation*}
into Duhamel's formula for the nonlinear term $\alpha(x) u(t,0)^2$ to find that
\begin{equation} \label{equ:intro_ideas_nonres4}
 \begin{aligned}
  \frac{1}{2i} \int_0^t \bigl( e^{i(t-s)\jn} \jn^{-1} \alpha \bigr) u(s,0)^2 \, \ud s = \frac{1}{2i} \int_0^t \bigl( e^{i(t-s)\jn} \jn^{-1} \alpha \bigr) \Bigl( e^{+2is} \bigl( e^{-is} v(s,0) \bigr)^2 + \ldots \Bigr) \, \ud s.  
 \end{aligned}
\end{equation}
Here we only display the most delicate term in the parentheses of the integrand on the right-hand side of~\eqref{equ:intro_ideas_nonres4}.
Owing to the non-resonance assumption~\eqref{equ:intro_ideas_nonres2} the integral~\eqref{equ:intro_ideas_nonres4} has a nice non-resonance property and we can integrate by parts in time $s$ to get 
\begin{equation} \label{equ:intro_ideas_nonres5}
 \begin{aligned}
  \int_0^t \bigl( e^{i(t-s)\jn} (2-\jn)^{-1} \jn^{-1} \alpha \bigr) e^{+2is} \partial_s \bigl( e^{-is} v(s,0) \bigr) \bigl( e^{-is} v(s,0) \bigr) \, \ud s + \ldots + \bigl\{ \text{boundary terms} \bigr\}.
 \end{aligned}
\end{equation}
More precisely, we use that $\widehat{\alpha}(\xi)$ is assumed to vanish at the frequencies $\xi = \pm \sqrt{3}$ where the symbol $(2-\jxi)^{-1}$ has a singularity.
Then we make the important observation that the time derivative of the phase filtered component $\pt ( e^{-it} v(t,0) )$ has stronger time decay at the origin. Specifically, we obtain in Lemma~\ref{lem:key_improved_decay} that
\begin{equation*}
 \bigl| \pt \bigl( e^{-it} v(t,0) \bigr) \bigr| \lesssim \frac{\varepsilon}{\jt^{1-\delta}},
\end{equation*}
where $0 < \delta \ll 1$ is a small absolute constant and $0 < \varepsilon \ll 1$ measures the small size of a weighted Sobolev norm of the initial datum. Assuming that the nonlinear solution $v(t)$ to~\eqref{equ:intro_ideas_nonres1} ends up having the decay rate $t^{-\hf}$ of linear Klein-Gordon waves in one space dimension, we can then view the first term on the right-hand side of~\eqref{equ:intro_ideas_nonres5} as Duhamel's formula for a nonlinear term of ``variable coefficient cubic type'' that is of the schematic form $\kappa(x) \varepsilon^2 \jt^{-\thf + \delta}$ for some smooth and decaying coefficient $\kappa(x)$. 

Put differently, in the first part of the proof of Theorem~\ref{thm:nonresonant} we introdue the {\it variable coefficient quadratic normal form}
\begin{equation*}
 \calQ := \alpha_1(x) v(t,0)^2 + \alpha_2(x) |v(t,0)|^2 + \alpha_3(x) \bar{v}(t,0)^2
\end{equation*}
with smooth and decaying coefficients $\alpha_j(x)$, $j = 1, 2, 3$, defined in terms of their Fourier transforms by 
\begin{equation*} 
\begin{aligned}
 \widehat{\alpha}_1(\xi) &:= \frac{1}{2} (2-\jxi)^{-1} \jxi^{-1} \widehat{\alpha}(\xi), \\
 \widehat{\alpha}_2(\xi) &:= - \jxi^{-2} \widehat{\alpha}(\xi), \\
 \widehat{\alpha}_3(\xi) &:= - \frac{1}{2} (2+\jxi)^{-1} \jxi^{-1} \widehat{\alpha}(\xi).
\end{aligned}
\end{equation*}
Then the new variable $v + \calQ$ satisfies a Klein-Gordon equation of the form
\begin{equation} \label{equ:intro_ideas_nonres3}
 (\pt - i \jn) (v + \calQ) = \frac{\beta_0}{2i} \jn^{-1} \bigl( u^3 \bigr) + \calO \Bigl( \kappa(x) \varepsilon^2 \jt^{-\thf+\delta} \Bigr).
\end{equation}
The right-hand side of~\eqref{equ:intro_ideas_nonres3} can be viewed as the sum of the constant coefficient cubic nonlinearity~$\beta_0 u^3$ and of several nonlinear terms of ``variable coefficient cubic type''. 
At this point we are in a position to implement a version of the strategy from our previous work~\cite{LLS19} to derive decay estimates and asymptotics for small solutions to~\eqref{equ:intro_ideas_nonres3}. 
More specifically, a standard result on the asymptotics of the Klein-Gordon propagator (Lemma~\ref{lem:KG_propagator_asymptotics}) reduces the proof of decay and asymptotics of the solution~$v(t)$ to establishing the uniform boundedness and asymptotics of the Fourier transform of the profile $f(t) := e^{-it\jn} v(t)$ of the solution~$v(t)$. Following the approach of the space-time resonances method by Germain-Masmoudi-Shatah~\cite{GMS09, GMS12_JMPA, GMS12_Ann} and Gustafson-Nakanishi-Tsai~\cite{GNT09}, we obtain the latter from the analysis of the stationary points of the oscillatory integrals that govern the equation satisfied by the Fourier transform of the profile $\hat{f}(t,\xi)$, see Proposition~\ref{prop:Linfty_bound_profile}. Some technical aspects of this step are inspired by the proof of modified scattering for the mKdV equation contained in~\cite{GPR16}. In order to control various remainder terms arising in this analysis, we need a slow growth estimate for the energy $\|\jn L v(t)\|_{L^2_x}$ of the operator $L = \jn x - i t \px = e^{it\jn} \jn x e^{-it\jn}$, which yields control of the spatial localization of the profile. In the derivation of this slow growth estimate, we use a key idea from our previous work~\cite{LLS19} and employ local decay estimates for the Klein-Gordon propagator. These provide a crucial source of additional time decay in certain energy estimates to compensate the divergent factor of $t$ that occurs when the operator $L$ falls onto the variable coefficients, see Lemma~\ref{lem:local_decay} and Proposition~\ref{prop:slow_growth_H1L}.

\section{Preliminaries}

\subsection{Notation and conventions}

For nonnegative $X$, $Y$ we write $X \lesssim Y$ or $X = \calO(Y)$ if $X \leq C Y$ for some absolute constant $C > 0$. 
We use the notation $X \lesssim_\nu Y$ to indicate that the implicit constant depends on a parameter $\nu$ and we write $X \ll Y$ if the implicit constant should be regarded as small. 
Moreover, we use the notation $\jap{x} = (1+x^2)^\hf$, $\jap{t} = (1+t^2)^\hf$, and $\jxi = (1+\xi^2)^\hf$.
For any number $a \in \bbR$ we denote by $a+$, respectively by $a-$, a number that is larger, respectively smaller, than $a$, but that can be chosen arbitrarily close to $a$.

Our conventions for the Fourier transform and its inverse are as follows:
\begin{equation*}
 \calF[f](\xi) = \hat{f}(\xi) := \frac{1}{\sqrt{2\pi}} \int_{\bbR} e^{-ix\xi} f(x) \, \ud x, \qquad \calF^{-1}[f](x) = \check{f}(x) := \frac{1}{\sqrt{2\pi}} \int_{\bbR} e^{+ix\xi} f(\xi) \, \ud \xi.
\end{equation*}
Then we define the operator $\jn$ by $\calF[\jn f](\xi) = \jxi \hat{f}(\xi)$ and we define the action of the Klein-Gordon propagator $e^{\pm it \jn}$ by $\calF[e^{\pm it\jn} f](\xi) = e^{\pm it\jxi} \hat{f}(\xi)$.

We denote the Lorentz boost by $Z = t \px + x \pt$ and we introduce the operator $L = \jn x - i t \px$. Note that $\jn x$ conjugates to $L$ via $e^{it\jn}$ in the sense that 
\begin{equation*}
 L = \jap{\nabla} x - i t \px = e^{i t \jap{\nabla}} \jn x e^{- i t \jap{\nabla}} = \calF^{-1} e^{i t \jap{\xi}} \jap{\xi} i \partial_\xi e^{- i t \jap{\xi}} \calF. 
\end{equation*}
The Lorentz boost $Z$ and the operator $L$ are closely related by the identity
\begin{equation*}
 Z = i L + (\partial_t - i \jn) x = i L + i \jn^{-1} \px + x (\partial_t - i \jn).
\end{equation*}
In the derivation of energy estimates in Subsection~\ref{subsec:nonresonant_energy_estimates} we repeatedly use the following commutator identities
\begin{equation} \label{equ:commutator_identities}
\begin{aligned}
 \bigl[ (\partial_t - i \jap{\nabla}), L \bigr] &= 0, \\
 [ (\partial_t - i \jn), Z ] &= i \jn^{-1} \px (\pt - i \jn), \\
 [ x, \jap{\nabla}^k ] &= k \jap{\nabla}^{k-2} \px, \quad k \in \bbZ, \\
 [ L, \jn^{-1} ] &= - \jn^{-2} \px, \\
 [Z, \jap{\nabla}^{-1}] &= - \jn^{-3} \px \pt.
\end{aligned}
\end{equation}

Finally, we record a standard energy estimate for the first-order Klein-Gordon equation.
 \begin{lemma}(Energy estimate)
 Let $w(t)$ be a solution to 
 \begin{align*}
  (\partial_t - i \jn) w = F \quad \text{ on } [0,T] \times \bbR.
 \end{align*}
 Then we have for $0 \leq t \leq T$ that
 \begin{equation} \label{equ:energy_est_sq}
  \|w(t)\|_{L^2_x(\bbR)}^2 \leq \|w(0)\|_{L^2_x}^2 + 2 \biggl| \int_0^t \int_{\bbR} F(s) \, \overline{w(s)} \, \ud x \, \ud s \biggr|
 \end{equation}
 and that
 \begin{equation} \label{equ:energy_est_std}
  \|w\|_{L^\infty_t L^2_x([0,T]\times\bbR)} \lesssim \|w(0)\|_{L^2_x(\bbR)} + \|F\|_{L^1_t L^2_x([0,T]\times\bbR)}.
 \end{equation}
\end{lemma}

\subsection{Decay estimates}

We begin by recalling the asymptotics of the Klein-Gordon propagator $e^{it\jn}$ in one space dimension, see for instance\cite[Section III]{HN12} for a proof.
\begin{lemma} \label{lem:KG_propagator_asymptotics}
 The asymptotics of linear Klein-Gordon waves in one space dimension are given by
 \begin{equation} \label{equ:KG_propagator_asymptotics}
  \begin{aligned}
  \bigl( e^{it\jn} f \bigr)(x) &= \frac{1}{t^\hf} e^{i\frac{\pi}{4}} e^{i\rho} \jap{ {\textstyle \frac{x}{\rho}} }^{\frac{3}{2}} \hat{f}\bigl(- {\textstyle \frac{x}{\rho}} \bigr) \theta\bigl( {\textstyle \frac{x}{t}} \bigr) + \frac{1}{t^{\frac{5}{8}}} \calO \bigl( \| \jap{x} f \|_{H^2_x} \bigr), \quad t \geq 1, \quad x \in \bbR,
  \end{aligned}
 \end{equation}
 where $\rho := (t^2 - x^2)^{\hf}$ and $\theta(\cdot)$ is a sharp cut-off function with $\theta(z) = 1$ for $|z| < 1$ and $\theta(z) = 0$ for $|z| \geq 1$.
 In particular, we have 
 \begin{equation} \label{equ:KG_propagator_decay_est}
  \bigl\| e^{it\jn} f \bigr\|_{L^\infty_x} \lesssim \frac{1}{t^{\hf}} \bigl\| \jap{\xi}^{\frac{3}{2}} \hat{f}(\xi) \bigr\|_{L^\infty_{\xi}} + \frac{1}{t^{\frac{5}{8}}} \bigl\| \jap{x} f \bigr\|_{H^2_x}, \quad t \geq 1.
 \end{equation}
\end{lemma}
We point out that the identity $\jap{ \frac{x}{\rho} } = \frac{t}{\rho}$ gives rise to a slightly different way of expressing the asymptotics~\eqref{equ:KG_propagator_asymptotics} of linear Klein-Gordon waves as
\begin{equation*}
 \bigl( e^{it\jn} f \bigr)(x) = \frac{1}{t^\hf} e^{i\frac{\pi}{4}} e^{i\rho} \bigl( {\textstyle \frac{t}{\rho}} \bigr)^{\frac{3}{2}} \hat{f}\bigl(- {\textstyle \frac{x}{\rho}} \bigr) \theta\bigl( {\textstyle \frac{x}{t}} \bigr) + \frac{1}{t^{\frac{5}{8}}} \calO \bigl( \| \jap{x} f \|_{H^2_x} \bigr), \quad t \geq 1, \quad x \in \bbR.
\end{equation*}
Moreover, we record that the bound~\eqref{equ:KG_propagator_decay_est} implies a decay estimate for any time-dependent function $v(t)$ in terms of its Klein-Gordon profile $f(t) := e^{-it\jn} v(t)$ given by
\begin{equation}
 \|v(t)\|_{L^\infty_x} \lesssim \frac{1}{t^{\hf}} \bigl\| \jap{\xi}^{\frac{3}{2}} \hat{f}(t,\xi) \bigr\|_{L^\infty_\xi} + \frac{1}{t^{\frac{5}{8}}} \Bigl( \bigl\| (\jn L v)(t) \bigr\|_{L^2_x} + \bigl\| (\jn^2 v)(t) \bigr\|_{L^2_x} \Bigr), \quad t \geq 1.
\end{equation}

\begin{remark} \label{rem:stronger_remainder_decay}
 For technical reasons, in one part of the proof of Theorem~\ref{thm:resonant} we need a decay estimate for the remainder term in the asymptotics of linear Klein-Gordon waves that is stronger than the one in~\eqref{equ:KG_propagator_asymptotics}. An inspection of the proof of \cite[Theorem 7.2.1]{H97} gives that under stronger assumptions on the initial datum we have 
 \begin{equation}
  \bigl( e^{it\jn} f \bigr)(x) = \frac{1}{t^\hf} e^{i\frac{\pi}{4}} e^{i\rho} \jap{ {\textstyle \frac{x}{\rho}} }^{\frac{3}{2}} \hat{f}\bigl(- {\textstyle \frac{x}{\rho}} \bigr) \theta\bigl( {\textstyle \frac{x}{t}} \bigr) + \frac{1}{t} \calO \bigl( \| \jap{x}^2 f \|_{H^4_x} \bigr), \quad t \geq 1.
 \end{equation}
\end{remark}

On occasion we also use the following standard dispersive decay estimate for the Klein-Gordon propagator in one space dimension, see for instance H\"ormander~\cite[Corollary 7.2.4]{H97} for a proof.
\begin{lemma} \label{lem:dispersive_decay}
 We have uniformly for all $t \in \bbR$ that 
 \begin{equation} \label{equ:dispersive_decay}
  \bigl\| e^{\pm it \jap{\nabla}} f \bigr\|_{L^\infty_x} \lesssim \frac{1}{\jap{t}^{\frac{1}{2}}} \|\jap{\nabla}^2 f\|_{L^1_x}.
 \end{equation}
\end{lemma}

\medskip 

The following pointwise-in-time local decay estimates for the Klein-Gordon propagator in one space dimension play a key role in our arguments.
Such local decay estimates for much larger classes of unitary operators originate in the works of Rauch~\cite{Rauch78}, Jensen-Kato~\cite{KJ79}, and Jensen~\cite{Jensen80, Jensen84}, see also~\cite{HSS99, JSS91, Ger08, GLS16, LarS15, Schl07} and references therein.
\begin{lemma} \label{lem:local_decay}
 Let $a \geq 1, b \geq 0$. We have uniformly for all $t \in \bbR$ that 
 \begin{align}
  \bigl\| \jap{x}^{-a} \jap{\nabla}^{-b} e^{\pm i t \jap{\nabla}} \jap{x}^{-a} \bigr\|_{L^2_x(\bbR) \to L^2_x(\bbR)} &\lesssim \frac{1}{\jap{t}^{\frac{1}{2}}}, \label{equ:local_decay} \\
  \Bigl\| \jap{x}^{-2} \frac{\partial_x}{\jap{\nabla}} e^{\pm i t \jap{\nabla}} \jap{x}^{-2} \Bigr\|_{L^2_x(\bbR) \to L^2_x(\bbR)} &\lesssim \frac{1}{\jap{t}^{\frac{3}{2}}}, \label{equ:local_decay_van1} \\
  \Bigl\| \jap{x}^{-2} \frac{\jap{\nabla}-1}{\jap{\nabla}} e^{\pm i t \jap{\nabla}} \jap{x}^{-2} \Bigr\|_{L^2_x(\bbR) \to L^2_x(\bbR)} &\lesssim \frac{1}{\jap{t}^2}. \label{equ:local_decay_van2} 
 \end{align}
\end{lemma}
\begin{proof}
 See Lemma 2.1 and Lemma 2.2 in~\cite{LLS19} for the proofs of~\eqref{equ:local_decay} and \eqref{equ:local_decay_van1}. The proof of~\eqref{equ:local_decay_van2} follows along the lines of the proof of Lemma 2.2 in~\cite{LLS19} by exploiting that the symbol of $\jap{\nabla}-1$ vanishes to second order at the origin, namely $\jap{\xi} - 1 = \calO(\xi^2)$ for $0 \leq \xi \ll 1$. Thus, one can just integrate by parts twice in the frequency variable to get the $\jap{t}^{-2}$ decay.
\end{proof}

\subsection{Harmonic analysis tools}

The derivation of the differential equation for the profile of the solution in Lemma~\ref{lem:ode_weighted_profile} in the non-resonant case relies on the following stationary phase lemma in two dimensions.
\begin{lemma} \label{lem:stationary_phase}
 Let $\chi \in \mathcal{C}_c^\infty$ be a smooth bump function such that $\chi = 0$ in $B(0,2)^c$, and $|\nabla \chi| + |\nabla^2 \chi| \lesssim 1$. Let $\psi \in \mathcal{C}^\infty$ be such that on the support of $\chi$ it holds that $|\mathrm{det} \, \Hess \, \psi| \geq \mu$ for some $0 < \mu \leq 1$ and that $|\nabla \psi| + |\nabla^2 \psi| + |\nabla^3 \psi| \lesssim 1$. Consider the oscillatory integral
 \begin{equation*}
  I = \iint e^{i\lambda \psi(\eta,\sigma)} F(\eta,\sigma) \chi(\eta,\sigma)\,d\eta \,d\sigma.
 \end{equation*}
 For any $\alpha \in [0,1]$ we have: 
 \setlength{\leftmargini}{2em}
 \begin{itemize}
  \item[(i)] If $\nabla \psi$ only vanishes at $(\eta_0,\sigma_0)$,
   \begin{equation*}
    I = \frac{2\pi e^{i\frac{\pi}{4}s}}{\sqrt{\Delta}}\frac{e^{i\lambda \psi(\eta_0,\sigma_0)}}{\lambda}  F(\eta_0,\sigma_0) \chi(\eta_0,\sigma_0)
    + \calO \left( \frac{ \bigl\| \langle (x,y) \rangle^{2 \alpha} \widehat{F} \bigr\|_{L^1}}{\mu^{\hf + 2\alpha} \lambda^{1+\alpha}} \right),
   \end{equation*}
   where $s = \sign \Hess \, \psi(\eta_0, \sigma_0)$ and $\Delta = |\mathrm{det} \, \Hess \, \psi(\eta_0,\sigma_0)|$.
 \item[(ii)] If $|\nabla \psi| \gtrsim 1$,
  $$
   I =  \calO \left( \frac{\bigl\| \langle (x,y) \rangle^{\alpha} \widehat{F} \bigr\|_{L^1}}{\mu^{\hf + \alpha} \lambda^{1+\alpha}} \right).
  $$
 \end{itemize}
\end{lemma}
\begin{proof}
 The proof is a minor adaptation of the proof of \cite[Lemma A.1]{GPR16}, taking into account the possibly small lower bound $|\mathrm{det} \, \Hess \, \psi| \geq \mu > 0$.
\end{proof}

We also use the following result on bounds for pseudo-product operators from \cite[Lemma A.2]{GPR16} in the course of the proof of Lemma~\ref{lem:ode_weighted_profile}.

\begin{lemma} \label{lem:pseudoproduct_op_bound}
 Assume that $m \in L^1(\bbR \times \bbR)$ satisfies 
 \begin{equation}
  \biggl\| \int_{\bbR\times\bbR} m(\eta,\sigma) e^{ix\eta} e^{iy\sigma} \, \ud \eta \, \ud \sigma \biggr\|_{L^1_{x,y}(\bbR\times\bbR)} \leq A
 \end{equation}
 for some $A>0$. Then for all $p, q, r \in [1,\infty]$ with $\frac{1}{p} + \frac{1}{q} = \frac{1}{r}$, the pseudo-product operator $T_m$ defined by
 \begin{equation*}
  \calF\bigl[ T_m(f,g) \bigr](\xi) := \int_{\bbR} m(\xi, \eta) \hat{f}(\xi-\eta) \hat{g}(\eta) \, \ud \eta,
 \end{equation*} 
 satisfies
 \begin{equation}
  \bigl\| T_m(f,g) \bigr\|_{L^r(\bbR)} \lesssim A \|f\|_{L^p(\bbR)} \|g\|_{L^q(\bbR)}.
 \end{equation}
 Moreover, if $\frac{1}{p} + \frac{1}{q} + \frac{1}{r} = 1$, we have that
 \begin{equation}
  \biggl| \int_{\bbR\times\bbR} m(\eta, \sigma) \hat{f}(\eta) \hat{g}(\sigma) \hat{h}(-\eta-\sigma) \, \ud \eta \, \ud \sigma \biggr| \lesssim A \|f\|_{L^p(\bbR)} \|g\|_{L^q(\bbR)} \|h\|_{L^r(\bbR)}.
 \end{equation}
\end{lemma}

\section{Resonant Case}

In this section we investigate the asymptotic behavior of small solutions to
\begin{equation} \label{equ:nlkg_resonant_sec}
 (\pt - i \jn) v = \frac{1}{2i} \jn^{-1} \bigl( \alpha(\cdot) u^2 \bigr) \text{ on } \bbR^{1+1}
\end{equation}
in the {\bf resonant case}
\begin{equation*} 
 \widehat{\alpha}(+\sqrt{3}) \neq 0 \quad \text{ or } \quad \widehat{\alpha}(-\sqrt{3}) \neq 0.
\end{equation*}
Global existence of small regular solutions to~\eqref{equ:nlkg_resonant_sec} is well-known, see for instance~\cite{H97}. In the proof of Theorem~\ref{thm:resonant} we therefore focus on deriving global-in-time a priori bounds for small solutions to~\eqref{equ:nlkg_resonant_sec} from which we can infer sharp decay estimates and asymptotics.

\subsection{Local decay estimates}

We first derive a collection of local decay bounds for the solution.

\begin{proposition}[Local decay estimates] \label{prop:resonant_bootstrap_bounds}
 Let $\alpha(x)$ be a smooth function satisfying $\| \jx^7 \alpha(x) \|_{H^3_x} < \infty$. Then there exists a small absolute constant $\varepsilon_0 > 0$ so that for any initial datum $v_0$ with 
 \begin{equation*}
  \varepsilon := \| \jx^2 v_0 \|_{H^4_x} \leq \varepsilon_0,
 \end{equation*}
 the solution $v(t)$ to~\eqref{equ:resonant_nlkg} satisfies
 \begin{equation} \label{equ:resonant_bootstrap_bounds}
  \begin{aligned}
   &\sup_{t\in\bbR} \, \biggl\{  \jap{t}^{\frac{1}{2}} \| \jap{x}^{-2} v(t) \|_{L^2_x} + \sum_{j=1}^3 \, \jap{t} \| \jap{x}^{-2} \partial_x^j v(t) \|_{L^2_x} + \sum_{j=0}^{1} \, \jap{t} \bigl\| \jap{x}^{-2} \partial_x^j \partial_t \bigl( e^{-i t} v(t) \bigr) \bigr\|_{L^2_x} \biggr\} \lesssim \varepsilon.
  \end{aligned}
 \end{equation}
\end{proposition}
\begin{proof}
 The proof proceeds via a continuity argument. By time-reversal symmetry it suffices to argue forward in time. For any $T > 0$ we define the bootstrap quantity 
\begin{equation} \label{equ:resonant_bootstrap_quantity}
 \begin{aligned}
  M(T) &:= \sup_{0 \leq t \leq T} \, \biggl\{ \jap{t}^{\frac{1}{2}} \| \jap{x}^{-2} v(t) \|_{L^2_x} + \sum_{j=1}^3 \, \jap{t} \| \jap{x}^{-2} \partial_x^j v(t) \|_{L^2_x} + \sum_{j=0}^{1} \, \jap{t} \bigl\| \jap{x}^{-2} \partial_x^j \partial_t \bigl( e^{-i t} v(t) \bigr) \bigr\|_{L^2_x} \biggr\}.
 \end{aligned}
\end{equation}
We first observe that one-dimensional Sobolev estimates imply the following weighted $L^\infty_x$-bounds that will be used throughtout this proof
\begin{align}
 \sup_{0 \leq t \leq T} \, \jap{t}^{\frac{1}{2}} \| \jap{x}^{-2} \partial_x^j v(t) \|_{L^\infty_x} &\lesssim M(T), \quad 0 \leq j \leq 2,\\
 \sup_{0 \leq t \leq T} \, \jap{t} \bigl\| \jap{x}^{-2} \partial_t \bigl( e^{-it} v(t) \bigr) \bigr\|_{L^\infty_x} &\lesssim M(T).
\end{align}
In what follows we only consider times $0\leq t \leq T$.

\medskip 

\noindent \underline{{\it Local decay for $\pt (e^{-it} v(t))$}:}
We begin with the local decay estimates for the time derivative of the phase-filtered component $\partial_t \bigl( e^{-it} v(t) \bigr)$. To this end we compute that 
\begin{equation} 
 \begin{aligned}
  \partial_t \bigl( e^{-it} v(t) \bigr) &= e^{-it} \biggl( i (\jap{\nabla}-1) e^{it\jap{\nabla}} v_0 + \frac{1}{2 i} \jn^{-1} \bigl( \alpha(\cdot) u(t)^2 \bigr) \\
  &\qquad \qquad \qquad + \frac{1}{2} \int_0^t \frac{\jap{\nabla}-1}{\jap{\nabla}} e^{i(t-s)\jap{\nabla}} \bigl( \alpha u(s)^2 \bigr) \, \ud s \biggr).
 \end{aligned}
\end{equation}
By the local decay estimate~\eqref{equ:local_decay_van2} for the Klein-Gordon propagator, it follows that 
\begin{align*}
 \bigl\| \jap{x}^{-2} \partial_t \bigl( e^{-it} v(t) \bigr) \bigr\|_{L^2_x} &\lesssim \frac{\|\jx^2 v_0 \|_{H^1_x}}{\jt^2} + \|\jap{x}^4 \alpha(x)\|_{L^2_x} \|\jap{x}^{-2} v(t)\|_{L^\infty_x}^2 \\
 &\quad \quad + \int_0^t \biggl\| \jap{x}^{-2} \frac{\jap{\nabla}-1}{\jap{\nabla}} e^{i(t-s)\jap{\nabla}} \jap{x}^{-2} \biggr\|_{L^2_x \to L^2_x} \bigl\| \jap{x}^6 \alpha(x) \bigr\|_{L^2_x} \|\jap{x}^{-2} v(s)\|_{L^\infty_x}^2 \, \ud s \\
 &\lesssim \frac{\|\jx^2 v_0 \|_{H^1_x}}{\jt^2} + \frac{M(T)^2}{\jap{t}} + \int_0^t \frac{1}{\jap{t-s}^2} \frac{M(T)^2}{\jap{s}} \, \ud s \\
 &\lesssim \frac{1}{\jap{t}} \bigl( \|\jx^2 v_0 \|_{H^1_x} + M(T)^2 \bigr).
\end{align*}
The derivation of the local decay bound for $\px \pt (e^{-it} v(t))$ proceeds analogously.

\medskip 

\noindent \underline{{\it Local decay for $\px^j v(t)$:}} Next, we turn to the local decay bound for $\partial_x v(t)$. Using the local decay estimate~\eqref{equ:local_decay_van1} for the Klein-Gordon propagator, it is straightforward to obtain the desired bound
\begin{align*}
 \| \jap{x}^{-2} \partial_x v(t) \|_{L^2_x} &\lesssim \frac{\|\jx^2 v_0\|_{H^1_x}}{\jap{t}^{\frac{3}{2}}} + \int_0^t \biggl\| \jap{x}^{-2} \frac{\partial_x}{\jap{\nabla}} e^{i(t-s)\jap{\nabla}}  \jap{x}^{-2} \biggr\|_{L^2_x \to L^2_x} \bigr\| \jap{x}^2 \bigl( \alpha(\cdot) u(s)^2 \bigr) \bigr\|_{L^2_x} \, \ud s \\
 &\lesssim \frac{\|\jx^2 v_0\|_{H^1_x}}{\jap{t}^{\frac{3}{2}}} + \int_0^t \frac{1}{\jap{t-s}^{\frac{3}{2}}} \| \jap{x}^6 \alpha(x) \|_{L^2_x} \|\jap{x}^{-2} v(s)\|_{L^\infty_x}^2 \, \ud s \\
 &\lesssim \frac{\|\jx^2 v_0\|_{H^1_x}}{\jap{t}^{\frac{3}{2}}} + \int_0^t \frac{1}{\jap{t-s}^{\frac{3}{2}}} \frac{M(T)^2}{\jap{s}} \, \ud s \\
 &\lesssim \frac{1}{\jap{t}} \bigl( \|\jx^2 v_0\|_{H^1_x} + M(T)^2 \bigr).
\end{align*}
In the same manner we also obtain the desired local decay bounds for higher derivatives $\px^j v(t)$, namely 
\begin{equation*}
 \| \jap{x}^{-2} \partial_x^j v(t) \|_{L^2_x} \lesssim \frac{1}{\jap{t}} \bigl( \|\jx^2 v_0\|_{H^j_x}  + M(T)^2 \bigr) \quad \text{ for } j = 2, 3.
\end{equation*}

\medskip 

\noindent \underline{{\it Local decay for $v(t)$}:} Finally, the derivation of the local decay estimate for $v(t)$ requires a more careful argument. It suffices to describe how to treat the nonlinear part of Duhamel's formula for $v(t)$. By freely adding and subtracting $u(s,0)^2$, we rewrite it as
\begin{equation} \label{equ:resonant_duhamel_addsubtract}
 \begin{aligned}
  &\frac{1}{2i} \int_0^t e^{i(t-s)\jap{\nabla}} \jn^{-1} \bigl( \alpha(\cdot) u(s, \cdot)^2 \bigr) \, \ud s \\
  &\quad \quad = \frac{1}{2i} \int_0^t e^{i(t-s)\jap{\nabla}} \jn^{-1} \Bigl( \alpha(\cdot) \bigl( u(s, \cdot)^2 - u(s,0)^2 \bigr) \Bigr) \, \ud s + \frac{1}{2i} \int_0^t \bigr( e^{i(t-s)\jap{\nabla}} \jn^{-1} \alpha \bigr) u(s,0)^2 \, \ud s.
 \end{aligned}
\end{equation}
In the first term on the right-hand side we can pick up a derivative via the fundamental theorem of calculus and therefore expect better decay properties in view of the stronger local decay bound for $\px v(t)$. Indeed, for any $x \in \bbR$ we have 
\begin{equation} \label{equ:resonant_trick_derivative}
 \begin{aligned}
  u(s, x)^2 - u(s,0)^2 &= \bigl( u(s, x) - u(s, 0) \bigr) \bigl( u(s,x) + u(s,0) \bigr) \\
  &= x \Bigl( \int_0^1 (\partial_x u)(s, \eta x) \, \ud \eta \Bigr) (u(s,x) + u(s,0)).
 \end{aligned}
\end{equation}
Using the local decay estimate~\eqref{equ:local_decay} for the Klein-Gordon propagator and a change of variables, we obtain the desired bound
\begin{align*}
 &\biggl\| \jap{x}^{-2} \int_0^t e^{i(t-s)\jap{\nabla}} \jn^{-1} \Bigl( \alpha(\cdot) \bigl( u(s, \cdot)^2 - u(s,0)^2 \bigr) \Bigr) \, \ud s \biggr\|_{L^2_x} \\
 &\quad \lesssim \int_0^t \biggl\| \jap{x}^{-2} \frac{e^{i(t-s)\jap{\nabla}}}{\jap{\nabla}} \jap{x}^{-2} \biggr\|_{L^2_x \to L^2_x} \biggl( \int_0^1 \bigl\| \jap{x}^2 x \alpha(x) (\partial_x u)(s, \eta x) (u(s,x) + u(s,0)) \bigr\|_{L^2_x} \, \ud \eta \biggr) \, \ud s \\
 &\quad \lesssim \int_0^t \frac{1}{\jap{t-s}^{\frac{1}{2}}} \bigl\| \jap{x}^7 \alpha(x) \bigr\|_{L^\infty_x} \| \jap{x}^{-2} v(s) \|_{L^\infty_x} \biggl( \int_0^1  \bigl\| \jap{\eta x}^{-2} (\partial_x v)(s, \eta x) \bigr\|_{L^2_x}  \, \ud \eta \biggr) \, \ud s \\
 &\quad \lesssim \int_0^t \frac{1}{\jap{t-s}^{\frac{1}{2}}} \bigl\| \jap{x}^7 \alpha(x) \bigr\|_{L^\infty_x} \| \jap{x}^{-2} v(s) \|_{L^\infty_x} \| \jap{x}^{-2} (\partial_x v)(s) \|_{L^2_x} \biggl( \int_0^1 \eta^{-\frac{1}{2}} \, \ud \eta \biggr) \, \ud s \\ 
 &\quad \lesssim \int_0^t \frac{1}{\jap{t-s}^{\frac{1}{2}}} \frac{M(T)^2}{\jap{s}^{\frac{3}{2}}} \, \ud s \\
 &\quad \lesssim \frac{M(T)^2}{\jap{t}^{\frac{1}{2}}}.
\end{align*}
It remains to estimate the last term on the right-hand side of~\eqref{equ:resonant_duhamel_addsubtract}, which is the heart of the matter. To this end we decompose the variable coefficient $\alpha(x)$ into a component that is frequency localized around $\xi \simeq \pm \sqrt{3}$ and a remainder term. Specifically, let $\varphi \in C_c^\infty(\bbR)$ be a smooth bump function with $\varphi(\xi) = 1$ in a small neighborhood of $0$. 
Then we define 
\begin{equation} \label{equ:definition_alpha1}
 \alpha_r(x) := \calF^{-1} \bigl[ \bigl(\varphi(\xi - \sqrt{3}) + \varphi(\xi + \sqrt{3}) \bigl) \widehat{\alpha}(\xi) \bigr](x)
\end{equation}
and set 
\[
 \alpha_{nr}(x) := \alpha(x) - \alpha_r(x).
\]
Correspondingly, we split the last term on the right-hand side of~\eqref{equ:resonant_duhamel_addsubtract} into 
\begin{equation} \label{equ:resonant_difficult_term_decompose}
 \begin{aligned}
  \frac{1}{2i} \int_0^t \bigr( e^{i(t-s)\jap{\nabla}} \jn^{-1} \alpha \bigr) u(s,0)^2 \, \ud s &= \frac{1}{2i} \int_0^t \bigr( e^{i(t-s)\jap{\nabla}} \jn^{-1} \alpha_r \bigr) u(s,0)^2 \, \ud s \\
  &\quad \quad + \frac{1}{2i} \int_0^t \bigr( e^{i(t-s)\jap{\nabla}} \jn^{-1} \alpha_{nr} \bigr) u(s,0)^2 \, \ud s. 
 \end{aligned}
\end{equation}
Since the Fourier transform of the coefficient $\alpha_r$ is supported away from zero frequency $\xi = 0$, there exists a smooth and decaying function $\widetilde{\alpha}_r(x)$ such that
\begin{equation*}
 \alpha_r(x) = \partial_x \widetilde{\alpha}_r(x).
\end{equation*}
Thus, by the local decay estimate~\eqref{equ:local_decay_van1} for the Klein-Gordon propagator we obtain for the first term on the right-hand side of the decomposition~\eqref{equ:resonant_difficult_term_decompose} that
\begin{align*}
 \biggl\| \jap{x}^{-2} \frac{1}{2i} \int_0^t \bigr( e^{i(t-s)\jap{\nabla}} \jn^{-1} \alpha_r \bigr) u(s,0)^2 \, \ud s \biggr\|_{L^2_x} &\lesssim \int_0^t \biggl\| \jap{x}^{-2} \frac{\partial_x}{\jap{\nabla}} e^{+i(t-s)\jap{\nabla}} \widetilde{\alpha}_r \biggr\|_{L^2_x} |v(s,0)|^2 \, \ud s \\
 &\lesssim \int_0^t \frac{1}{\jap{t-s}^{\frac{3}{2}}} \| \jap{x}^2 \widetilde{\alpha}_r(x) \|_{L^2_x} \| \jap{x}^{-2} v(s) \|_{L^\infty_x}^2 \, \ud s \\
 &\lesssim \int_0^t \frac{1}{\jap{t-s}^{\frac{3}{2}}} \frac{M(T)^2}{\jap{s}} \, \ud s \\
 &\lesssim \frac{M(T)^2}{\jap{t}}.
\end{align*}
Finally, we estimate the last term on the right-hand side of~\eqref{equ:resonant_difficult_term_decompose}. To this end we insert 
\begin{equation*}
 u(s,0) = e^{+is} \bigl( e^{-is} v(s,0) \bigr) + e^{-is} \bigl( e^{+is} \bar{v}(s,0) \bigr)
\end{equation*}
to find that
\begin{equation} 
 \begin{aligned}
  &\frac{1}{2i} \int_0^t \bigr( e^{i(t-s)\jap{\nabla}} \jn^{-1} \alpha_{nr} \bigr) u(s,0)^2 \, \ud s \\
  &\quad \quad = \frac{1}{2i} \int_0^t \bigr( e^{i(t-s)\jap{\nabla}} \jn^{-1} \alpha_{nr} \bigr)  \Bigl( e^{+is} \bigl( e^{-is} v(s,0) \bigr) + e^{-is} \bigl( e^{+is} \bar{v}(s,0) \bigr) \Bigr)^2 \, \ud s \\
  &\quad \quad = \frac{1}{2i} \int_0^t \bigr( e^{i(t-s)\jap{\nabla}} \jn^{-1} \alpha_{nr} \bigr)  e^{+2is} \bigl( e^{-is} v(s,0) \bigr)^2 \, \ud s \\
  &\quad \qquad + \frac{1}{2i} \int_0^t \bigr( e^{i(t-s)\jap{\nabla}} \jn^{-1} \alpha_{nr} \bigr)  2 \bigl( e^{-is} v(s,0) \bigr) \bigl( e^{+is} \bar{v}(s,0) \bigr) \, \ud s \\ 
  &\quad \qquad + \frac{1}{2i} \int_0^t \bigr( e^{i(t-s)\jap{\nabla}} \jn^{-1} \alpha_{nr} \bigr)  e^{-2is} \bigl( e^{+is} \bar{v}(s,0) \bigr)^2 \, \ud s \\
  &\quad \quad \equiv I + II + III.
 \end{aligned}
\end{equation}
Now we describe in detail how to treat the term $I$. All other terms can be handled similarly. Since the Fourier transform of the coefficient $\alpha_{nr}$ is supported away from the frequencies $\xi = \pm \sqrt{3}$ and since the symbol $2-\jxi$ vanishes if and only if $\xi = \pm \sqrt{3}$, we can integrate by parts in $s$ to find that
\begin{equation} 
 \begin{aligned}
 I &= -\frac{1}{2} \Bigl( (2-\jap{\nabla})^{-1} \jn^{-1} \alpha_{nr} \Bigr) e^{2it} \bigl( e^{-it} v(t,0) \bigr)^2 \\
 &\quad + \frac{1}{2} \Bigl( e^{it\jn} (2-\jap{\nabla})^{-1} \jn^{-1} \alpha_{nr} \Bigr) \bigl( v(0,0) \bigr)^2 \\
 &\quad + \frac{1}{2} \int_0^t \Bigl( e^{i(t-s)\jn} (2-\jap{\nabla})^{-1} \jn^{-1} \alpha_{nr} \Bigr) e^{2is} \partial_s \bigl( e^{-is} v(s,0) \bigr) \bigl( e^{-is} v(s,0) \bigr) \, \ud s \\
 &\equiv I_{(a)} + I_{(b)} + I_{(c)}.
 \end{aligned}
\end{equation}
Then we can easily estimate the term $I_{(a)}$ by
\begin{equation*}
 \| \jap{x}^{-2} I_{(a)} \|_{L^2_x} \lesssim \bigl\| (2-\jap{\nabla})^{-1} \jn^{-1} \alpha_{nr} \bigr\|_{L^2_x} |v(t,0)|^2 \lesssim \| \jap{x}^{-2} v(t) \|_{L^\infty_x}^2 \lesssim \frac{M(T)^2}{\jap{t}}.
\end{equation*}
The bound on the second term $I_{(b)}$ just follows from the local decay estimate~\eqref{equ:local_decay} for the Klein-Gordon propagator acting on the smooth and decaying function $(2-\jap{\nabla})^{-1} \jn^{-1} \alpha_{nr}$, to wit 
\begin{align*}
 \| \jap{x}^{-2} I_{(b)} \|_{L^2_x} &\lesssim \Bigl\| \jap{x}^{-2} e^{it\jap{\nabla}} (2 - \jap{\nabla})^{-1} \jn^{-1} \alpha_{nr} \Bigr\|_{L^2_x} |v(0,0)|^2 \\
 &\lesssim \frac{1}{\jap{t}^{\frac{1}{2}}} \Bigl\| \jx^2 (2 - \jap{\nabla})^{-1} \jn^{-1} \alpha_{nr} \Bigr\|_{L^2_x} \|v_0\|_{H^1_x}^2.
\end{align*}
Lastly, we obtain that
\begin{align*}
 \| \jap{x}^{-2} I_{(c)} \|_{L^2_x} &\lesssim \int_0^t \bigl\| \jap{x}^{-2} e^{i(t-s)\jn} (2-\jap{\nabla})^{-1} \jn^{-1} \alpha_{nr} \bigr\|_{L^2_x} \bigl| \partial_s \bigl( e^{-is} v(s,0) \bigr) \bigr| \bigl| e^{-is} v(s,0) \bigr| \, \ud s \\
 &\lesssim \int_0^t \frac{1}{\jap{t-s}^{\frac{1}{2}}} \bigl\| \jap{x}^2 (2-\jap{\nabla})^{-1} \jn^{-1} \alpha_{nr} \bigr\|_{L^2_x} \bigl\| \jap{x}^{-2} \partial_s \bigl( e^{-is} v(s) \bigr) \bigr\|_{L^\infty_x} \|\jap{x}^{-2} v(s)\|_{L^\infty_x} \, \ud s \\
 &\lesssim \int_0^t \frac{1}{\jap{t-s}^{\frac{1}{2}}} \frac{M(T)}{\jap{s}} \frac{M(T)}{\jap{s}^{\frac{1}{2}}} \, \ud s \\
 &\lesssim \frac{M(T)^2}{\jap{t}^{\frac{1}{2}}}.
\end{align*}

Putting all of the above estimates together, we conclude that
\begin{equation*}
 M(T) \lesssim \varepsilon + M(T)^2.
\end{equation*}
The assertion of Proposition~\ref{prop:resonant_bootstrap_bounds} now follows by a standard continuity argument.
\end{proof}

\subsection{Asymptotics at the origin}

Having the local decay bounds from Proposition~\ref{prop:resonant_bootstrap_bounds} at our disposal, we are already in a position to determine the asymptotics of the solution $v(t)$ to~\eqref{equ:resonant_nlkg} at the origin $x=0$.

\begin{proposition}[Asymptotics at the origin] \label{prop:resonant_asymptotics_origin}
 Suppose that the assumptions of Proposition~\ref{prop:resonant_bootstrap_bounds} are in place. Then there exists a small amplitude $a_0 \in \bbC$ with $|a_0| \lesssim \varepsilon$ so that the asymptotics of the solution $v(t)$ to~\eqref{equ:resonant_nlkg} at the origin $x=0$ are given by 
 \begin{equation} \label{equ:resonant_asymptotics_origin}
  v(t,0) = \frac{1}{t^\hf} e^{i \frac{\pi}{4}} e^{it} a_0 + r(t) \quad \text{ for all } t \geq 1,
 \end{equation}
 with a remainder term satisfying
 \begin{equation*}
  |r(t)| \lesssim \frac{\varepsilon}{t^{1-}}.
 \end{equation*}
 Analogous asymptotics for $v(t,0)$ hold for negative times.
\end{proposition}

\begin{proof}
By Duhamel's formula for $v(t)$ we have that 
\begin{equation} \label{equ:resonant_asymptotics_duhamel_v}
 v(t,0) = \bigl( e^{it\jn} v_0 \bigr)(0) + \frac{1}{2i} \int_0^t \Bigl( e^{i(t-s)\jap{\nabla}} \jn^{-1} \bigl( \alpha(\cdot) u(s, \cdot)^2 \bigr) \Bigr)(0) \, \ud s.
\end{equation}
The asymptotics for the Klein-Gordon propagator as in Remark~\ref{rem:stronger_remainder_decay} give 
\begin{equation*}
 \bigl( e^{it\jn} v_0 \bigr)(0) = \frac{1}{t^\hf} e^{i\frac{\pi}{4}} e^{it} \hat{v}_0(0) + \frac{1}{t} \calO \bigl( \|\jx^2 v_0\|_{H^4_x} \bigr), \quad t \geq 1.
\end{equation*}
We can therefore focus on determining the asymptotics of the nonlinear term on the right-hand side of~\eqref{equ:resonant_asymptotics_duhamel_v}. Let $t \geq 1$. The contribution from integration over the time interval $t-1 \leq s \leq t$ can easily be seen to be of order $\calO(\varepsilon^2 t^{-1})$ using the local decay bounds for $v(t)$ from~Proposition~\ref{prop:resonant_bootstrap_bounds}. On the time integration interval $0 \leq s \leq t-1$ we may insert the asymptotics from Remark~\ref{rem:stronger_remainder_decay} for the retarded Klein-Gordon propagator $e^{i(t-s)\jap{\nabla}}$ to find that
\begin{equation} \label{equ:resonant_asymptotics_insert_for_retarded}
 \begin{aligned}
  \int_0^{t-1} \Bigl( e^{i(t-s)\jap{\nabla}} \jn^{-1} \bigl( \alpha(\cdot) u(s, \cdot)^2 \bigr) \Bigr)(0) \, \ud s &= \int_0^{t-1} e^{i\frac{\pi}{4}} \frac{ e^{i(t-s)} }{(t-s)^{\frac{1}{2}}} \calF\bigl[ \alpha(\cdot) u(s, \cdot)^2 \bigr](0) \, \ud s + \tilde{r}(t) \\
  &= e^{i\frac{\pi}{4}} \frac{e^{it}}{t^{\frac{1}{2}}} \int_0^{t-1} \frac{ t^{\frac{1}{2}} }{(t-s)^{\frac{1}{2}}} e^{-is} \calF\bigl[ \alpha(\cdot) u(s, \cdot)^2 \bigr](0) \, \ud s + \tilde{r}(t), 
 \end{aligned}
\end{equation}
where the remainder term $\tilde{r}(t)$ satisfies the stronger decay estimate
\begin{align*}
 |\tilde{r}(t)| \lesssim \int_0^{t-1} \frac{1}{t-s} \bigl\| \jap{x}^2 \alpha(x) u(s,x)^2 \bigr\|_{H^3_x} \, \ud s &\lesssim \int_0^{t-1} \frac{1}{t-s} \bigl\| \jap{x}^6 \alpha(x) \bigr\|_{H^3_x} \bigl\| \jap{x}^{-2} v(s) \bigr\|_{H^3_x}^2 \, \ud s \\
 &\lesssim \int_0^{t-1} \frac{1}{t-s} \frac{\varepsilon^2}{\jap{s}} \, \ud s \\
 &\lesssim \frac{\varepsilon^2}{\jap{t}^{1-}}.
\end{align*}
It remains to determine the leading order behavior of the first term on the right-hand side of~\eqref{equ:resonant_asymptotics_insert_for_retarded}. To this end we insert the decomposition 
\begin{equation*}
 u(s) = e^{+is} \bigl( e^{-is} v(s) \bigr) + e^{-is} \bigl( e^{+is} \bar{v}(s) \bigr)
\end{equation*}
to find that
\begin{align*}
 &\sqrt{2\pi} \int_0^{t-1} \frac{ t^{\frac{1}{2}} }{(t-s)^{\frac{1}{2}}} e^{-is} \calF\bigl[ \alpha(\cdot) u(s, \cdot)^2 \bigr](0) \, \ud s \\
 &\quad \quad = \int_0^{t-1} \frac{ t^{\frac{1}{2}} }{(t-s)^{\frac{1}{2}}} e^{-is} \biggl( \int_{\bbR} \alpha(x) u(s,x)^2 \, \ud x \biggr) \, \ud s \\
 &\quad \quad = \int_0^{t-1} \frac{ t^{\frac{1}{2}} }{(t-s)^{\frac{1}{2}}} e^{+is} \biggl( \int_{\bbR} \alpha(x) \bigl( e^{-is} v(s,x) \bigr)^2 \, \ud x \biggr) \, \ud s \\
 &\quad \quad \quad + 2 \int_0^{t-1} \frac{ t^{\frac{1}{2}} }{(t-s)^{\frac{1}{2}}} e^{-is} \biggl( \int_{\bbR} \alpha(x) \bigl( e^{-is} v(s,x) \bigr) \bigl( e^{+is} \bar{v}(s,x) \bigr) \, \ud x \biggr) \, \ud s \\
 &\quad \quad \quad + \int_0^{t-1} \frac{ t^{\frac{1}{2}} }{(t-s)^{\frac{1}{2}}} e^{-3is} \biggl( \int_{\bbR} \alpha(x) \bigl( e^{+is} \bar{v}(s,x) \bigr)^2 \, \ud x \biggr) \, \ud s \\
 &\quad \quad \equiv I + II + III.
\end{align*}
Then we exploit the oscillations in each term and integrate by parts in time. For the term $I$ we find that
\begin{align*}
 I &= -i t^{\frac{1}{2}} e^{+i(t-1)} \int_{\bbR} \alpha(x) \bigl( e^{-i(t-1)} v(t-1, x) \bigr)^2 \, \ud x \\ 
 &\quad \quad + i \int_{\bbR} \alpha(x) v(0,x)^2 \, \ud x \\
 &\quad \quad + \frac{i}{2} \int_0^{t-1} \frac{ t^{\frac{1}{2}} }{(t-s)^{\frac{3}{2}}} e^{+is} \biggl( \int_{\bbR} \alpha(x) \bigl( e^{-is} v(s,x) \bigr)^2 \, \ud x \biggr) \, \ud s \\
 &\quad \quad + 2i \int_0^{t-1} \frac{ t^{\frac{1}{2}} }{(t-s)^{\frac{1}{2}}} e^{+is} \biggl( \int_{\bbR} \alpha(x) \partial_s \bigl( e^{-is} v(s,x) \bigr) \bigl( e^{-is} v(s,x) \bigr) \, \ud x \biggr) \, \ud s \\
 &\equiv I_{(a)} + I_{(b)} + I_{(c)} + I_{(d)}.
\end{align*}
It is clear that the term $I_{(b)}$ contributes to the leading order behavior of $I$. Using the local decay estimates for $v(t)$ from Proposition~\ref{prop:resonant_bootstrap_bounds} we obtain 
\begin{equation*}
 | I_{(a)} | \lesssim t^{\frac{1}{2}} \| \jap{x}^4 \alpha(x) \|_{L^\infty_x} \bigl\| \jap{x}^{-2} v(t-1) \bigr\|_{L^2_x}^2 \lesssim \frac{\varepsilon^2}{\jap{t}^{\frac{1}{2}}}
\end{equation*}
and
\begin{equation*}
 | I_{(c)} | \lesssim t^{\frac{1}{2}} \int_0^{t-1} \frac{1}{(t-s)^{\frac{3}{2}}} \| \jap{x}^4 \alpha(x) \bigr\|_{L^\infty_x} \bigl\| \jap{x}^{-2} v(s,x) \bigr\|_{L^2_x}^2 \, \ud s \lesssim t^{\frac{1}{2}} \int_0^{t-1} \frac{1}{(t-s)^{\frac{3}{2}}} \frac{\varepsilon^2}{\jap{s}} \, \ud s \lesssim \frac{\varepsilon^2}{\jap{t}^{\frac{1}{2}}}.
\end{equation*}
Then we rewrite the last term $I_{(d)}$ as
\begin{align*}
 I_{(d)} &= 2i \int_0^{t-1} \frac{ t^{\frac{1}{2}} }{(t-s)^{\frac{1}{2}}} e^{+is} \biggl( \int_{\bbR} \alpha(x) \partial_s \bigl( e^{-is} v(s,x) \bigr) \bigl( e^{-is} v(s,x) \bigr) \, \ud x \biggr) \, \ud s \\ 
 &= 2i \int_0^\infty e^{+is} \int_{\bbR} \alpha(x) \partial_s \bigl( e^{-is} v(s,x) \bigr) \bigl( e^{-is} v(s,x) \bigr) \, \ud x \, \ud s \\
 &\quad -2i \int_{\frac{t}{2}}^\infty e^{+is} \int_{\bbR} \alpha(x) \partial_s \bigl( e^{-is} v(s,x) \bigr) \bigl( e^{-is} v(s,x) \bigr) \, \ud x \, \ud s \\
 &\quad + 2i \int_{\frac{t}{2}}^{t-1} \frac{t^{\frac{1}{2}}}{(t-s)^{\frac{1}{2}}} e^{+is} \int_{\bbR} \alpha(x) \partial_s \bigl( e^{-is} v(s,x) \bigr) \bigl( e^{-is} v(s,x) \bigr) \, \ud x \, \ud s \\
 &\quad + 2i \int_0^{\frac{t}{2}} \biggl( \frac{t^{\frac{1}{2}}}{(t-s)^{\frac{1}{2}}} - 1 \biggr) e^{+is} \int_{\bbR} \alpha(x) \partial_s \bigl( e^{-is} v(s,x) \bigr) \bigl( e^{-is} v(s,x) \bigr) \, \ud x \, \ud s \\
 &\equiv I_{(d)}^{(1)} + I_{(d)}^{(2)} + I_{(d)}^{(3)} + I_{(d)}^{(4)}.
\end{align*}
Using the local decay estimates from Proposition~\ref{prop:resonant_bootstrap_bounds} it is easy to see that the improper integral $I_{(d)}^{(1)}$ converges and contributes to the leading order behavior of the term $I$, while the other terms $I_{(d)}^{(2)}$, $I_{(d)}^{(3)}$, and $I_{(d)}^{(4)}$ can be seen to be of order $\calO( \varepsilon^2 \jap{t}^{-\frac{1}{2}}) $. Indeed, we have 
\begin{align*}
 \bigl| I_{(d)}^{(1)} \bigr| &\lesssim \int_0^\infty \| \jap{x}^4 \alpha(x) \|_{L^\infty_x} \bigl\| \jap{x}^{-2} \partial_s \bigl( e^{-is} v(s) \bigr) \bigr\|_{L^2_x} \bigl\| \jap{x}^{-2} v(s) \bigr\|_{L^2_x} \, \ud s \lesssim \int_{0}^\infty \frac{\varepsilon^2}{\jap{s}^{\frac{3}{2}}} \, \ud s \lesssim \varepsilon^2
\end{align*}
and 
\begin{align*}
 \bigl| I_{(d)}^{(2)} \bigr| &\lesssim \int_{\frac{t}{2}}^\infty \| \jap{x}^4 \alpha(x) \|_{L^\infty_x} \bigl\| \jap{x}^{-2} \partial_s \bigl( e^{-is} v(s) \bigr) \bigr\|_{L^2_x} \bigl\| \jap{x}^{-2} v(s) \bigr\|_{L^2_x} \, \ud s \lesssim \int_{\frac{t}{2}}^\infty \frac{\varepsilon^2}{\jap{s}^{\frac{3}{2}}} \, \ud s \lesssim \frac{\varepsilon^2}{\jap{t}^{\frac{1}{2}}}.
\end{align*}
Similarly, we find that 
\begin{align*}
 \bigl| I_{(d)}^{(3)} \bigr| &\lesssim \int_{\frac{t}{2}}^{t-1} \frac{t^{\frac{1}{2}}}{(t-s)^{\frac{1}{2}}} \frac{\varepsilon^2}{\jap{s}^{\frac{3}{2}}} \, \ud s \lesssim \frac{\varepsilon^2}{\jap{t}^{\frac{1}{2}}}
\end{align*}
and
\begin{align*}
 \bigl| I_{(d)}^{(4)} \bigr| &\lesssim \int_0^{\frac{t}{2}} \frac{s}{(t-s)^{\frac{1}{2}} ( t^{\frac{1}{2}} + (t-s)^{\frac{1}{2}} )} \frac{\varepsilon^2}{\jap{s}^{\frac{3}{2}}} \ud s \lesssim \frac{\varepsilon^2}{\jap{t}^{\frac{1}{2}}}.
\end{align*}
Hence, we have found that the leading order behavior of the term $I$ is given by
\begin{align*}
 I = i \int_{\bbR} \alpha(x) v(0,x)^2 \, \ud x + 2i \int_0^\infty e^{+is} \biggl( \int_{\bbR} \alpha(x) \partial_s \bigl( e^{-is} v(s,x) \bigr) \bigl( e^{-is} v(s,x) \bigr) \, \ud x \biggr) \, \ud s + \calO \biggl( \frac{\varepsilon^2}{\jap{t}^{\frac{1}{2}}} \biggr).
\end{align*}
In an analogous manner, we compute that the leading order behaviors of the terms $II$ and $III$ are given by
\begin{align*}
 II &= - 2i \int_{\bbR} \alpha(x) |v(0,x)|^2 \, \ud x - 2i \int_0^\infty e^{-is}  \biggl( \int_{\bbR} \alpha(x) \partial_s \Bigl( \bigl( e^{-is} v(s,x) \bigr) \bigl( e^{+is} \bar{v}(s,x) \bigr) \Bigr) \, \ud x \biggr) \, \ud s + \calO \biggl( \frac{\varepsilon^2}{\jap{t}^{\frac{1}{2}}} \biggr)
\end{align*}
and 
\begin{align*}
 III &= - \frac{i}{3} \int_{\bbR} \alpha(x) \bar{v}(0,x)^2 \, \ud x - \frac{2i}{3} \int_0^\infty e^{-3is} \biggl( \int_{\bbR} \alpha(x) \partial_s \bigl( e^{+is} \bar{v}(s,x) \bigr) \bigl( e^{+is} \bar{v}(s,x) \bigr) \, \ud x \biggr) \, \ud s + \calO \biggl( \frac{\varepsilon^2}{\jap{t}^{\frac{1}{2}}} \biggr).
\end{align*}

Putting things together, we conclude that the asymptotic behavior of $v(t,0)$ is
\begin{equation*}
 v(t,0) = \frac{1}{t^\hf} e^{i\frac{\pi}{4}} e^{it} a_0 + \calO \biggl( \frac{\varepsilon}{\jap{t}^{1-}} \biggr), \quad t \geq 1, 
\end{equation*}
where the amplitude $a_0$ is given by
\begin{equation} \label{equ:resonant_amplitude_a0}
\begin{aligned}
 a_0 &= \hat{v}_0(0) + \frac{1}{\sqrt{2\pi}} \biggl( \frac{1}{2} \int_{\bbR} \alpha(x) v(0,x)^2 \, \ud x - \int_{\bbR} \alpha(x) |v(0,x)|^2 \, \ud x - \frac{1}{6} \int_{\bbR} \alpha(x) \bar{v}(0,x)^2 \, \ud x \biggr) \\
 &\quad \quad \quad \, \, \, + \frac{1}{\sqrt{2\pi}} \biggl( \int_0^\infty e^{+is} \int_{\bbR} \alpha(x) \partial_s \bigl( e^{-is} v(s,x) \bigr) \bigl( e^{-is} v(s,x) \bigr) \, \ud x \, \ud s \\
 &\qquad \qquad \qquad \qquad - \int_0^\infty e^{-is} \int_{\bbR} \alpha(x) \partial_s \Bigl( \bigl( e^{-is} v(s,x) \bigr) \bigl( e^{+is} \bar{v}(s,x) \bigr) \Bigr) \, \ud x \, \ud s \\
 &\qquad \qquad \qquad \qquad - \frac{1}{3} \int_0^\infty e^{-3is} \int_{\bbR} \alpha(x) \partial_s \bigl( e^{+is} \bar{v}(s,x) \bigr) \bigl( e^{+is} \bar{v}(s,x) \bigr) \, \ud x \, \ud s \biggr).
\end{aligned}
\end{equation}
This finishes the proof of Proposition~\ref{prop:resonant_asymptotics_origin}.
\end{proof}

\subsection{Proof of Theorem~\ref{thm:resonant}}

 In this subsection we give the proof of Theorem~\ref{thm:resonant}. We begin by deriving the decay estimate~\eqref{equ:resonant_thm_sharp_decay} for the solution $v(t)$ to~\eqref{equ:resonant_nlkg}.
 Using the standard dispersive decay estimate~\eqref{equ:dispersive_decay} for the Klein-Gordon propagator as well as the local decay estimates for $v(t)$ established in Proposition~\ref{prop:resonant_bootstrap_bounds}, we obtain from the Duhamel representation of $v(t)$ that
 \begin{align*}
  \|v(t)\|_{L^\infty_x} &\lesssim \frac{\|\jn^2 v_0\|_{L^1_x}}{\jt^\hf}  + \int_0^t \frac{1}{\jap{t-s}^\hf} \bigl\| \jn \bigl( \alpha(\cdot) u(s)^2 \bigr) \bigr\|_{L^1_x} \, \ud s \\
  &\lesssim \frac{\|\jx v_0\|_{H^2_x} }{\jt^{\hf}} + \int_0^t \frac{1}{\jap{t-s}^\hf}  \Bigl( \bigl\| \jx^{-2} v(s) \bigr\|_{L^2_x}^2 + \bigl\| \jx^{-2} \px v(s) \bigr\|_{L^2_x} \bigl\| \jx^{-2} v(s) \bigr\|_{L^2_x} \Bigr) \, \ud s \\
  &\lesssim \frac{\varepsilon}{\jt^\hf} + \int_0^t \frac{1}{\jap{t-s}^\hf} \frac{\varepsilon^2}{\js} \, \ud s \\
  &\lesssim \frac{\log(1+\jt)}{\jt^\hf} \varepsilon.
 \end{align*}
 
 Next we study the long-time behavior of the solution $v(t)$ in more detail. We note that by time-reversal symmetry it suffices to consider times $t>0$. By Proposition~\ref{prop:resonant_asymptotics_origin} the asymptotics of $v(t)$ at the origin $x=0$ are given by
 \begin{equation*}
  v(t,0) = \frac{1}{t^\hf} e^{i\frac{\pi}{4}} e^{it} a_0 + r(t), \quad t \geq 1,
 \end{equation*}
 where $a_0 \in \bbC$ with $|a_0| \lesssim \varepsilon$ and $|r(t)| \lesssim \varepsilon t^{-(1-)}$. Hence, we may write 
 \begin{align*}
  u(t,0)^2 = \bigl( v(t,0) + \bar{v}(t,0) \bigr)^2 = \frac{1}{t} e^{i\frac{\pi}{2}} e^{2it} a_0^2 + \frac{2}{t} |a_0|^2 + \frac{1}{t} e^{-i\frac{\pi}{2}} e^{-2it} \bar{a}_0^2 + \tilde{r}(t)
 \end{align*}
 for some remainder term $\tilde{r}(t)$ satisfying $|\tilde{r}(t)| \lesssim \varepsilon^2 t^{-(\frac{3}{2}-)}$. 
 Inserting this expansion into the Duhamel formula for $v(t)$, we obtain for times $t \geq 1$ that
 \begin{equation*}
  \begin{aligned}
   v(t) &= e^{it\jn} v_0 + \frac{1}{2i} \int_0^t e^{i(t-s)\jn} \jn^{-1} \bigl( \alpha(\cdot) u(s)^2 \bigr) \, \ud s \\
   &= e^{it\jn} v_0 + \frac{1}{2i} \int_0^1 e^{i(t-s)\jn} \jn^{-1} \bigl( \alpha(\cdot) u(s)^2 \bigr) \, \ud s + \frac{1}{2i} \int_1^t e^{i(t-s)\jn} \jn^{-1} \Bigl( \alpha(\cdot) \bigl( u(s)^2 - u(s,0)^2 \bigr) \Bigr) \, \ud s \\
   &\qquad + \frac{a_0^2}{2} \int_1^t \bigl( e^{i(t-s)\jn} \jn^{-1} \alpha \bigr) \frac{e^{2is}}{s} \, \ud s -i |a_0|^2 \int_1^t \bigl( e^{i(t-s)\jn} \jn^{-1} \alpha \bigr) \frac{1}{s} \, \ud s \\
   &\qquad - \frac{\bar{a}_0^2}{2} \int_1^t \bigl( e^{i(t-s)\jn} \jn^{-1} \alpha \bigr) \frac{e^{-2is}}{s} \, \ud s + \frac{1}{2i} \int_1^t \bigl( e^{i(t-s)\jn} \jn^{-1} \alpha \bigr) \tilde{r}(s) \, \ud s.
  \end{aligned}
 \end{equation*}
 We find below that the modified scattering behavior of $v(t)$ is caused by the component 
 \begin{equation*}
  v_{mod}(t) := \frac{a_0^2}{2} \int_1^t \bigl( e^{i(t-s)\jn} \jn^{-1} \alpha \bigr) \frac{e^{2is}}{s} \, \ud s.
 \end{equation*}
 All other components of $v(t)$ turn out to behave asymptotically like linear Klein-Gordon waves. Correspondingly,  we set 
 \begin{equation*}
  v_{free}(t) := v(t) - v_{mod}(t).
 \end{equation*}

 \medskip  
 
 \noindent \underline{{\it Decay and asymptotics of $v_{free}(t)$}:}
 We denote the profile of $v_{free}(t)$ by $g(t) = e^{-it\jn} v_{free}(t)$. In what follows, we establish that
 \begin{equation} \label{equ:resonant_Linfty_bound_profile}
  \sup_{t \geq 0} \, \bigl\| \jxi^\thf \hat{g}(t,\xi) \bigr\|_{L^\infty_\xi} \lesssim \varepsilon
 \end{equation}
 and that there exists $\widehat{V} \in L^\infty$ such that
 \begin{equation} \label{equ:resonant_limit_profile}
  \bigl\| \widehat{V}(\xi) - \jxi^\thf \hat{g}(t,\xi) \bigr\|_{L^\infty} \lesssim \frac{\varepsilon^2}{\jt^{\hf-}}, \quad t \geq 1.
 \end{equation}
 Moreover, we show that 
 \begin{equation} \label{equ:resonant_growth_bound_profile}
  \sup_{t \geq 0} \, \jt^{-(0+)} \bigl\| \jx g(t) \bigr\|_{H^2_x} \lesssim \varepsilon.
 \end{equation}
 Then the asymptotics~\eqref{equ:resonant_thm_asymptotics_vfree} of $v_{free}(t)$ asserted in Theorem~\ref{thm:resonant} are a standard consequence of~\eqref{equ:resonant_Linfty_bound_profile}--\eqref{equ:resonant_growth_bound_profile} and Lemma~\ref{lem:KG_propagator_asymptotics}.
 Before we turn to the proofs of \eqref{equ:resonant_Linfty_bound_profile}--\eqref{equ:resonant_growth_bound_profile} we record that the Fourier transform $\hat{g}(t,\xi)$ of the profile~$g(t)$ is given by
 \begin{equation*}
  \begin{aligned}
   \hat{g}(t,\xi) &= \hat{v}_0(\xi) + \frac{1}{2i} \int_0^1 e^{-is \jxi} \jxi^{-1} \calF\bigl[ \alpha(\cdot) u(s)^2 \bigr](\xi) \, \ud s + \frac{1}{2i} \int_1^t e^{-is \jxi} \jxi^{-1} \calF\Bigl[ \alpha(\cdot) \bigl( u(s)^2 - u(s,0)^2 \bigr) \Bigr](\xi) \, \ud s \\
   &\quad -i |a_0|^2 \int_1^t e^{-is \jxi} \jxi^{-1} \widehat{\alpha}(\xi) \frac{1}{s} \, \ud s - \frac{\bar{a}_0^2}{2} \int_1^t e^{-is \jxi} \jxi^{-1} \widehat{\alpha}(\xi) \frac{e^{-2is}}{s} \, \ud s + \frac{1}{2i} \int_1^t e^{-is \jxi} \jxi^{-1} \widehat{\alpha}(\xi) \tilde{r}(s) \, \ud s \\
   &\equiv \hat{v}_0(\xi) + \hat{g}_1(t,\xi) + \ldots + \hat{g}_5(t,\xi).
  \end{aligned}
 \end{equation*}
 The phases in $\hat{g}_3(t,\xi)$ and in $\hat{g}_4(t,\xi)$ are non-stationary in $s$ so that we can recast these terms into a better form by integrating by parts
 \begin{align}
  \hat{g}_3(t,\xi) &= |a_0|^2 e^{-it\jxi} \jxi^{-2} \widehat{\alpha}(\xi) \frac{1}{t} - |a_0|^2 e^{-i\jxi} \jxi^{-2} \widehat{\alpha}(\xi) + |a_0|^2 \int_1^t e^{-is\jxi} \jxi^{-2} \widehat{\alpha}(\xi) \frac{1}{s^2} \, \ud s, \label{equ:resonant_g4hat_intbyparts} \\
  \hat{g}_4(t,\xi) &= \frac{\bar{a}_0^2}{2i} e^{-it(2+\jxi)} (2+\jxi)^{-1} \jxi^{-1} \widehat{\alpha}(\xi) \frac{1}{t} - \frac{\bar{a}_0^2}{2i} e^{-i(2+\jxi)} (2+\jxi)^{-1} \jxi^{-1} \widehat{\alpha}(\xi)   \label{equ:resonant_g5hat_intbyparts} \\
   &\quad - \frac{\bar{a}_0^2}{2i} \int_1^t e^{-is(2+\jxi)} (2+\jxi)^{-1} \jxi^{-1} \widehat{\alpha}(\xi) \frac{1}{s^2} \, \ud s. \nonumber
 \end{align}

 \begin{proof}[Proofs of~\eqref{equ:resonant_Linfty_bound_profile} and~\eqref{equ:resonant_limit_profile}]
 The bound~\eqref{equ:resonant_Linfty_bound_profile} follows by direct computation for each component $\hat{g}(t,\xi)$. It is clear for $\hat{v}_0(\xi)$. Then to estimate the components $\hat{g}_1(t,\xi)$ and $\hat{g}_2(t,\xi)$ we rely on the local decay bounds from Proposition~\ref{prop:resonant_bootstrap_bounds}. In particular, we note that the integrand in $\hat{g}_2(t,\xi)$ is integrable since we have 
 \begin{align*}
  \bigl\| \jxi^\hf \calF\bigl[ \alpha(\cdot) \bigl( u(s)^2 - u(s,0)^2 \bigr) \bigr](\xi) \bigr\|_{L^\infty_\xi} &\lesssim \bigl\| \jxi^\hf \calF\bigl[ \alpha(\cdot) \bigl( u(s)^2 - u(s,0)^2 \bigr) \bigr](\xi) \bigr\|_{H^1_\xi} \\
  &\lesssim \bigl\| \jx \alpha(x) \bigl( u(s)^2 - u(s,0)^2 \bigr) \bigr\|_{H^1_x} \\
  &\lesssim \bigl\| \jx^6 \alpha(x) \bigr\|_{H^1_x} \bigl\| \jx^{-2} \px v(s) \bigr\|_{L^2_x} \bigl\|\jx^{-2} v(s)\bigr\|_{H^1_x} \\
  &\lesssim \frac{\varepsilon^2}{\js^\thf}.
 \end{align*}
 The bounds for $\hat{g}_3(t,\xi)$ and $\hat{g}_4(t,\xi)$ follow directly using~\eqref{equ:resonant_g4hat_intbyparts} and~\eqref{equ:resonant_g5hat_intbyparts}, while the bound for~$\hat{g}_5(t,\xi)$ is straightforward since $\tilde{r}(s)$ is integrable. In fact, we obtain for any $t_1 \geq t_2 \geq 1$ that 
 \begin{align*}
  \bigl\| \jxi^\thf \hat{g}(t_1,\xi) - \jxi^\thf \hat{g}(t_2,\xi) \bigr\|_{L^\infty_\xi} &\lesssim \frac{\varepsilon^2}{t_2^{\hf-}}.
 \end{align*}
 Thus, $\bigl\{ \jxi^\thf \hat{g}(t,\xi) \bigr\}_{t \geq 1}$ is a Cauchy sequence in $L^\infty$, which implies the existence of a limit profile $\widehat{V} \in L^\infty$ satisfying \eqref{equ:resonant_limit_profile}.
 \end{proof}

 \begin{proof}[Proof of~\eqref{equ:resonant_growth_bound_profile}]
 We establish the estimate separately for each component of the profile $g(t)$. It is clear for $v_0$. Noting that $\|\jx g_j(t)\|_{H^2_x} \lesssim \| \jxi^2 \hat{g}_j(t) \|_{H^1_\xi}$ and using~\eqref{equ:resonant_g4hat_intbyparts}--\eqref{equ:resonant_g5hat_intbyparts}, it is a straightforward computation to obtain for $j = 1,3,4$ that $\|\jx g_j(t)\|_{H^2_x} \lesssim \varepsilon^2$ uniformly for all $t \geq 0$. The growth estimates for $g_2(t)$ and $g_5(t)$ have to be done more carefully. We present the details for $g_2(t)$. 
 Let $v_2(t) = e^{it\jn} g_2(t)$. Then it holds that
 \begin{equation*}
  \| \jx g_2(t) \|_{H^2_x} \lesssim \|v_2(t)\|_{H^2_x} + \|\jn L v_2(t)\|_{L^2_x}.
 \end{equation*}
 It is easy to see that $\|v_2(t)\|_{H^2_x} \lesssim \varepsilon^2$ uniformly for all $t \geq 0$. To bound the growth of the second term on the right-hand side we show below the auxiliary estimate
 \begin{equation} \label{equ:resonant_Lv3_local_decay}
  \sup_{t \geq 0} \, \jt^\hf \bigl\| \jx^{-2} \jn Lv_2(t) \bigr\|_{L^2_x} \lesssim \varepsilon^2.
 \end{equation}
 Then, using the energy estimate~\eqref{equ:energy_est_sq}, the auxiliary estimate~\eqref{equ:resonant_Lv3_local_decay}, and the local decay bounds from Proposition~\ref{prop:resonant_bootstrap_bounds}, we obtain 
 \begin{align*}
  \bigl\| \jn L v_2(t) \bigr\|_{L^2_x}^2 &\lesssim \int_1^t \bigl\| \jn L (\pt-i\jn)v_2(s) \, \overline{ \jn L v_2(s) } \bigr\|_{L^1_x} \, \ud s \\
  &\lesssim \int_1^t \bigl\| \jx^2 \jn L (\pt-i\jn)v_2(s) \bigr\|_{L^2_x} \bigl\| \jx^{-2} \jn L v_2(s) \bigr\|_{L^2_x} \, \ud s \\
  &\lesssim \int_1^t \jap{s} \bigl\| \jx^3 \alpha(x) \bigl( u(s)^2 - u(s,0)^2 \bigr) \bigr\|_{H^1_x} \bigl\| \jx^{-2} \jn L v_2(s) \bigr\|_{L^2_x} \\
  &\lesssim \int_1^t \jap{s} \bigl\| \jx^8 \alpha(x) \bigr\|_{H^2_x} \bigl\| \jx^{-2} \px v(s) \bigr\|_{L^2_x} \bigl\|\jx^{-2} v(s)\bigr\|_{H^1_x} \bigl\| \jx^{-2} \jn L v_2(s) \bigr\|_{L^2_x} \, \ud s \\
  &\lesssim \int_1^t \frac{\varepsilon^2}{\js} \, \ud s \\
  &\lesssim \varepsilon^2 \log(t).
 \end{align*}
 Hence, we arrive at the desired growth bound 
 \begin{equation} \label{equ:resonant_x_g3_slow_growth}
  \sup_{t \geq 0} \, \jt^{-(0+)} \| \jx g_2(t) \|_{H^2_x} \lesssim \varepsilon^2.
 \end{equation}
 It remains to prove~\eqref{equ:resonant_Lv3_local_decay}. Using the local decay estimates for the Klein-Gordon propagator from Lemma~\ref{lem:local_decay} and the local decay bounds for $v(t)$ from Proposition~\ref{prop:resonant_bootstrap_bounds}, we obtain from the Duhamel formula for $v_2(t)$ that
 \begin{align*}
  \bigl\| \jx^{-2} \jn L v_2(t) \bigr\|_{L^2_x} &\lesssim \bigl\| \jx^{-1} \jn^2 v_2(t) \bigr\|_{L^2_x} + t \bigl\| \jx^{-2} \jn \px v_2(t) \bigr\|_{L^2_x} \\
  &\lesssim \int_1^t \bigl\| \jx^{-1} e^{i(t-s)\jn} \jx^{-1} \bigr\|_{L^2_x \to L^2_x} \bigl\| \jx \alpha(x) \bigl( u(s)^2 - u(s,0)^2 \bigr) \bigr\|_{H^1_x} \, \ud s \\
  &\quad + t \int_1^t \Bigl\| \jx^{-2} \frac{\px}{\jn} e^{i(t-s)\jn} \jx^{-2} \Bigr\|_{L^2_x \to L^2_x}  \bigl\| \jx^2 \alpha(x) \bigl( u(s)^2 - u(s,0)^2 \bigr) \bigr\|_{H^1_x} \, \ud s \\
  &\lesssim \int_1^t \frac{1}{\jap{t-s}^\hf} \frac{\varepsilon^2}{\js^\thf} \, \ud s + t \int_1^t \frac{1}{\jap{t-s}^\thf} \frac{\varepsilon^2}{\jap{s}^\thf} \, \ud s \\
  &\lesssim \frac{\varepsilon^2}{\jt^\hf}.
 \end{align*}
 
 Finally, we remark that the growth bound~\eqref{equ:resonant_x_g3_slow_growth} can be derived for $g_5(t)$ in an analogous manner. It is at this point that the strong decay of the remainder term $|\tilde{r}(t)| \lesssim \varepsilon^2 t^{-(\thf-)}$ is needed.
 \end{proof}

 \medskip
 
 \noindent \underline{{\it Decay and asymptotics of $v_{mod}(t)$}:} 
 The modified scattering behavior of $v_{mod}(t)$ is only caused by the frequencies of the variable coefficient $\alpha(x)$ near $\xi = \pm \sqrt{3}$. We therefore decompose $\widehat{\alpha}(\xi)$ into
 \begin{equation*}
  \widehat{\alpha}(\xi) = \widehat{\alpha}_+(\xi) + \widehat{\alpha}_{-}(\xi) + \widehat{\alpha}_{nr}(\xi)
 \end{equation*}
 with
 \begin{equation*}
  \widehat{\alpha}_{\pm}(\xi) = \varphi(\xi \mp \sqrt{3}) \widehat{\alpha}(\xi).
 \end{equation*}
 Here, $\varphi(\xi)$ is a smooth bump function such that $\varphi(\xi) = 1$ in a small neighborhood around $\xi = 0$ and such that
 \begin{equation} \label{equ:definition_varphi}
  \varphi(\xi) = 0 \quad  \text{for } |\xi| \geq \tilde{\delta}
 \end{equation}
 for some small $\tilde{\delta} = \tilde{\delta}(\delta) > 0$, whose size will be specified further below. 
 Correspondingly, we define
 \begin{align}
  v_{mod, \pm}(t) &:= \frac{a_0^2}{2} \int_1^t \bigl( e^{i(t-s)\jn} \jn^{-1} \alpha_{\pm} \bigr) \frac{e^{2is}}{s} \, \ud s, \label{equ:resonant_def_v_modpm_earlier}  \\
  v_{mod, nr}(t) &:= \frac{a_0^2}{2} \int_1^t \bigl( e^{i(t-s)\jn} \jn^{-1} \alpha_{nr} \bigr) \frac{e^{2is}}{s} \, \ud s.
 \end{align}

 \medskip 
 
\noindent {\it Decay of $v_{mod, nr}(t)$}:
Since $\widehat{\alpha}_{nr}(\pm \sqrt{3}) = 0$, we can integrate by parts in time $s$ in the Duhamel integral for $v_{mod, nr}(t)$. Then using the standard dispersive decay estimate~\eqref{equ:dispersive_decay}, we obtain uniformly for all $t \geq 1$ that
\begin{align*}
 \|v_{mod,nr}(t)\|_{L^\infty_x} &\lesssim \frac{\varepsilon^2}{t^\hf} \bigl\| (2-\jn)^{-1} \jn \alpha_{nr} \bigr\|_{L^1_x} + \frac{\varepsilon^2}{t} \bigl\| (2-\jn)^{-1} \jn^{-1} \alpha_{nr} \bigr\|_{L^\infty_x} \\
 &\quad + \varepsilon^2 \int_1^t \frac{1}{\jap{t-s}^\hf} \bigl\| (2-\jn)^{-1} \jn \alpha_{nr} \bigr\|_{L^1_x} \frac{1}{s^2} \, \ud s \\
 &\lesssim_{\delta} \frac{\varepsilon^2}{t^\hf}.
\end{align*}

\noindent {\it Decay of $v_{mod, \pm}(t)$ away from a small conic neighborhood of $x = \pm \frac{\sqrt{3}}{2} t$}:
We can infer time decay of $v_{mod, \pm}(t)$ away from a small conic neighborhood of the rays $x = \pm \frac{\sqrt{3}}{2} t$ just by integrating by parts in the frequency variable. Indeed, writing $v_{mod,\pm}(t)$ as the (non-standard) double oscillatory integral
\begin{equation} \label{equ:resonant_def_v_modpm}
 v_{mod,\pm}(t,x) = \frac{a_0^2}{2 \sqrt{2\pi}} \int_1^t \int_\bbR e^{i(x \xi + (t-s) \jxi + 2s)} \jxi^{-1} \widehat{\alpha}_{\pm}(\xi) \frac{1}{s} \, \ud \xi \, \ud s,
\end{equation} 
we note that the phase 
\begin{equation*}
 \psi(s, \xi; t,x) := x \xi + (t-s) \jap{\xi} + 2s
\end{equation*}
satisfies
\begin{align*}
 \partial_\xi \psi(s, \xi; t, x) = x + (t-s) \frac{\xi}{\jap{\xi}}, \qquad \partial_\xi^2 \psi(s, \xi; t,x) = \frac{t-s}{\jap{\xi}^3}. 
\end{align*}
Then for any given small $\delta > 0$, we may choose the constant $\tilde{\delta} \equiv \tilde{\delta}(\delta) > 0$ in the definition~\eqref{equ:definition_varphi} of the cut-off funtion $\varphi$ above so small such that
\begin{align*}
 \Bigl| \frac{\xi}{\jap{\xi}} - \Bigl( \pm \frac{\sqrt{3}}{2} \Bigr) \Bigr| \leq \frac{\delta}{2} \quad \text{ whenever } \quad \widehat{\alpha}_{\pm}(\xi) = \varphi(\xi \mp \sqrt{3}) \widehat{\alpha}(\xi) \neq 0.
\end{align*}
Moreover, we have $|\partial_\xi^2 \psi(s,\xi; t,x)| \simeq (t-s)$ on the support of $\widehat{\alpha}_\pm(\xi)$.
Now we distinguish two cases. If $|x| \geq \bigl( \frac{\sqrt{3}}{2} + \delta \bigr) t$, then on the support of $\widehat{\alpha}_{\pm}(\xi)$ the phase satisfies 
\begin{align*}
 |\partial_\xi \psi| &\geq |x| - (t-s) \frac{|\xi|}{\jap{\xi}} \geq \Bigl( \frac{\sqrt{3}}{2} + \delta \Bigr) t - \Bigl( \frac{\sqrt{3}}{2} + \frac{\delta}{2} \Bigr) (t-s) \geq \frac{\delta}{2} t.
\end{align*}
Integrating by parts $N$ times in $\xi$, we obtain that 
\begin{align*}
 \bigl| v_{mod, \pm}(t,x) \bigr| \lesssim_{\delta, N} \int_1^t \frac{1}{t^N} \frac{\varepsilon^2}{s} \, \ud s \lesssim_{\delta, N} \frac{\varepsilon^2}{t^{N-}}.
\end{align*}
Instead, if $0 \leq |x| \leq \bigl( \frac{\sqrt{3}}{2} - \delta \bigr) t$, we divide the time integration interval into two subintervals $$[1, t] = [1, t_1] \cup [t_1, t],$$ where
\begin{align*}
 t_1 :=  \frac{\delta}{2 (\sqrt{3} + \delta)} t. 
\end{align*}
On the support of $\widehat{\alpha}_{\pm}(\xi)$ the phase satisfies for $1 \leq s \leq t_1$ that
\begin{align*}
 |\partial_\xi \psi| &\geq t \frac{|\xi|}{\jap{\xi}} - |x| - s \frac{|\xi|}{\jap{\xi}} \geq t \Bigl( \frac{\sqrt{3}}{2} - \frac{\delta}{2} \Bigr) - t \Bigl( \frac{\sqrt{3}}{2} - \delta \Bigr) - \frac{\delta}{2 (\sqrt{3} + \delta)} t \Bigl( \frac{\sqrt{3}}{2} + \frac{\delta}{2} \Bigr) = \frac{\delta}{4} t
\end{align*}
and integration by parts in $\xi$ pays off. When $s \geq t_1$ we can just use the usual $\jap{t-s}^{-\frac{1}{2}}$ dispersive decay~\eqref{equ:dispersive_decay} of $e^{i(t-s)\jap{\nabla}}$ and crudely bound $\frac{1}{s} \leq \frac{1}{t_1} \lesssim_\delta \frac{1}{t}$. Overall, we obtain in the case $0 \leq |x| \leq \bigl( \frac{\sqrt{3}}{2} - \delta \bigr) t$ that 
\begin{equation} 
 \begin{aligned}
  |v_{mod,\pm}(t,x)| &\lesssim_{\delta} \int_1^{t_1} \frac{1}{t} \frac{\varepsilon^2}{s} \, \ud s + \int_{t_1}^t \frac{1}{\jap{t-s}^{\frac{1}{2}}}  \frac{\varepsilon^2}{t} \, \ud s \lesssim_{\delta} \frac{\varepsilon^2}{t^{1-}} + \frac{\varepsilon^2}{t^\hf} \lesssim_{\delta} \frac{\varepsilon^2}{t^\hf}.
 \end{aligned}
\end{equation}

\medskip 

\noindent {\it Asymptotics of $v_{mod,\pm}(t,x)$ along the rays $x = \pm \frac{\sqrt{3}}{2} t$}: We provide the details for $v_{mod,-}(t,x)$ noting that the treatment of $v_{mod,+}(t,x)$ proceeds analogously. First, we may restrict the time integration in the definition~\eqref{equ:resonant_def_v_modpm_earlier} of $v_{mod,-}(t,x)$ to times $1 \leq s \leq t-1$ at the expense of picking up a remainder term of order $\calO\bigl( \varepsilon^2 t^{-1} \bigr)$. Moreover, by Lemma~\ref{lem:KG_propagator_asymptotics} on the asymptotics of the Klein-Gordon propagator we have for $1 \leq s \leq t-1$ that 
\begin{equation} \label{equ:resonant_asymptotics_retarded_free_along_ray}
 \begin{aligned}
  &\Bigl( e^{i(t-s)\jap{\nabla}} \jn^{-1} \alpha_{-} \Bigr)\Bigl( \pm \frac{\sqrt{3}}{2} t \Bigr) \\
  &\quad \quad = \frac{1}{\rho(t-s, \pm \frac{\sqrt{3}}{2} t )^{\frac{1}{2}}} e^{i \frac{\pi}{4}} e^{i\rho(t-s,\pm \frac{\sqrt{3}}{2} t )} \widehat{\alpha}_-\biggl( - \frac{\pm \frac{\sqrt{3}}{2} t}{\rho(t-s,\pm \frac{\sqrt{3}}{2} t )} \biggr) \theta \biggl( \frac{\pm \frac{\sqrt{3}}{2} t}{t-s} \biggr) + \frac{1}{(t-s)^{\frac{5}{8}}} \calO \bigl( \| \jx \alpha_{-} \|_{H^2_x} \bigr),
 \end{aligned}
\end{equation}
where $\theta(z) = 1$ for $|z| < 1$, $\theta(z) = 0$ for $|z| \geq 1$, and 
\begin{align*}
 \rho(t-s, {\textstyle \pm \frac{\sqrt{3}}{2} t}) &= \bigl( (t-s)^2 - {\textstyle \frac{3}{4}} t^2 \bigr)^{\frac{1}{2}} = {\textstyle \frac{t}{2}} \bigl( 1 - 8 {\textstyle \frac{s}{t}} + 4  {\textstyle \frac{s^2}{t^2}} \bigr)^{\frac{1}{2}}. 
\end{align*}
Inserting the asymptotics~\eqref{equ:resonant_asymptotics_retarded_free_along_ray} into~\eqref{equ:resonant_def_v_modpm_earlier} gives
\begin{equation*}
 v_{mod, -}\Bigl(t, \pm \frac{\sqrt{3}}{2} t \Bigr) = \frac{a_0^2}{2} \int_1^{t-\frac{\sqrt{3}}{2} t} \frac{1}{\rho(t-s, \pm \frac{\sqrt{3}}{2} t)^{\frac{1}{2}}} e^{i \frac{\pi}{4}} e^{i\rho(t-s, \pm \frac{\sqrt{3}}{2} t)} \widehat{\alpha}_-\biggl( - \frac{\pm \frac{\sqrt{3}}{2} t}{\rho(t-s, \pm \frac{\sqrt{3}}{2} t)} \biggr) \frac{e^{2is}}{s} \, \ud s + \calO \Bigl( \frac{\varepsilon^2}{t^{\frac{5}{8}-}} \Bigr).
\end{equation*}
Since $\widehat{\alpha}_-(\xi) = 0$ for $\xi > 0$, we have along the ray $x = -\frac{\sqrt{3}}{2} t$ that
\begin{equation*}
 v_{mod, -}\Bigl(t, - \frac{\sqrt{3}}{2} t \Bigr) = \calO \Bigl( \frac{\varepsilon^2}{t^{\frac{5}{8}-}} \Bigr).
\end{equation*}
Moreover, due to the sharp localization of the frequency support of $\widehat{\alpha}_-(\xi)$ around $\xi = -\sqrt{3}$, for $t \gg 1$ the time integration in the last identity for $v_{mod,-}(t, +\frac{\sqrt{3}}{2} t)$ is in fact only over an interval $1 \leq s \leq ct$ for some small constant $0 < c \ll 1$. Thus, along the ray $x = + \frac{\sqrt{3}}{2} t$ it holds that
\begin{equation} \label{equ:resonant_v_mod_minus_along_plus_ray}
 v_{mod,-}\Bigl(t, + \frac{\sqrt{3}}{2} t \Bigr) = \frac{a_0^2}{2} \int_1^{ct} \frac{1}{\rho(t-s, \frac{\sqrt{3}}{2} t)^{\frac{1}{2}}} e^{i \frac{\pi}{4}} e^{i\rho(t-s, \frac{\sqrt{3}}{2} t)}  \widehat{\alpha}_-\biggl( - \frac{\frac{\sqrt{3}}{2} t}{\rho(t-s, \frac{\sqrt{3}}{2} t)} \biggr) \frac{e^{2is}}{s} \, \ud s + \calO \Bigl( \frac{\varepsilon^2}{t^{\frac{5}{8}-}} \Bigr).
\end{equation}
In view of the approximate identities
\begin{align*}
 - \frac{ \frac{\sqrt{3}}{2} t }{ \rho(t-s,  \frac{\sqrt{3}}{2} t) } &= -\sqrt{3} + \calO \Bigl( \frac{s}{t} \Bigr), \qquad \frac{1}{\rho(t-s, \frac{\sqrt{3}}{2} t)^{\frac{1}{2}}} = \frac{\sqrt{2}}{t^{\frac{1}{2}}} + \calO \Bigl( \frac{s}{t^{\frac{3}{2}}} \Bigr), 
\end{align*}
it follows that 
\begin{equation} \label{equ:resonant_v_mod_minus_along_plus_ray_almost_done}
 v_{mod,-}\Bigl(t, + \frac{\sqrt{3}}{2} t \Bigr) = \frac{a_0^2}{\sqrt{2}} e^{i\frac{\pi}{4}} \widehat{\alpha}(-\sqrt{3}) \frac{1}{t^{\frac{1}{2}}} \int_1^{ct} e^{i(\rho(t-s,\frac{\sqrt{3}}{2} t) + 2s)} \frac{1}{s} \, \ud s + \calO \Bigl( \frac{\varepsilon^2}{t^\hf} \Bigr).
\end{equation}
At this point we observe that the phase 
\begin{equation*}
 \phi(s;t) := \rho \Bigl(t-s,\frac{\sqrt{3}}{2} t\Bigr) + 2s
\end{equation*}
is stationary at $s=0$ and that its Taylor expansion about $s=0$ is of the form
\begin{equation*}
 \phi(s;t) = \frac{t}{2} + \calO \Bigl( \frac{s^2}{t} \Bigr).
\end{equation*}
Thus, for $1 \leq s \ll t^\hf$ the phase $\phi(s;t)$ is essentially constant and the integrand in~\eqref{equ:resonant_v_mod_minus_along_plus_ray_almost_done} is effectively monotone, which causes the buildup of a $\log(t)$ factor. In order to arrive at a sharp formula for the asymptotics, we split the time integration interval into the two subintervals $1 \leq s \leq 10^{-3} t^{\frac{1}{2}}$ and $10^{-3} t^{\frac{1}{2}} \leq s \leq c t$. 
For the interval $1 \leq s \leq 10^{-3} t^\hf$ we compute that
\begin{align*}
 \int_1^{10^{-3} t^{\frac{1}{2}}} e^{i \phi(s;t)} \frac{1}{s} \, \ud s &= e^{i \frac{t}{2}} \int_1^{10^{-3} t^{\frac{1}{2}}} \frac{1}{s} \, \ud s + \int_1^{10^{-3} t^{\frac{1}{2}}} \calO \Bigl( \frac{s}{t} \Bigr) \, \ud s = \frac{e^{i \frac{t}{2}}}{2} \log(t) + \calO(1).
\end{align*}
Instead, on the interval $10^{-3} t^{\frac{1}{2}} \leq s \leq c t$ we integrate by parts. Using that there $\partial_s \phi(s;t) = \calO \bigl( \frac{s}{t} \bigr)$ and $\partial_s^2 \phi(s;t) = \calO \bigl( \frac{1}{t} \bigr)$, we find 
\begin{equation*}
 \biggl| \int_{10^{-3} t^{\frac{1}{2}}}^{c t} e^{i \phi(s;t)} \frac{1}{s} \, \ud s \biggr| \lesssim \int_{10^{-3} t^{\frac{1}{2}}}^{ct} \frac{t}{s^3} \, \ud s + \biggl| \frac{t}{s^2} \Bigr|_{s=10^{-3} t^\hf}^{s=ct} \biggr| \lesssim 1.
\end{equation*}
Putting things together, we obtain the asymptotics
\begin{equation*}
 v_{mod,-}\Bigl(t, + \frac{\sqrt{3}}{2} t \Bigr) = \frac{a_0^2}{\sqrt{8}}  e^{i\frac{\pi}{4}} e^{i \frac{t}{2}} \widehat{\alpha}(-\sqrt{3}) \frac{\log(t)}{t^{\frac{1}{2}}} + \calO \Bigl( \frac{\varepsilon^2}{t^\hf} \Bigr), \quad t \gg 1.
\end{equation*}
This finishes the proof of Theorem~\ref{thm:resonant}.

\section{Non-Resonant Case}

This section is devoted to the proof of Theorem~\ref{thm:nonresonant}, which establishes sharp decay estimates and asymptotics for small global solutions $v(t)$ to 
\begin{equation} \label{equ:nlkg_nonresonant_sec}
 (\pt - i \jn) v = \frac{1}{2i} \jn^{-1} \bigl( \alpha(\cdot) u^2 + \beta_0 u^3 + \beta(\cdot) u^3 \bigr) \text{ on } \bbR^{1+1}
\end{equation}
in the {\bf non-resonant case}
\begin{equation} \label{equ:nonresonance_assumption_sec}
 \widehat{\alpha}(+\sqrt{3}) = 0 \quad \text{ and } \quad \widehat{\alpha}(-\sqrt{3}) = 0.
\end{equation}
Global existence of small regular solutions to~\eqref{equ:nlkg_nonresonant_sec} is well-known, see for instance~\cite{H97}. Our goal is therefore to derive global-in-time a priori bounds for small solutions to~\eqref{equ:nlkg_nonresonant_sec} that yield sharp decay estimates and asymptotics. 
By time-reversal symmetry it suffices to only consider positive times. 
The main new difficulty here is to deal with the variable coefficient quadratic nonlinearity on the right-hand side of~\eqref{equ:nlkg_nonresonant_sec}. As explained in Subsection~\ref{subsec:intro_proof_ideas}, the potentially problematic contributions can only come from the part $\alpha(x) u(t,0)^2$ of the variable coefficient quadratic nonlinearity. To isolate that component, we rewrite~\eqref{equ:nlkg_nonresonant_sec} as
\begin{equation} \label{equ:nonresonant_first_order_kg_rewritten}
 \begin{aligned}
  (\pt - i\jn) v &= \frac{1}{2i} (\jn^{-1} \alpha)(\cdot) u(t,0)^2 + \frac{1}{2i} \jn^{-1} \Bigl( \alpha(\cdot) \bigl( u(t, \cdot)^2 - u(t,0)^2 \bigr) \Bigr) \\
  &\qquad + \frac{\beta_0}{2i} \jn^{-1} \bigl( u^3 \bigr) + \frac{1}{2i} \jn^{-1} \bigl( \beta(\cdot) u^3 \bigr).
 \end{aligned}
\end{equation}
Relying on the non-resonance assumption~\eqref{equ:nonresonance_assumption_sec} we introduce a novel variable coefficient quadratic normal form to transform this problematic component into more favorable terms that turn out to behave like ``variable coefficient cubic nonlinearities''. Then we can control their contributions using local decay estimates, following the strategy from our previous work~\cite{LLS19}. 
The constant coefficient cubic term on the right-hand side of~\eqref{equ:nlkg_nonresonant_sec} causes a logarithmic phase correction in the asymptotics of the solution $v(t)$ to~\eqref{equ:nlkg_nonresonant_sec}, which we capture by deriving an ODE for the profile $f(t) := e^{-it\jn} v(t)$ in the spirit of the space-time resonances method~\cite{GMS12_Ann, GMS12_JMPA, GMS09, GNT09}. 

The main part of the proof of Theorem~\ref{thm:nonresonant} consists in closing a bootstrap argument for sufficiently small data for the quantity
\begin{equation} \label{equ:nonresonant_NT_def}
\begin{aligned}
 N(T) &:= \sup_{0 \leq t \leq T} \, \biggl\{ \jt^{\frac{1}{2}} \|v(t)\|_{L^\infty_x} + \jt^{-\delta} \bigl\| \jn^2 v(t) \bigr\|_{L^2_x} + \jt^{-\delta} \bigl\| \jn L v(t) \bigr\|_{L^2_x} \\
 &\qquad \qquad \qquad \qquad \qquad \qquad \qquad \qquad \qquad + \jt^{-1-\delta} \|x v(t)\|_{L^2_x} + \bigl\| \jap{\xi}^{\frac{3}{2}} \hat{f}(t,\xi) \bigr\|_{L^\infty_\xi} \biggr\},
\end{aligned}
\end{equation}
where $T > 0$ and $0 < \delta \ll 1$ is a sufficiently small absolute constant.
Throughout this section we can therefore work under the assumption that $N(T) \leq 1$, which simplifies the bookkeeping for some of the nonlinear estimates.

\subsection{Normal form transformation}

Our first task is to transform the problematic first term on the right-hand side of~\eqref{equ:nonresonant_first_order_kg_rewritten} into a better form. To this end we consider the equation satisfied by the Fourier transform of the profile $\hat{f}(t,\xi)$. From~\eqref{equ:nonresonant_first_order_kg_rewritten} it follows that 
\begin{equation} \label{equ:pt_ft_profile_rewritten}
 \begin{aligned}
  \pt \hat{f}(t,\xi) &= \frac{1}{2i} e^{-it\jxi} \jxi^{-1} \widehat{\alpha}(\xi) u(t,0)^2 + \frac{1}{2i} e^{-it\jxi} \jxi^{-1} \calF\Bigl[ \alpha(\cdot) \bigl( u(t, \cdot)^2 - u(t,0)^2 \bigr) \Bigr](\xi) \\
  &\qquad + \frac{\beta_0}{2i} e^{-it\jxi} \jxi^{-1} \calF\bigl[ u(t)^3 \bigr](\xi) + \frac{1}{2i} e^{-it\jxi} \jxi^{-1} \calF\bigl[ \beta(\cdot) u(t)^3 \bigr](\xi).
 \end{aligned}
\end{equation}
Then we insert the decomposition of $u(t,0)$ into its ``phase-filtered components'' 
\begin{equation*}
 u(t,0) = v(t,0) + \bar{v}(t,0) = e^{+it} \bigl( e^{-it} v(t,0) \bigr) + e^{-it} \bigl( e^{+it} \bar{v}(t,0) \bigr),
\end{equation*}
into the first term on the right-hand side of~\eqref{equ:pt_ft_profile_rewritten} to find that
\begin{equation} \label{equ:nonresonant_phase_filtered_inserted}
\begin{aligned}
 \frac{1}{2i} e^{-it\jxi} \jxi^{-1} \widehat{\alpha}(\xi) u(t,0)^2 &= \frac{1}{2i} e^{it(2-\jxi)} \jxi^{-1} \widehat{\alpha}(\xi) \bigl( e^{-it} v(t,0) \bigr)^2 \\
 &\quad + \frac{1}{i} e^{-it\jxi} \jxi^{-1} \widehat{\alpha}(\xi) \bigl( e^{-it} v(t,0) \bigr) \bigl( e^{+it} \bar{v}(t,0) \bigr) \\
 &\quad + \frac{1}{2i} e^{-it(2+\jxi)} \jxi^{-1} \widehat{\alpha}(\xi) \bigl( e^{+it} \bar{v}(t,0) \bigr)^2.
\end{aligned}
\end{equation}
Now we observe that thanks to the non-resonance assumption $\widehat{\alpha}(\pm \sqrt{3}) = 0$, the symbol $(2-\jxi)^{-1} \widehat{\alpha}(\xi)$ does not have a singularity, although $2-\jxi$ vanishes for $\xi = \pm \sqrt{3}$. We can therefore pull out a time derivative and rewrite~\eqref{equ:nonresonant_phase_filtered_inserted} as
\begin{equation} \label{equ:nonresonant_pull_out_pt}
\begin{aligned}
 \frac{1}{2i} e^{-it\jxi} \jxi^{-1} \widehat{\alpha}(\xi) u(t,0)^2 &= \pt \biggl( - \hf e^{-it\jxi} (2-\jxi)^{-1} \jxi^{-1} \widehat{\alpha}(\xi) e^{2it} \bigl( e^{-it} v(t,0) \bigr)^2 \biggr) \\
 &\quad + e^{-it\jxi} (2-\jxi)^{-1} \jxi^{-1} \widehat{\alpha}(\xi) e^{2it} \pt \bigl( e^{-it} v(t,0) \bigr) \bigl( e^{-it} v(t,0) \bigr) \\
 &\quad + \pt \biggl( e^{-it\jxi} \jxi^{-2} \widehat{\alpha}(\xi)  \bigl( e^{-it} v(t,0) \bigr) \bigl( e^{+it} \bar{v}(t,0) \bigr) \biggr) \\
 &\quad - 2 e^{-it\jxi} \jxi^{-2} \widehat{\alpha}(\xi) \, \Re \Bigl( \pt \bigl( e^{-it} v(t,0) \bigr) \bigl( e^{+it} \bar{v}(t,0) \bigr) \Bigr) \\
 &\quad + \pt \biggl( \hf e^{-it\jxi} (2+\jxi)^{-1} \jxi^{-1} \widehat{\alpha}(\xi) e^{-2it} \bigl( e^{+it} \bar{v}(t,0) \bigr)^2 \biggr) \\
 &\quad - e^{-it\jxi} (2+\jxi)^{-1} \jxi^{-1} \widehat{\alpha}(\xi) e^{-2it} \pt \bigl( e^{+it} \bar{v}(t,0) \bigr) \bigl( e^{+it} \bar{v}(t,0) \bigr). 
\end{aligned}
\end{equation}
Next, we introduce short-hand notations for the smooth and decaying coefficients that emerge on the right-hand side of~\eqref{equ:nonresonant_pull_out_pt}
\begin{equation} \label{equ:nonresonant_defs_alpha_coeff}
\begin{aligned}
 \alpha_1(x) &:= \frac{1}{2} \calF^{-1}\bigl[(2-\jxi)^{-1} \jxi^{-1} \widehat{\alpha}\bigr](x), \\
 \alpha_2(x) &:= - \calF^{-1}\bigl[ \jxi^{-2} \widehat{\alpha} \bigr](x), \\
 \alpha_3(x) &:= - \frac{1}{2} \calF^{-1}\bigl[ (2+\jxi)^{-1} \jxi^{-1} \widehat{\alpha}\bigr](x).
\end{aligned}
\end{equation}
Inserting~\eqref{equ:nonresonant_pull_out_pt} and~\eqref{equ:nonresonant_defs_alpha_coeff} back into~\eqref{equ:pt_ft_profile_rewritten} we arrive at 
\begin{equation} \label{equ:pt_ft_profile_recast}
 \begin{aligned}
 &\pt \Bigl( \hat{f}(t, \xi) + e^{-it\jap{\xi}} \bigl( \widehat{\alpha}_1(\xi) v(t,0)^2 + \widehat{\alpha}_2(\xi) |v(t,0)|^2 + \widehat{\alpha}_3(\xi) \bar{v}(t,0)^2 \bigr) \Bigr) \\
 &= 2 e^{-it\jap{\xi}} \widehat{\alpha}_1(\xi) e^{2it} \partial_t \bigl( e^{-it} v(t,0) \bigr) \bigl( e^{-it} v(t,0) \bigr) \\
 &\quad + 2 e^{-it\jap{\xi}} \widehat{\alpha}_2(\xi) \, \Re \Bigl( \partial_t \bigl( e^{-it} v(t,0) \bigr) \bigl( e^{+it} \bar{v}(t,0) \bigr) \Bigr) \\
 &\quad + 2 e^{-it\jap{\xi}} \widehat{\alpha}_3(\xi) e^{-2it} \partial_t \bigl( e^{+it} \bar{v}(t,0) \bigr) \bigl( e^{+it} \bar{v}(t,0) \bigr) \\ 
 &\quad + \frac{1}{2i} e^{-it\jxi} \jxi^{-1} \calF\Bigl[ \alpha(\cdot) \bigl( u(t, \cdot)^2 - u(t,0)^2 \bigr) \Bigr](\xi) \\
 &\quad + \frac{\beta_0}{2i} e^{-it\jxi} \jxi^{-1} \calF\bigl[ u(t)^3 \bigr](\xi) + \frac{1}{2i} e^{-it\jxi} \jxi^{-1} \calF\bigl[ \beta(\cdot) u(t)^3 \bigr](\xi).
 \end{aligned}
\end{equation}
Hence, upon introducing the variable coefficient quadratic normal form 
\begin{equation*}
 \calQ := \alpha_1(x) v(t,0)^2 + \alpha_2(x) |v(t,0)|^2 + \alpha_3(x) \bar{v}(t,0)^2,
\end{equation*}
it follows that the equation~\eqref{equ:nlkg_nonresonant_sec} for $v(t)$ takes on the form
\begin{equation} \label{equ:nlkg_nonresonant_normal_form_subt}
 \begin{aligned}
  (\pt - i \jn)(v + \calQ) &= 2 \alpha_1(x) e^{2it} \partial_t \bigl( e^{-it} v(t,0) \bigr) \bigl( e^{-it} v(t,0) \bigr) + 2 \alpha_2(x) \, \Re \Bigl(\pt \bigl( e^{-it} v(t,0) \bigr) \bigl( e^{+it} \bar{v}(t,0) \bigr) \Bigr) \\
 &\quad  + 2 \alpha_3(x) e^{-2it} \pt \bigl( e^{+it} \bar{v}(t,0) \bigr) \bigl( e^{+it} \bar{v}(t,0) \bigr) + \frac{1}{2i} \jn^{-1} \Bigl( \alpha(\cdot) \bigl( u(t, \cdot)^2 - u(t,0)^2 \bigr) \Bigr) \\
 &\quad + \frac{1}{2i} \jn^{-1} \bigl( \beta(\cdot) u^3 \bigr) + \frac{\beta_0}{2i} \jap{\nabla}^{-1} \bigl( u^3 \bigr).
 \end{aligned}
\end{equation}
Integrating~\eqref{equ:nlkg_nonresonant_normal_form_subt} in time, we also obtain that Duhamel's formula for $v(t)$ can be rewritten as
\begin{equation} \label{equ:duh_v_rearr}
\begin{aligned}
 v(t) &= e^{it\jap{\nabla}} \Bigl( v_0 + \alpha_1 v(0,0)^2 + \alpha_2 |v(0,0)|^2 + \alpha_3 \bar{v}(0,0)^2 \Bigr) \\
 &\qquad - \alpha_1(x) v(t,0)^2 - \alpha_2(x) |v(t,0)|^2 - \alpha_3(x) \bar{v}(t,0)^2 \\
 &\qquad + 2 \int_0^t \bigl( e^{i (t-s)\jn} \alpha_1 \bigr) e^{2is} \partial_s \bigl( e^{-is} v(s,0) \bigr) \bigl( e^{-is} v(s,0) \bigr) \, \ud s \\
 &\qquad  + 2 \int_0^t \bigl( e^{i(t-s)\jap{\nabla}} \alpha_2 \bigr) \, \Re \Bigl( \partial_s \bigl( e^{-is} v(s,0) \bigr) \bigl( e^{+is} \bar{v}(s,0) \bigr) \Bigr) \, \ud s \\
 &\qquad  + 2 \int_0^t \bigl( e^{i(t-s)\jap{\nabla}} \alpha_3 \bigr) e^{-2is} \partial_s \bigl( e^{+is} \bar{v}(s,0) \bigr) \bigl( e^{+is} \bar{v}(s,0) \bigr) \, \ud s \\
 &\qquad + \frac{1}{2i} \int_0^t e^{+i(t-s)\jn} \jn^{-1} \Bigl( \alpha(\cdot) \bigl( u(s, \cdot)^2 - u(s,0)^2 \bigr) \Bigr) \, \ud s \\
 &\qquad + \frac{1}{2i} \int_0^t e^{i(t-s)\jap{\nabla}} \jap{\nabla}^{-1} \bigl( \beta(\cdot) u(s, \cdot)^3 \bigr) \, \ud s \\  
 &\qquad + \frac{\beta_0}{2i} \int_0^t e^{i(t-s)\jap{\nabla}} \jap{\nabla}^{-1} \bigl( u(s, \cdot)^3 \bigr) \, \ud s. 
\end{aligned}
\end{equation}
The equation~\eqref{equ:duh_v_rearr} will be convenient for deriving energy estimates in the next subsection.
For ease of notation in the remainder of this section, we denote the constant coefficient cubic term in Duhamel's formula~\eqref{equ:duh_v_rearr} for $v(t)$ by
\begin{align*}
 \calC(t) &:= \frac{\beta_0}{2i} \int_0^t e^{i(t-s)\jap{\nabla}} \jap{\nabla}^{-1} \bigl( u(s, \cdot)^3 \bigr) \, \ud s.
\end{align*}
Moreover, we introduce the following short-hand notations for the variable coefficient nonlinear terms in~\eqref{equ:duh_v_rearr}, which should all be thought of as ``variable coefficient cubic terms'',
\begin{align*}
 \calV_1(t) &:= 2 \int_0^t \bigl( e^{i (t-s)\jn} \alpha_1 \bigr) e^{2is} \partial_s \bigl( e^{-is} v(s,0) \bigr) \bigl( e^{-is} v(s,0) \bigr) \, \ud s, \\
 \calV_2(t) &:= 2 \int_0^t \bigl( e^{i(t-s)\jap{\nabla}} \alpha_2 \bigr) \, \Re \Bigl( \partial_s \bigl( e^{-is} v(s,0) \bigr) \bigl( e^{+is} \bar{v}(s,0) \bigr) \Bigr) \, \ud s, \\
 \calV_3(t) &:= 2 \int_0^t \bigl( e^{i(t-s)\jap{\nabla}} \alpha_3 \bigr) e^{-2is} \partial_s \bigl( e^{+is} \bar{v}(s,0) \bigr) \bigl( e^{+is} \bar{v}(s,0) \bigr) \, \ud s, \\
 \calV_4(t) &:= \frac{1}{2i} \int_0^t e^{+i(t-s)\jn} \jn^{-1} \Bigl( \alpha(\cdot) \bigl( u(s, \cdot)^2 - u(s,0)^2 \bigr) \Bigr) \, \ud s, \\ 
 \calV_5(t) &:= \frac{1}{2i} \int_0^t e^{i(t-s)\jap{\nabla}} \jap{\nabla}^{-1} \bigl( \beta(\cdot) u(s, \cdot)^3 \bigr) \, \ud s. 
\end{align*}

\subsection{Energy estimates} \label{subsec:nonresonant_energy_estimates}

In this subsection we establish the main energy estimates for the proof of Theorem~\ref{thm:nonresonant}. Without explicitly mentioning this, we will frequently use the commutator identities~\eqref{equ:commutator_identities}.

\subsubsection{Preparatory lemmas}

We begin with a crucial improved decay estimate for $\pt( e^{-it} v(t,0) )$ thanks to which the nonlinear terms $\calV_1(t), \calV_2(t), \calV_3(t)$ produced by the variable coefficient quadratic normal form~$\calQ$ can be considered as ``variable coefficient cubic terms''. 
\begin{lemma}[Key improved decay] \label{lem:key_improved_decay}
 Let $v(t)$ be the solution to~\eqref{equ:first_order_kg}. Then we have uniformly for all $0 \leq t \leq T$ that
 \begin{equation}
  \bigl| \pt \bigl( e^{-it} v(t,0) \bigr) \bigr| \lesssim \frac{N(T)}{\jt^{1-\delta}}.
 \end{equation}
\end{lemma}
\begin{proof}
First, we compute that
\begin{align*}
 \pt \bigl( e^{-it} v(t,0) \bigr) &= -i e^{-it} v(t,0) + e^{-it} (\pt v)(t,0) \\
 &= i e^{-it} \bigl( (-1 + \jn) v \bigr)(t,0) + e^{-it} \bigl( (\pt - i \jn) v \bigr)(t,0).
\end{align*}
Correspondingly, we have 
\begin{equation} \label{equ:key_improved_decay_starting_bound}
 \bigl| \pt \bigl( e^{-it} v(t,0) \bigr) \bigr| \leq \bigl| \bigl( (-1 + \jn) v \bigr)(t,0) \bigr| + \bigl| \bigl( (\pt - i \jn) v \bigr)(t,0) \bigr|.
\end{equation}
The first term on the right-hand side of~\eqref{equ:key_improved_decay_starting_bound} can be written on the Fourier side as
\begin{align*}
 \bigl( (-1 + \jn) v \bigr)(t,0) &= \frac{1}{\sqrt{2\pi}} \int_{\bbR} (-1+\jap{\xi}) \hat{v}(t, \xi) \, \ud \xi = \frac{1}{\sqrt{2\pi}} \int_{\bbR} (-1+\jap{\xi}) e^{it\jap{\xi}} e^{-it\jap{\xi}} \hat{v}(t, \xi) \, \ud \xi.
\end{align*}
Using that $(-1+\jap{\xi}) = \calO(\xi^2)$ for $|\xi| \ll 1$, we can integrate by parts in $\xi$ to find for any $t > 0$ that
\begin{align*}
 &\sqrt{2\pi} \bigl( (-1+\jn) v \bigr)(t,0) \\
 &\quad = \int_{\bbR} (-1+\jap{\xi}) \frac{1}{it} \frac{\jap{\xi}}{\xi} \partial_\xi \bigl( e^{it\jap{\xi}} \bigr) e^{-it\jap{\xi}} \hat{v}(t, \xi) \, \ud \xi \\
 &\quad = - \frac{1}{it} \int_{\bbR} e^{it\jap{\xi}}  \biggl( \underbrace{\partial_\xi \biggl( \frac{(-1+\jap{\xi}) \jxi}{\xi} \biggr)}_{\calO(1)} e^{-it\jap{\xi}} \hat{v}(t, \xi) + \underbrace{\biggl( \frac{(-1+\jxi)}{\xi \jxi} \biggr)}_{\calO(\jxi^{-1})} \jxi^2 \partial_\xi \bigl( e^{-it\jap{\xi}} \hat{v}(t, \xi) \bigr) \biggr) \, \ud \xi.
\end{align*}
Thus, upon recalling that $e^{-it\jn} (Lv)(t) = \calF^{-1} [ \jxi i \partial_\xi ( e^{-it\jxi} \hat{v}(t,\xi) ) ]$, we obtain by Cauchy-Schwarz for any time $0 < t \leq T$ that
\begin{align*}
 \bigl|\bigl( (-1 + \jn) v \bigr)(t,0)\bigr| &\lesssim \frac{1}{t} \Bigl( \bigl\|\jxi^{-1}\bigr\|_{L^2_\xi} \bigl\| \jxi \hat{v}(t, \xi) \bigr\|_{L^2_\xi} + \bigl\|\jxi^{-1}\bigr\|_{L^2_\xi} \bigl\| \jxi^2 \partial_\xi \bigl( e^{-it\jap{\xi}} \hat{v}(t, \xi) \bigr) \bigr\|_{L^2_\xi} \Bigr) \\
 &\lesssim \frac{1}{t} \Bigl( \bigl\| \jn v(t) \bigr\|_{L^2_x} + \bigl\| \jn L v(t) \bigr\|_{L^2_x} \Bigr) \\
 &\lesssim \frac{1}{t} N(T) \jt^{+\delta}.
\end{align*}
For short times $0 \leq t \leq 1$ we have $\bigl|\bigl( (-1 + \jn) v \bigr)(t,0)\bigr| \lesssim \|v(t)\|_{H^2_x}$ just by Sobolev embedding so that overall we get uniformly for all times $0 \leq t \leq T$ that 
\begin{align*}
 \bigl|\bigl( (-1 + \jn) v \bigr)(t,0)\bigr| \lesssim \frac{N(T)}{\jt^{1-\delta}}.
\end{align*}
For the second term on the right-hand side of~\eqref{equ:key_improved_decay_starting_bound} we easily obtain from the equation~\eqref{equ:first_order_kg} for $v(t)$ that
\begin{align*}
 \bigl| \bigl( (\pt - i \jn) v \bigr)(t,0) \bigr| &\lesssim \bigl\|  \jap{\nabla}^{-1} \bigl( \alpha(\cdot) u(t)^2 \bigr) \bigr\|_{L^\infty_x} + \bigl\| \jap{\nabla}^{-1} \bigl( u^3 \bigr) \bigr\|_{L^\infty_x} + \bigl\| \jap{\nabla}^{-1} \bigl( \beta(\cdot) u(t)^3 \bigr) \bigr\|_{L^\infty_x} \\
 &\lesssim \|u(t)\|_{L^\infty_x}^2 + \|u(t)\|_{L^\infty_x}^3 \\
 &\lesssim \frac{N(T)^2}{\jt}.
\end{align*}
This finishes the proof of Lemma~\ref{lem:key_improved_decay}.
\end{proof}

In the next lemma we establish some auxiliary energy estimates that are needed for the proof of one of the main energy estimates in Proposition~\ref{prop:slow_growth_H1L} further below.

\begin{lemma}[Auxiliary bounds] \label{lem:auxiliary_bounds}
Let $v(t)$ be the solution to~\eqref{equ:first_order_kg}. Then we have uniformly for all $0 \leq t \leq T$ that
\begin{align}
 \bigl\| \pt v(t) \bigr\|_{L^2_x} &\lesssim N(T) \jt^{+\delta}, \label{equ:auxiliary_ptv} \\
 \bigl\| Zv(t) \bigr\|_{L^2_x} &\lesssim N(T) \jt^{+\delta}. \label{equ:auxiliary_Zv}  
\end{align}
\end{lemma}
\begin{proof}
 The estimate~\eqref{equ:auxiliary_ptv} follows readily from the original first-order equation~\eqref{equ:first_order_kg} satisfied by $v(t)$
 \begin{align*}
  \|(\pt v)(t)\|_{L^2_x} &\lesssim \|(\jn v)(t)\|_{L^2_x} + \bigl\|\alpha(x) u(t)^2\bigr\|_{L^2_x} + \bigl\| \beta_0 u(t)^3 \bigr\|_{L^2_x} + \bigl\| \beta(x) u(t)^3 \bigr\|_{L^2_x} \\
  &\lesssim \|(\jn v)(t)\|_{L^2_x} + \|\alpha(x)\|_{L^2_x} \|v(t)\|_{L^\infty_x}^2 + \|v(t)\|_{L^2_x} \|v(t)\|_{L^\infty_x}^2 + \|\beta(x)\|_{L^2_x} \|v(t)\|_{L^\infty_x}^3 \\
  &\lesssim N(T) \jt^{+\delta}.
 \end{align*}
 To establish the bound~\eqref{equ:auxiliary_Zv}, we first use the identity $Z = iL + i \jn^{-1} \px + x (\pt - i\jn)$ to find that
 \begin{align*}
  \|(Zv)(t)\|_{L^2_x} &\lesssim \|(L v)(t)\|_{L^2_x} + \|v(t)\|_{L^2_x} + \bigl\| x (\pt - i \jn) v(t) \bigr\|_{L^2_x} \lesssim N(T) \jt^{+\delta} + \bigl\| x (\pt - i \jn) v(t) \bigr\|_{L^2_x}.
 \end{align*}
 From the equation~\eqref{equ:first_order_kg} for $v(t)$ we then obtain that
 \begin{align*}
  \bigl\| x (\pt - i \jn) v(t) \bigr\|_{L^2_x} &\lesssim \bigl\| x \jn^{-1} \bigl( \alpha(\cdot) u(t)^2 + \beta_0 u(t)^3 + \beta(\cdot) u(t)^3 \bigr) \bigr\|_{L^2_x} \\
  &\lesssim \bigl\| \jx \bigl( \alpha(x) u(t)^2 + \beta_0 u(t)^3 + \beta(x) u(t)^3 \bigr) \bigr\|_{L^2_x} \\
  &\lesssim \| \jx \alpha(x) \|_{L^2_x} \|v(t)\|_{L^\infty_x}^2 + \bigl( \|v(t)\|_{L^2_x} + \|x v(t)\|_{L^2_x} \bigr) \|v(t)\|_{L^\infty_x}^2 + \| \jx \beta(x)\|_{L^2_x} \|v(t)\|_{L^\infty_x}^3 \\
  &\lesssim \frac{N(T)^2}{\jt} + N(T) \jt^{1+\delta} \frac{N(T)^2}{\jt} + \frac{N(T)^3}{\jt^{\thf}} \\
  &\lesssim N(T)^2 \jt^{+\delta}.
 \end{align*}
 Hence, 
 \begin{equation*}
  \|(Zv)(t)\|_{L^2_x} \lesssim N(T)^2 \jt^{+\delta}.
 \end{equation*}
\end{proof}

We repeatedly exploit that spatial derivatives of the solution to~\eqref{equ:first_order_kg} have stronger local decay that is quantified in the following lemma.
\begin{lemma}[Improved local decay of spatial derivatives] \label{lem:improved_decay_localized_derivatives}
Let $\kappa(x)$ be a smooth and sufficiently decaying variable coefficient and let $v(t)$ be the solution to~\eqref{equ:first_order_kg}. Then we have uniformly for all $0 \leq t \leq T$ that
 \begin{align}
  \bigl\| \kappa(x) (\px u)(t) \bigr\|_{L^2_x} &\lesssim \bigl\| \jx \kappa(x) \bigr\|_{L^\infty_x} \frac{N(T)}{\jt^{1-\delta}}, \label{equ:decay_localized_px} \\
  \Bigl\| \kappa(x) \bigl( u(t, x)^2 - u(t,0)^2 \bigr) \Bigr\|_{L^2_x} &\lesssim \bigl\| \jx^2 \kappa(x) \bigr\|_{L^\infty_x} \frac{N(T)^2}{\jt^{\frac{3}{2}-\delta}}. \label{equ:decay_localized_difference}
 \end{align}
\end{lemma}
\begin{proof}
 We start with the proof of~\eqref{equ:decay_localized_px}.
For times $0 \leq t \leq 1$ we just bound 
\begin{align*}
 \bigl\| \kappa(x) (\px u)(t) \bigr\|_{L^2_x} \lesssim \|\kappa(x)\|_{L^\infty_x} \|\px v(t)\|_{L^2_x} \lesssim \|\kappa(x)\|_{L^\infty_x} N(T) \jt^{+\delta}.
\end{align*}
For $t \geq 1$ we first observe that
\begin{equation} \label{equ:tpx_decay_pickup}
\begin{aligned}
 \kappa(x) (\px v)(t) &= \frac{1}{(-it)} \kappa(x) \, (-it) (\px v)(t) \\
 &= \frac{1}{(-it)} \kappa(x) \, \bigl( -\jn x + L \bigr) v(t) \\
 &= \frac{1}{(-it)} \kappa(x) \, \bigl( - x \jn + \jn^{-1} \px + L \bigr) v(t) \\
 &= - \frac{1}{(-it)} x \kappa(x) (\jn v)(t) + \frac{1}{(-it)} \kappa(x) (\jn^{-1} \px v)(t) + \frac{1}{(-it)} \kappa(x) (L v)(t) 
\end{aligned}
\end{equation}
and thus
\begin{align*}
 \bigl\| \kappa(x) (\px v)(t) \bigr\|_{L^2_x} &\lesssim \frac{1}{t} \| x \kappa(x) \|_{L^\infty_x} \|(\jn v)(t)\|_{L^2_x} + \frac{1}{t} \|\kappa(x)\|_{L^\infty_x} \|v(t)\|_{L^2_x} + \frac{1}{t} \|\kappa(x)\|_{L^\infty_x} \|(L v)(t)\|_{L^2_x} \\
 &\lesssim \frac{1}{t} \bigl\| \jx \kappa(x) \bigr\|_{L^\infty_x} \bigl( \| (\jn v)(t) \|_{L^2_x} + \| (L v)(t) \|_{L^2_x} \bigr). 
\end{align*}
Hence, for $t \geq 1$ we have
\begin{align*}
 \bigl\| \kappa(x) (\px u)(t) \bigr\|_{L^2_x} &\lesssim \bigl\| \kappa(x) (\px v)(t) \bigr\|_{L^2_x} \lesssim \frac{1}{t} \bigl\| \jx \kappa(x) \bigr\|_{L^\infty_x}  N(T) \jt^{+\delta}.
\end{align*}
Putting things together, we obtain uniformly for all $t \geq 0$ that
\begin{align*}
 \bigl\| \kappa(x) (\px u)(t) \bigr\|_{L^2_x} \lesssim \|\jx \kappa(x)\|_{L^\infty_x} \frac{N(T)}{\jt^{1-\delta}}.
\end{align*}

Next, we give the proof of~\eqref{equ:decay_localized_difference}.
For times $0 \leq t \leq 1$ we just bound  
\begin{equation}
 \bigl\| \kappa(x) \bigl( u(t, x)^2 - u(t,0)^2 \bigr) \bigr\|_{L^2_x} \lesssim \|\kappa(x)\|_{L^2_x} \|v(t)\|_{L^\infty_x}^2 \lesssim \frac{N(T)^2}{\jt}.
\end{equation}
For $t \geq 1$ we first observe that by the fundamental theorem of calculus we have 
\begin{align*}
 \kappa(x) \bigl( u(t, x)^2 - u(t,0)^2 \bigr) &= \kappa(x) \bigl( u(t,x) - u(t,0) \bigr) \bigl( u(t,x) + u(t,0) \bigr) \\
 &= x \kappa(x) \Bigl( \int_0^1 (\px u)(t, \eta x) \, \ud \eta \Bigr) \bigl( u(t,x) + u(t,0) \bigr).
\end{align*}
Then in view of~\eqref{equ:tpx_decay_pickup} we may write
\begin{align*}
 x \kappa(x) \biggl( \int_0^1 (\px v)(t, \eta x) \, \ud \eta \biggr) &= - \frac{1}{(-it)} x^2 \kappa(x) \biggl( \int_0^1 \eta (\jn v)(t, \eta x) \, \ud \eta \biggr) + \frac{1}{(-it)} x \kappa(x) \biggl( \int_0^1 (\jn^{-1} \px v)(t, \eta x) \, \ud \eta \biggr) \\
 &\qquad + \frac{1}{(-it)} x \kappa(x) \biggl( \int_0^1 (L v)(t, \eta x) \, \ud \eta \biggr)
\end{align*}
and therefore bound 
\begin{align*}
 &\biggl\| x \kappa(x) \biggl( \int_0^1 (\px v)(t, \eta x) \, \ud \eta \biggr) \biggr\|_{L^2_x} \\
 &\lesssim \frac{1}{t} \bigl\| \jx^2 \kappa(x) \bigr\|_{L^\infty_x} \biggl( \int_0^1 \eta \|(\jn v)(t,\eta \cdot)\|_{L^2_x} \, \ud \eta + \int_0^1 \| v(t, \eta \cdot) \|_{L^2_x} \, \ud \eta + \int_0^1 \|(Lv)(t,\eta \cdot)\|_{L^2_x} \, \ud \eta \biggr) \\
 &\lesssim \frac{1}{t} \bigl\| \jx^2 \kappa(x) \bigr\|_{L^\infty_x} \biggl( \int_0^1 \eta^{\hf} \|(\jn v)(t)\|_{L^2_x} \, \ud \eta + \int_0^1 \eta^{-\hf} \|v(t)\|_{L^2_x} \, \ud \eta + \int_0^1 \eta^{-\hf} \|Lv(t)\|_{L^2_x} \, \ud \eta \biggr) \\
 &\lesssim \frac{1}{t} \bigl\| \jx^2 \kappa(x) \bigr\|_{L^\infty_x} \Bigl( \|\jn v(t)\|_{L^2_x} + \|v(t)\|_{L^2_x} + \|Lv(t)\|_{L^2_x} \Bigr).
\end{align*}
Hence, for $t \geq 1$ we find
\begin{align*}
 \Bigl\| \kappa(x) \bigl( u(t, x)^2 - u(t,0)^2 \bigr) \Bigr\|_{L^2_x} &\lesssim \biggl\| x \kappa(x) \Bigl( \int_0^1 (\px u)(t, \eta x) \, \ud \eta \Bigr) \bigl( u(t,x) + u(t,0) \bigr) \biggr\|_{L^2_x}   \\
 &\lesssim \biggl\| x \kappa(x) \biggl( \int_0^1 (\px v)(t, \eta \, \cdot) \, \ud \eta \biggr) \biggr\|_{L^2_x} \|v(t)\|_{L^\infty_x} \\
 &\lesssim \frac{1}{t} \bigl\| \jx^2 \kappa(x) \bigr\|_{L^\infty_x} \Bigl( \|\jn v(t)\|_{L^2_x} + \|v(t)\|_{L^2_x} + \|Lv(t)\|_{L^2_x} \Bigr) \|v(t)\|_{L^\infty_x} \\
 &\lesssim \frac{1}{t} \bigl\| \jx^2 \kappa(x) \bigr\|_{L^\infty_x} N(T) \jt^{+\delta} \frac{N(T)}{\jt^{\hf}}.
\end{align*}
Putting things together we arrive at the desired estimate 
\begin{equation}
 \Bigl\| \kappa(x) \bigl( u(t, x)^2 - u(t,0)^2 \bigr) \Bigr\|_{L^2_x} \leq \bigl\| \jx^2 \kappa(x) \bigr\|_{L^\infty_x} \frac{N(T)^2}{\jt^{\thf-\delta}}.
\end{equation} 
\end{proof}

Finally, the derivation of a slow growth estimate for the energy of $\jn L v(t)$ in Proposition~\ref{prop:slow_growth_H1L} below relies on the following local decay bounds for $\jn L \calV_\ell(t)$, $\ell = 1, \ldots, 5$. 
\begin{lemma}[Local decay bounds] \label{lem:weighted_energy_bounds}
Let $v(t)$ be the solution to~\eqref{equ:first_order_kg}. Then we have uniformly for all $0 \leq t \leq T$ that
\begin{align*}
 \bigl\| \jx^{-2} \jn L \calV_\ell(t) \bigr\|_{L^2_x} &\lesssim \frac{N(T)^2}{\jt^{\hf-\delta}} \quad \text{ for } \ell = 1, \ldots, 4, \\
 \bigl\| \jx^{-2} \jn L \calV_5(t) \bigr\|_{L^2_x} &\lesssim \frac{N(T)^3}{\jt^{\hf}}, \\
\end{align*}
\end{lemma}
\begin{proof}
We start off by observing that for $\ell = 1, \ldots, 5$ we have 
\begin{equation*}
 \bigl\| \jx^{-2} \jn L \calV_\ell(t) \bigr\|_{L^2_x} \lesssim \bigl\| \jx^{-1} \jn^2 \calV_\ell(t) \bigr\|_{L^2_x} + t \bigl\| \jx^{-2} \jn \px \calV_\ell(t) \bigr\|_{L^2_x}.
\end{equation*}
Then we obtain for the terms $\calV_\ell(t)$, $\ell = 1, 2, 3$, by the local decay estimates from Lemma~\ref{lem:local_decay} and the improved decay estimates from Lemma~\ref{lem:key_improved_decay} that
\begin{align*}
 &\sum_{1 \leq \ell \leq 3} \bigl\| \jx^{-2} \jn L \calV_\ell(t) \bigr\|_{L^2_x} \\
 &\quad \lesssim \sum_{1 \leq \ell \leq 3} \int_0^t \bigl\| \jx^{-1} e^{i(t-s)\jn} \jx^{-1} \bigr\|_{L^2_x \to L^2_x} \bigl\| \jx \alpha_\ell\bigr\|_{H^2_x} \bigl|\partial_s \bigl( e^{-is} v(s,0) \bigr)\bigr| |v(s,0)| \, \ud s \\
 &\quad \quad + \sum_{1 \leq \ell \leq 3} t \int_0^t \bigl\| \jx^{-2} \jn^{-1} \px e^{i(t-s)\jn} \jx^{-2} \bigr\|_{L^2_x \to L^2_x} \bigl\| \jx^2 \alpha_\ell \bigr\|_{H^2_x} \bigl|\partial_s \bigl( e^{-is} v(s,0) \bigr)\bigr| |v(s,0)| \, \ud s \\
 &\quad \lesssim \int_0^t \frac{1}{\jap{t-s}^{\hf}} \frac{N(T)^2}{\js^{\thf-\delta}} \, \ud s + t \int_0^t \frac{1}{\jap{t-s}^{\thf}} \frac{N(T)^2}{\js^{\thf-\delta}} \, \ud s \\
 &\quad \lesssim \frac{N(T)^2}{\jt^{\hf-\delta}}.
\end{align*}
Similarly, for the term $\calV_4(t)$ the local decay estimates from Lemma~\ref{lem:local_decay} in combination with the improved decay estimates from Lemma~\ref{lem:improved_decay_localized_derivatives} yield that
\begin{align*}
 \bigl\| \jx^{-2} \jn L \calV_4(t) \bigr\|_{L^2_x} &\lesssim \int_0^t \bigl\| \jx^{-1} e^{i(t-s)\jn} \jx^{-1} \bigr\|_{L^2_x \to L^2_x} \bigl\| \jx \alpha(x) \bigl( u(s,x)^2 - u(s,0)^2 \bigr) \bigr\|_{H^1_x} \, \ud s \\
 &\quad + t \int_0^t \bigl\| \jx^{-2} \jn^{-1} \px e^{i(t-s)\jn} \jx^{-2} \bigr\|_{L^2_x \to L^2_x} \bigl\| \jx^2 \alpha(x) \bigl( u(s,x)^2 - u(s,0)^2 \bigr) \bigr\|_{H^1_x} \, \ud s \\
 &\lesssim \int_0^t \frac{1}{\jap{t-s}^{\hf}} \frac{N(T)^2}{\js^{\thf-\delta}} \, \ud s + t \int_0^t \frac{1}{\jap{t-s}^{\thf}} \frac{N(T)^2}{\js^{\thf-\delta}} \, \ud s \\
 &\lesssim \frac{N(T)^2}{\jt^{\hf-\delta}}.
\end{align*}
Finally, for the term $\calV_5(t)$ we have that
\begin{align*}
 \bigl\| \jx^{-2} \jn L \calV_5(t) \bigr\|_{L^2_x} &\lesssim \int_0^t \bigl\| \jx^{-1} e^{i(t-s)\jn} \jx^{-1} \bigr\|_{L^2_x \to L^2_x} \bigl\| \jx \beta(x) u(s)^3 \bigr\|_{H^1_x} \, \ud s \\
 &\quad + t \int_0^t \bigl\| \jx^{-2} \jn^{-1} \px e^{i(t-s)\jn} \jx^{-2} \bigr\|_{L^2_x \to L^2_x} \bigl\| \jx^2 \beta(x) u(s)^3 \bigr\|_{H^1_x} \, \ud s.
\end{align*}
By Lemma~\ref{lem:improved_decay_localized_derivatives} we obtain the bound
\begin{align*}
 \bigl\| \jx^2 \beta(x) u(s)^3 \bigr\|_{H^1_x} &\lesssim \bigl\| \jx^2 \beta(x) \bigr\|_{H^1_x} \|v(s)\|_{L^\infty_x}^3 + \bigl\| \jx^2 \beta(x) (\px u)(s) \bigr\|_{L^2_x} \|v(s)\|_{L^\infty_x}^2 \lesssim \frac{N(T)^3}{\js^{\thf}},
\end{align*}
and hence, using the local decay estimates from Lemma~\ref{lem:local_decay}, we arrive at the desired estimate
\begin{align*}
 \bigl\| \jx^{-2} \jn L \calV_5(t) \bigr\|_{L^2_x} &\lesssim \int_0^t \frac{1}{\jap{t-s}^{\hf}} \frac{N(T)^3}{\js^{\thf}} \, \ud s + t \int_0^t \frac{1}{\jap{t-s}^{\thf}} \frac{N(T)^3}{\js^{\thf}} \, \ud s \lesssim \frac{N(T)^3}{\jt^{\hf}}. 
\end{align*}
\end{proof}

\subsubsection{Main energy growth estimates}

We are now in the position to derive growth bounds for all energy norms that are part of the bootstrap quantity $N(T)$.
\begin{proposition} \label{prop:slow_growth_H2v}
Let $v(t)$ be the solution to~\eqref{equ:first_order_kg}. Then uniformly for all $0 \leq t \leq T$ it holds that
 \begin{equation}
  \bigl\| \jn^2 v(t) \bigr\|_{L^2_x} \lesssim \|v_0\|_{H^2_x} + \|v_0\|_{H^1_x}^2 + N(T)^2 \jt^{+\delta}.  
 \end{equation}
\end{proposition}
\begin{proof}
 From the Duhamel formula~\eqref{equ:duh_v_rearr} for $v(t)$ we obtain that
 \begin{align*}
  \bigl\|\jn^2 v(t)\bigr\|_{L^2_x} &\lesssim \|v_0\|_{H^2_x} + \sum_{1 \leq j \leq 3} \bigl\| \jn^2 \alpha_j \bigr\|_{L^2_x} \|v_0\|_{L^\infty_x}^2 + \sum_{1 \leq j \leq 3} \bigl\| \jn^2 \alpha_j \bigr\|_{L^2_x} \|v(t)\|_{L^\infty_x}^2 \\
  &\quad + \int_0^t \bigl\| \beta_0 u(s)^3 \bigr\|_{H^1_x} \, \ud s + \int_0^t \bigl\| \beta(x) u(s)^3 \bigr\|_{H^1_x} \, \ud s + \int_0^t \bigl\| \alpha(x) \bigl( u(s,x)^2 - u(s,0)^2 \bigr) \bigr\|_{H^1_x} \, \ud s \\
  &\quad + \sum_{1 \leq j \leq 3} \int_0^t \bigl\| \jn^2 \alpha_j \bigr\|_{L^2_x} \bigl| \partial_s \bigl( e^{-is} v(s,0) \bigr) \bigr| |v(s,0)| \, \ud s.
 \end{align*}
 Then we have 
 \begin{align*}
  \bigl\| u(s)^3 \bigr\|_{H^1_x} \lesssim \|v(s)\|_{H^1_x} \|v(s)\|_{L^\infty_x}^2 \lesssim \frac{N(T)^3}{\js^{1-\delta}}
 \end{align*}
 and by Lemma~\ref{lem:improved_decay_localized_derivatives} we also have 
 \begin{align*}
  \bigl\| \beta(x) u(s)^3 \bigr\|_{H^1_x} &\lesssim \bigl\| \beta(x) u(s)^3 \bigr\|_{L^2_x} + \bigl\| \beta'(x) u(s)^3 \bigr\|_{L^2_x} + \bigl\| \beta(x) (\px u)(s) u(s)^2 \bigr\|_{L^2_x} \\
  &\lesssim \|\beta(x)\|_{H^1_x} \|v(s)\|_{L^\infty_x}^3 + \bigl\| \beta(x) (\px u)(s) \bigr\|_{L^2_x} \|v(s)\|^2_{L^\infty_x} \\
  &\lesssim \frac{N(T)^3}{\js^{\thf}} + \frac{N(T)^3}{\js^{\thf-\delta}}, \\
  \bigl\| \alpha(x) \bigl( u(s,x)^2 - u(s,0)^2 \bigr) \bigr\|_{H^1_x} &\lesssim \bigl\| \bigl( |\alpha(x)| + |\alpha'(x)| \bigr) \bigl( u(s,x)^2 - u(s,0)^2 \bigr) \bigr\|_{L^2_x} + \bigl\| \alpha(x) (\px u)(s) \bigr\|_{L^2_x} \|v(s)\|_{L^\infty_x} \\
  &\lesssim \frac{N(T)^2}{\jt^{\thf-\delta}}.
 \end{align*}
 Using the previous bounds and Lemma~\ref{lem:key_improved_decay}, we conclude that
 \begin{align*}
  \bigl\|\jn^2 v(t)\bigr\|_{L^2_x} &\lesssim \|v_0\|_{H^2_x} + \|v_0\|_{H^1_x}^2 + \frac{N(T)^2}{\jt} + \int_0^t \frac{N(T)^3}{\js^{1-\delta}} \, \ud s + \int_0^t \frac{N(T)^3}{\js^{\thf}} \, \ud s + \int_0^t \frac{N(T)^2}{\js^{\thf-\delta}} \, \ud s \\
  &\lesssim \|v_0\|_{H^2_x} + \|v_0\|_{H^1_x}^2 + N(T)^2 \jt^{+\delta}.
 \end{align*}
\end{proof}

\begin{proposition} \label{prop:slow_growth_xv}
Let $v(t)$ be the solution to~\eqref{equ:first_order_kg}. Then uniformly for all $0 \leq t \leq T$ it holds that
 \begin{equation}
  \| x v(t) \|_{L^2_x} \lesssim \bigl( \|x v_0\|_{L^2_x} + \|v_0\|_{H^1_x}^2 \bigr) \jt + N(T)^2 \jt^{1+\delta}.
 \end{equation}
\end{proposition}
\begin{proof}
 Starting from the Duhamel formula~\eqref{equ:duh_v_rearr} for $v(t)$ and using that 
 \begin{equation*}
  x e^{it\jn} = e^{it\jn} (x + i t \jn^{-1} \px),
 \end{equation*}
 we find 
 \begin{align*}
  \|x v(t)\|_{L^2_x} &\lesssim \|x v_0\|_{L^2_x} + \sum_{1 \leq j \leq 3} \|x \alpha_j(x)\|_{L^2_x} \|v_0\|_{H^1_x}^2 + t \Bigl( \|v_0\|_{L^2_x} + \sum_{1 \leq j \leq 3} \|\alpha_j(x)\|_{L^2_x} \|v_0\|_{H^1_x}^2 \Bigr) \\
  &\quad + \sum_{1 \leq j \leq 3} \|x \alpha_j(x)\|_{L^2_x} \|v(t)\|_{L^\infty_x}^2 + \bigl\| x \calC(t) \bigr\|_{L^2_x} + \sum_{1 \leq \ell \leq 5} \bigl\| x \calV_\ell(t) \bigr\|_{L^2_x}. 
 \end{align*}
 Then we crudely bound 
 \begin{align*}
  \bigl\| x \calC(t) \bigr\|_{L^2_x} &\lesssim \int_0^t \bigl\| \bigl( x + i (t-s) \jn^{-1} \px \bigr) \jn^{-1} \bigl( u(s)^3 \bigr) \bigr\|_{L^2_x} \, \ud s \\
  &\lesssim \int_0^t \bigl( \|xv(s)\|_{L^2_x} + \|v(s)\|_{L^2_x} \bigr) \|v(s)\|_{L^\infty_x}^2 \, \ud s + t \int_0^t \|v(s)\|_{L^2_x} \|v(s)\|_{L^\infty_x}^2 \, \ud s \\
  &\lesssim \int_0^t N(T) \js^{1+\delta} \frac{N(T)^2}{\js} \, \ud s + t \int_0^t N(T) \js^{+\delta} \frac{N(T)^2}{\js} \, \ud s \\
  &\lesssim N(T)^3 \jt^{1+\delta}
 \end{align*}
 and in a similar manner, we obtain for $1 \leq \ell \leq 5$ that
 \begin{align*}
  \bigl\| x \calV_\ell(t) \bigr\|_{L^2_x} \lesssim N(T)^2 \jt.
 \end{align*}
 Putting all of the previous estimates together, we arrive at the desired bound
 \begin{align*}
  \|x v(t)\|_{L^2_x} &\lesssim \bigl( \|x v_0\|_{L^2_x} + \|v_0\|_{H^1_x}^2 \bigr) \jt + N(T)^2 \jt^{1+\delta}.
 \end{align*}
\end{proof}

\begin{proposition} \label{prop:slow_growth_H1L}
Let $v(t)$ be the solution to~\eqref{equ:first_order_kg}. Then uniformly for all $0 \leq t \leq T$ it holds that
 \begin{equation}
  \bigl\| \jn L v(t) \bigr\|_{L^2_x} \lesssim \| x v_0 \|_{H^2_x} + \|v_0\|_{H^1_x}^2 + N(T)^2 \jt^{+\delta}.
 \end{equation}
\end{proposition}
\begin{proof}
From the Duhamel formula~\eqref{equ:duh_v_rearr} and the identity $\jn L e^{it\jn} = e^{it\jn} \jn^2 x$ we obtain that
\begin{align*}
 \bigl\| \jn L v(t) \bigr\|_{L^2_x} &\lesssim \|x v_0\|_{H^2_x} + \sum_{1 \leq j \leq 3} \|x \alpha_j(x)\|_{H^2_x} \|v_0\|_{L^\infty_x}^2 + \sum_{1 \leq j \leq 3} \bigl( \| x \alpha_j(x)\|_{H^2_x} + t \|\alpha_j(x)\|_{H^2_x} \bigr) \|v(t)\|_{L^\infty_x}^2 \\
 &\quad + \bigl\| \jn L \calC(t) \bigr\|_{L^2_x} + \sum_{1 \leq \ell \leq 5} \bigl\| \jn L \calV_\ell(t) \bigr\|_{L^2_x} \\
 &\lesssim \|x v_0\|_{H^2_x} + \|v_0\|_{H^1_x}^2 + \jt \frac{N(T)^2}{\jt} + \bigl\| \jn L \calC(t) \bigr\|_{L^2_x} + \sum_{1 \leq \ell \leq 5} \bigl\| \jn L \calV_\ell(t) \bigr\|_{L^2_x}.
\end{align*}
We now have to control the growth of the energies of $\jn L$ acting on the constant coefficient cubic term~$\calC(t)$ and on the variable coefficient terms $\calV_\ell(t)$, $\ell = 1, \ldots, 5$.
Since the action of the operator~$L$ on the constant coefficient cubic nonlinearity is difficult to compute, we first derive a bound on the growth of the energy of a Lorentz boost $\jn Z$ acting on $\calC(t)$ and then use that
\begin{equation} \label{equ:relate_jnL_jnZ}
 i \jn L = \jn Z - i \px - x \jn (\pt - i \jn) + \jn^{-1} \px (\pt - i \jn).
\end{equation}
To this end we compute that
\begin{align*}
 (\pt - i \jn) (\jn Z \calC) &= \frac{\beta_0}{2i} Z \bigl( u^3 \bigr) + \frac{\beta_0}{2i} \bigl[ \jn, Z \bigr] \jn^{-1} \bigl( u^3 \bigr) + \bigl[ (\pt - i \jn), \jn Z \bigr] \calC \\
 &= \frac{\beta_0}{2i} Z \bigl( u^3 \bigr) - \frac{\beta_0}{2i} \jn^{-2} \px \pt \bigl( u^3 \bigr) + i \px (\pt - i \jn) \calC \\
 &= \frac{3 \beta_0}{2i} (Zu) u^2 - \frac{3 \beta_0}{2i} \jn^{-2} \px \bigl( (\pt u) u^2 \bigr) + \frac{\beta_0}{2i} \jn^{-1} \px \bigl( u^3 \bigr).
\end{align*}
Using the energy estimate~\eqref{equ:energy_est_std} and the auxiliary bounds from Lemma~\ref{lem:auxiliary_bounds}, we then obtain that
\begin{align*}
 \bigl\| (\jn Z \calC)(t) \bigr\|_{L^2_x} &\lesssim \int_0^t \|Zv(s)\|_{L^2_x} \|v(s)\|_{L^\infty_x}^2 \, \ud s + \int_0^t \| \pt v(s) \|_{L^2_x}  \|v(s)\|_{L^\infty_x}^2 \, \ud s + \int_0^t \|v(s)\|_{L^2_x} \|v(s)\|_{L^\infty_x}^2 \, \ud s \\
 &\lesssim \int_0^t N(T) \js^{+\delta} \frac{N(T)^2}{\js} \, \ud s \\
 &\lesssim N(T)^3 \jt^{+\delta},
\end{align*}
and thus by~\eqref{equ:relate_jnL_jnZ} that
\begin{align*}
 \bigl\| (\jn L \calC)(t) \bigr\|_{L^2_x} &\lesssim \bigl\| (\jn Z \calC)(t) \bigr\|_{L^2_x} + \bigl\| \px \calC(t) \bigr\|_{L^2_x} + \bigl\| x \jn (\pt - i \jn) \calC(t) \bigr\|_{L^2_x} + \bigl\| \jn^{-1} \px (\pt - i \jn) \calC(t) \bigr\|_{L^2_x} \\
 &\lesssim N(T)^3 \jt^{+\delta} + \int_0^t \| v(s) \|_{L^2_x} \|v(s)\|_{L^\infty_x}^2 \, \ud s + \|x v(t)\|_{L^2_x} \|v(t)\|_{L^\infty_x}^2 + \|v(t)\|_{L^2_x} \|v(t)\|_{L^\infty_x}^2 \\
 &\lesssim N(T)^3 \jt^{+\delta} + \int_0^t N(T) \js^{+\delta} \frac{N(T)^2}{\js} \, \ud s + N(T) \jt^{1+\delta} \frac{N(T)^2}{\jt} + N(T) \jt^{+\delta} \frac{N(T)^2}{\jt} \\
 &\lesssim N(T)^3 \jt^{+\delta}.
\end{align*}

Next, we estimate the growth of the energies of $\jn L$ acting on the variable coefficient terms $\calV_\ell(t)$, $\ell = 1, \ldots, 5$. Using the energy estimate~\eqref{equ:energy_est_sq}, the local decay bounds from Lemma~\ref{lem:weighted_energy_bounds}, and the improved decay bound from Lemma~\ref{lem:key_improved_decay}, we obtain for $\ell = 1, 2, 3$ that
\begin{align*}
 \bigl\| \jn L \calV_\ell(t) \bigr\|_{L^2_x}^2 &\lesssim \int_0^t \bigl\| (\pt - i\jn) \bigl(\jn L \calV_\ell\bigr)(s) \overline{\bigl( \jn L \calV_\ell \bigr)(s)} \bigr\|_{L^1_x} \, \ud s \\
 &\lesssim \int_0^t \bigl\| \jx^2 \jn L (\pt - i\jn) \calV_\ell(s) \bigr\|_{L^2_x} \bigl\| \jx^{-2} \bigl( \jn L \calV_\ell \bigr)(s) \bigr\|_{L^2_x} \, \ud s \\
 &\lesssim \int_0^t \js \bigl\| \jx^3 \alpha_\ell(x) \bigr\|_{H^2_x} \bigl| \partial_s \bigl( e^{-is} v(s,0) \bigr) \bigr| |v(s,0)| \bigl\| \jx^{-2} \bigl( \jn L \calV_\ell \bigr)(s) \bigr\|_{L^2_x} \, \ud s \\
 &\lesssim \int_0^t \js \frac{N(T)^2}{\js^{\thf-\delta}} \frac{N(T)^2}{\js^{\hf-\delta}} \, \ud s \\
 &\lesssim N(T)^2 \jt^{+2\delta}.
\end{align*}
Similarly, for the term $\calV_4(t)$ we use the energy estimate~\eqref{equ:energy_est_sq}, the weighted energy bounds from Lemma~\ref{lem:weighted_energy_bounds} and the improved decay bounds from Lemma~\ref{lem:improved_decay_localized_derivatives}, to infer that
\begin{align*}
 \bigl\| \jn L \calV_4(t) \bigr\|_{L^2_x}^2 &\lesssim \int_0^t \bigl\| (\pt - i\jn) \bigl(\jn L \calV_4\bigr)(s) \overline{\bigl( \jn L \calV_4 \bigr)(s)} \bigr\|_{L^1_x} \, \ud s \\
 &\lesssim \int_0^t \bigl\| \jx^2 \jn L (\pt - i\jn) \calV_4(s) \bigr\|_{L^2_x} \bigl\| \jx^{-2} \bigl( \jn L \calV_4 \bigr)(s) \bigr\|_{L^2_x} \, \ud s \\
 &\lesssim \int_0^t \js \bigl\| \jx^3 \alpha(x) \bigl( u(s,x)^2 - u(s,0)^2 \bigr) \bigr\|_{H^1_x} \bigl\| \jx^{-2} \bigl( \jn L \calV_4 \bigr)(s) \bigr\|_{L^2_x} \, \ud s \\
 &\lesssim \int_0^t \js \bigl\| \jx^3 \alpha(x) \bigr\|_{H^1_x} \frac{N(T)^2}{\js^{\thf-\delta}} \frac{N(T)^2}{\js^{\hf-\delta}} \, \ud s \\
 &\lesssim N(T)^4 \jt^{+2\delta}.
\end{align*}
Analogously, we also find that $\| \jn L \calV_5(t) \|_{L^2_x}^2 \lesssim N(T)^6 \jt^{+2\delta}$. Combining all of the above estimates, we arrive at the desired bound
\begin{align*}
 \bigl\| \jn L v(t) \bigr\|_{L^2_x} &\lesssim \|x v_0\|_{H^2_x} + \|v_0\|_{H^1_x}^2 + N(T)^2 \jt^{+\delta}.
\end{align*}
\end{proof}

\subsection{$L^\infty_\xi$ control of the profile} 

In this subsection we obtain an a priori bound on a weighted $L^\infty_\xi$ norm of the profile of the solution to~\eqref{equ:first_order_kg}.

\begin{proposition} \label{prop:Linfty_bound_profile}
Let $f(t) = e^{-it\jn} v(t)$ be the profile of the solution $v(t)$ to~\eqref{equ:first_order_kg}. Then we have uniformly for all $1 \leq t \leq T$ that
 \begin{equation} \label{equ:Linfty_bound_profile}
  \bigl\| \jxi^{\thf} \hat{f}(t,\xi) \bigr\|_{L^\infty_\xi} \lesssim \bigl\| \jxi^{\thf} \hat{f}(1,\xi) \bigr\|_{L^\infty_\xi} + N(T)^2. 
 \end{equation}
\end{proposition}

The proof of Proposition~\ref{prop:Linfty_bound_profile} consists of an ODE argument. The main work in fact goes into deriving the following differential equation for the profile.
\begin{lemma} \label{lem:ode_weighted_profile}
 Let $f(t) = e^{-it\jn} v(t)$ be the profile of the solution $v(t)$ to~\eqref{equ:first_order_kg}. Then there exists a small absolute constant $0 < \nu \ll 1$ such that for all $1 \leq t \leq T$ it holds that
 \begin{equation} \label{equ:ode_weighted_profile}
  \begin{aligned}
   &\pt \Bigl( \jxi^{\thf} \hat{f}(t, \xi) + \calR(t,\xi) \Bigr) \\
   &\quad = \frac{1}{t} \frac{3 \beta_0}{2i} \jxi^{-1} \bigl| \jxi^\thf \hat{f}(t, \xi) \bigr|^2 \jxi^\thf \hat{f}(t,\xi) + \frac{1}{t} \frac{\beta_0}{2 \sqrt{3}} \jxi^{\hf} \jap{{\textstyle \frac{\xi}{3}}}^3 e^{it(-\jxi + 3 \jap{\frac{\xi}{3}})} \hat{f}(t, {\textstyle \frac{\xi}{3}})^3  \\
   &\quad \quad + \frac{1}{t} \frac{3 \beta_0}{2i} \jxi^{\frac{7}{2}} e^{-2 i t \jxi} \bigl| \hat{f}(t, -\xi) \bigr|^2 \bar{\hat{f}}(t,-\xi) - \frac{1}{t} \frac{\beta_0}{2 \sqrt{3}} \jap{\xi}^{\hf} \jap{{\textstyle \frac{\xi}{3}}}^3 e^{-it(\jxi + 3 \jap{\frac{\xi}{3}})} \hat{\bar{f}}(t, {\textstyle \frac{\xi}{3}})^3 + \calO \biggl( \frac{N(T)^2}{t^{1+\nu}} \biggr),
  \end{aligned}
 \end{equation}
 where 
 \begin{equation*}
  \| \calR(t,\xi) \|_{L^\infty_\xi} \lesssim \frac{N(T)^2}{t}, \quad 1 \leq t \leq T.
 \end{equation*}
\end{lemma}

We first give the short proof of Proposition~\ref{prop:Linfty_bound_profile} before we turn to the lengthier derivation of the differential equation for the profile.
\begin{proof}[Proof of Proposition~\ref{prop:Linfty_bound_profile}]
 We define 
 \begin{equation} \label{equ:integrating_factor_phase_def}
  B(t) := \frac{3 \beta_0}{2} \jxi^{-1} \int_1^t \frac{1}{s} \bigl| \jxi^\thf \hat{f}(s, \xi) \bigr|^2 \, \ud s, \quad 1 \leq t \leq T.
 \end{equation}
 Then we multiply the differential equation \eqref{equ:ode_weighted_profile} by the integrating factor $e^{iB(t)}$ to obtain that 
 \begin{equation} \label{equ:ode_weighted_profile_multiplied_intfactor}
  \pt \Bigl( \jxi^{\thf} \hat{f}(t, \xi) e^{iB(t)} + \calR(t,\xi) e^{iB(t)} \Bigr) = \calI^{(1)}(t) + \calI^{(2)}(t) + \calI^{(3)}(t) + \calO \biggl( \frac{N(T)^2}{t^{1+\nu}} \biggr),
 \end{equation}
 where 
 \begin{align*}
  \calI^{(1)}(t) &:= \frac{1}{t} \frac{\beta_0}{2 \sqrt{3}} \jxi^{\hf} \jap{{\textstyle \frac{\xi}{3}}}^3 e^{it(-\jxi + 3 \jap{\frac{\xi}{3}})} \hat{f}(t, {\textstyle \frac{\xi}{3}})^3 e^{iB(t)}, \\
  \calI^{(2)}(t) &:= \frac{1}{t} \frac{3 \beta_0}{2i} \jxi^{\frac{7}{2}} e^{-2 i t \jxi} \bigl| \hat{f}(t, -\xi) \bigr|^2 \bar{\hat{f}}(t,-\xi) e^{iB(t)}, \\
  \calI^{(3)}(t) &:= - \frac{1}{t} \frac{\beta_0}{2 \sqrt{3}} \jap{\xi}^{\hf} \jap{{\textstyle \frac{\xi}{3}}}^3 e^{-it(\jxi + 3 \jap{\frac{\xi}{3}})} \hat{\bar{f}}(t, {\textstyle \frac{\xi}{3}})^3 e^{iB(t)}.
 \end{align*}
 Upon showing that
 \begin{equation} \label{equ:integrated_in_time_bound_calIell}
  \sup_{1 \leq t \leq T} \, \biggl\| \int_1^t \calI^{(\ell)}(s) \, \ud s \biggr\|_{L^\infty_\xi} \lesssim N(T)^3 \quad \text{ for } 1 \leq \ell \leq 3,
 \end{equation}
 the asserted estimate~\eqref{equ:Linfty_bound_profile} follows from integrating~\eqref{equ:ode_weighted_profile_multiplied_intfactor} in time and taking the $L^\infty_\xi$ norm. 
 The latter bound~\eqref{equ:integrated_in_time_bound_calIell} is a consequence of the oscillations of the phases in the terms $\calI^{(\ell)}(t)$, $\ell = 1, 2, 3$. Indeed, starting with the term $\calI^{(1)}(t)$, we rewrite it as
 \begin{align*}
  \calI^{(1)}(t) &= \pt \biggl( \frac{1}{t} \frac{\beta_0}{2 \sqrt{3}} \jxi^{\hf} \jap{{\textstyle \frac{\xi}{3}}}^3 (-i) \bigl( (-\jxi + 3 \jap{{\textstyle \frac{\xi}{3}}} \bigr)^{-1}  e^{it(-\jxi + 3 \jap{\frac{\xi}{3}})} \hat{f}(t, {\textstyle \frac{\xi}{3}})^3 e^{iB(t)} \biggr) \\
  &\quad + \frac{1}{t^2} \frac{\beta_0}{2 \sqrt{3}} \jxi^{\hf} \jap{{\textstyle \frac{\xi}{3}}}^3 (-i) \bigl( (-\jxi + 3 \jap{{\textstyle \frac{\xi}{3}}} \bigr)^{-1}  e^{it(-\jxi + 3 \jap{\frac{\xi}{3}})} \hat{f}(t, {\textstyle \frac{\xi}{3}})^3 e^{iB(t)} \\
  &\quad - \frac{1}{t} \frac{3 \beta_0}{2 \sqrt{3}} \jxi^{\hf} \jap{{\textstyle \frac{\xi}{3}}}^3 (-i) \bigl( (-\jxi + 3 \jap{{\textstyle \frac{\xi}{3}}} \bigr)^{-1}  e^{it(-\jxi + 3 \jap{\frac{\xi}{3}})}  (\pt \hat{f})(t, {\textstyle \frac{\xi}{3}}) \hat{f}(t, {\textstyle \frac{\xi}{3}})^2 e^{iB(t)} \\
  &\quad - \frac{1}{t^2} \frac{3 \beta_0^2}{4 \sqrt{3}} \jxi^{\hf} \jap{{\textstyle \frac{\xi}{3}}}^3 \bigl( (-\jxi + 3 \jap{{\textstyle \frac{\xi}{3}}} \bigr)^{-1}  e^{it(-\jxi + 3 \jap{\frac{\xi}{3}})} \hat{f}(t, {\textstyle \frac{\xi}{3}})^3 \jxi^{-1} \bigl| \jxi^\thf \hat{f}(t, \xi) \bigr|^2 e^{iB(t)}.
 \end{align*}
 Then noting that $(-\jxi + 3 \jap{{\textstyle \frac{\xi}{3}}} \bigr)^{-1} = \calO \bigl( \jxi \bigr)$ and that~\eqref{equ:pt_ft_profile_rewritten} yields by direct computation the crude estimate 
 \begin{equation*}
  \bigl\| \jxi^\thf \pt \hat{f}(t,\xi) \bigr\|_{L^\infty_\xi} \lesssim N(T)^2 \jt^{-\frac{1}{2}+2\delta}, \quad 0 \leq t \leq T,
 \end{equation*}
 the bound~\eqref{equ:integrated_in_time_bound_calIell} for $\calI^{(1)}(t)$ follows readily
 \begin{align*}
  \sup_{1 \leq t \leq T} \, \biggl\| \int_1^t \calI^{(1)}(s) \, \ud s \biggr\|_{L^\infty_\xi} &\lesssim \sup_{1 \leq t \leq T} \, \bigl\| \jxi^\thf \hat{f}(t,\xi) \bigr\|_{L^\infty_\xi}^3 + \int_1^T \frac{1}{s^2} \bigl\| \jxi^\thf \hat{f}(s,\xi) \bigr\|_{L^\infty_\xi}^3 \, \ud s \\
  &\quad + \int_1^T \frac{1}{s} \bigl\| \jxi^\thf \pt \hat{f}(s,\xi) \bigr\|_{L^\infty_\xi} \bigl\| \jxi^\thf \hat{f}(s,\xi) \bigr\|_{L^\infty_\xi}^2 \, \ud s + \int_1^T \frac{1}{s^2} \bigl\| \jxi^\thf \hat{f}(s,\xi) \bigr\|_{L^\infty_\xi}^5 \, \ud s \\
  &\lesssim N(T)^3 + \int_1^T \frac{1}{s^2} N(T)^3 \, \ud s + \int_1^T \frac{1}{s^{\thf-2\delta}} N(T)^4 \, \ud s + \int_1^T \frac{1}{s^2} N(T)^5 \, \ud s \\
  &\lesssim N(T)^3.
 \end{align*}
 For the terms $\calI^{(2)}(t)$ and $\calI^{(3)}(t)$ the bound~\eqref{equ:integrated_in_time_bound_calIell} can be derived analogously.
\end{proof}

Now we turn to the derivation of the ODE~\eqref{equ:ode_weighted_profile} for the profile asserted in Lemma~\ref{lem:ode_weighted_profile}.
\begin{proof}[Proof of Lemma~\ref{lem:ode_weighted_profile}]
 Multiplying the differential equation~\eqref{equ:pt_ft_profile_recast} for $\hat{f}(t,\xi)$ by $\jxi^{\thf}$ gives
 \begin{equation} \label{equ:ode_weighted_profile_derivation1}
  \begin{aligned}
   &\pt \Bigl( \jxi^{\thf} \hat{f}(t, \xi) + e^{-it\jap{\xi}} \jxi^{\thf} \bigl( \widehat{\alpha}_1(\xi) v(t,0)^2 + \widehat{\alpha}_2(\xi) |v(t,0)|^2 + \widehat{\alpha}_3(\xi) \bar{v}(t,0)^2 \bigr) \Bigr) \\
   &= \frac{\beta_0}{2i} e^{-it\jap{\xi}} \jap{\xi}^{\hf} \calF\bigl[ u(t, \cdot)^3 \bigr](\xi) \\ 
   &\qquad + 2 e^{-it\jap{\xi}} \jap{\xi}^{\thf} \widehat{\alpha}_1(\xi) e^{2it} \partial_t \bigl( e^{-it} v(t,0) \bigr) \bigl( e^{-it} v(t,0) \bigr) \\
   &\qquad + 2 e^{-it\jap{\xi}} \jap{\xi}^{\thf} \widehat{\alpha}_2(\xi) \, \Re \Bigl( \partial_t \bigl( e^{-it} v(t,0) \bigr) \bigl( e^{+it} \bar{v}(t,0) \bigr) \Bigr) \\
   &\qquad + 2 e^{-it\jap{\xi}} \jap{\xi}^{\thf} \widehat{\alpha}_3(\xi) e^{-2it} \partial_t \bigl( e^{+it} \bar{v}(t,0) \bigr) \bigl( e^{+it} \bar{v}(t,0) \bigr) \\
   &\qquad + \frac{1}{2i} e^{-it\jap{\xi}} \jap{\xi}^{\hf} \calF\Bigl[ \alpha(\cdot) \bigl( u(t, \cdot)^2 - u(t,0)^2 \bigr) \Bigr](\xi) \\
   &\qquad + \frac{1}{2i} e^{-it\jap{\xi}} \jap{\xi}^{\hf} \calF\bigl[ \beta(\cdot) u(t, \cdot)^3 \bigr](\xi).
  \end{aligned}
 \end{equation}
 We already note that the term
 \begin{equation*}
  \calR(t,\xi) := e^{-it\jap{\xi}} \jxi^{\thf} \bigl( \widehat{\alpha}_1(\xi) v(t,0)^2 + \widehat{\alpha}_2(\xi) |v(t,0)|^2 + \widehat{\alpha}_3(\xi) \bar{v}(t,0)^2 \bigr) 
 \end{equation*}
 on the left-hand side of~\eqref{equ:ode_weighted_profile_derivation1} satisfies the claimed bound
 \begin{equation*}
  \| \calR(t,\xi) \|_{L^\infty_\xi} \lesssim \sum_{1\leq\ell\leq3} \bigl\| \jxi^\thf \widehat{\alpha}_\ell(\xi) \bigr\|_{L^\infty_\xi} \|v(t)\|_{L^\infty_x}^2 \lesssim \frac{N(T)^2}{\jt}, \quad 0 \leq t \leq T.
 \end{equation*}
 Moreover, we observe that all terms apart from the contribution of the constant coefficient cubic term, i.e. the first term on the right-hand side of~\eqref{equ:ode_weighted_profile_derivation1}, have integrable time decay. Indeed, for $1 \leq \ell \leq 3$ we can use Lemma~\ref{lem:key_improved_decay} to crudely bound
 \begin{align*}
  \Bigl\| \jap{\xi}^{\thf} \widehat{\alpha}_\ell(\xi) \bigl|\partial_t \bigl( e^{-it} v(t,0) \bigr)\bigr| \bigl| e^{-it} v(t,0) \bigr| \Bigr\|_{L^\infty_\xi} &\lesssim \bigl\| \jap{\xi}^{\thf} \widehat{\alpha}_\ell(\xi) \bigr\|_{L^\infty_\xi} \bigl| \partial_t \bigl( e^{-it} v(t,0) \bigr) \bigr| |v(t,0)| \lesssim \frac{N(T)^2}{\jt^{\thf-\delta}}, \quad 0 \leq t \leq T.
 \end{align*}
 Further, by Lemma~\ref{lem:improved_decay_localized_derivatives} it holds that
 \begin{align*}
  &\Bigl\| \jap{\xi}^{\hf} \calF\Bigl[ \alpha(\cdot) \bigl( u(t, \cdot)^2 - u(t,0)^2 \bigr) \Bigr](\xi) \Bigr\|_{L^\infty_\xi} \\
  &\quad \lesssim \bigl\| \jx \alpha(x) \bigl( u(t,x)^2 - u(t,0)^2 \bigr) \bigr\|_{H^1_x} \\
  &\quad \lesssim \sum_{j=0}^1 \, \bigl\| \px^j \bigl( \jx \alpha(x) \bigr) \bigl( u(t,x)^2 - u(t,0)^2 \bigr) \bigr\|_{L^2_x} + \bigl\| \jx \alpha(x) (\px u)(t) u(t) \bigr\|_{L^2_x} \\
  &\quad \lesssim \frac{N(T)^2}{\jt^{\thf-\delta}}, \quad 0 \leq t \leq T,
 \end{align*}
 as well as
 \begin{align*}
  \bigl\| \jxi^{\hf} \calF\bigl[ \beta(\cdot) u(t, \cdot)^3 \bigr](\xi) \bigr\|_{L^\infty_\xi} &\lesssim \bigl\| \jxi^\hf \calF\bigl[ \beta(\cdot) u(t, \cdot)^3 \bigr](\xi) \bigr\|_{H^1_\xi} \lesssim \bigl\| \jx \beta(x) u(t)^3 \bigr\|_{H^1_x} \lesssim \frac{N(T)^3}{\jt^\thf}, \quad 0 \leq t \leq T.
 \end{align*}
 Thus, we have uniformly for all times $0 \leq t \leq T$ that
 \begin{align*}
  \pt \Bigl( \jxi^{\thf} \hat{f}(t, \xi) + \calR(t,\xi) \Bigr) = \frac{\beta_0}{2i} e^{-it\jap{\xi}} \jap{\xi}^{\hf} \calF\bigl[ u(t, \cdot)^3 \bigr](\xi) + \calO \biggl( \frac{N(T)^2}{\jt^{\thf-\delta}} \biggr),
 \end{align*}
 and it remains to analyze the contribution of the constant coefficient cubic term 
 \begin{equation*}
  \frac{\beta_0}{2i} e^{-it\jap{\xi}} \jap{\xi}^{\hf} \calF\bigl[ u(t, \cdot)^3 \bigr](\xi).
 \end{equation*}
Inserting $u(t) = v(t) + \bar{v}(t) = e^{+it\jn} f(t) + e^{-it\jn} \overline{f(t)}$, we find that
\begin{equation} \label{equ:ft_const_cubic_decom_osc_integrals}
 \begin{aligned}
  \jxi^{\hf} e^{-it\jap{\xi}} \calF\bigl[ u(t, \cdot)^3 \bigr](\xi) &= \frac{\jxi^{\hf}}{2 \pi} \iint e^{i t \phi_1(\xi, \eta, \sigma)} \hat{f}(t, \xi-\eta-\sigma) \hat{f}(t, \eta) \hat{f}(t, \sigma) \, \ud \eta \, \ud \sigma \\
  &\quad + \frac{3 \jxi^{\hf}}{2\pi} \iint e^{i t \phi_2(\xi, \eta, \sigma)} \hat{f}(t, \xi-\eta-\sigma) \hat{\bar{f}}(t, \eta) \hat{f}(t, \sigma) \, \ud \eta \, \ud \sigma \\
  &\quad + \frac{3 \jxi^{\hf}}{2\pi} \iint e^{i t \phi_3(\xi, \eta, \sigma)} \hat{f}(t, \xi-\eta-\sigma) \hat{\bar{f}}(t, \eta) \hat{\bar{f}}(t, \sigma) \, \ud \eta \, \ud \sigma \\
  &\quad + \frac{\jxi^{\hf}}{2\pi} \iint e^{i t \phi_4(\xi, \eta, \sigma)} \hat{\bar{f}}(t, \xi-\eta-\sigma) \hat{\bar{f}}(t, \eta) \hat{\bar{f}}(t, \sigma) \, \ud \eta \, \ud \sigma \\
  &\equiv I + II + III + IV,
 \end{aligned}
\end{equation}
where we introduced the phase functions
\begin{align*}
 \phi_1(\xi, \eta, \sigma) &:= -\jxi + \jap{\xi-\eta-\sigma} + \jap{\eta} + \jap{\sigma}, \\
 \phi_2(\xi, \eta, \sigma) &:= -\jxi + \jap{\xi-\eta-\sigma} - \jap{\eta} + \jap{\sigma}, \\
 \phi_3(\xi, \eta, \sigma) &:= -\jxi + \jap{\xi-\eta-\sigma} - \jap{\eta} - \jap{\sigma}, \\
 \phi_4(\xi, \eta, \sigma) &:= -\jxi - \jap{\xi-\eta-\sigma} - \jap{\eta} - \jap{\sigma}.
\end{align*}
The long-time behavior of the oscillatory integrals on the right-hand side of~\eqref{equ:ft_const_cubic_decom_osc_integrals} is governed by the stationary points of the phases (in $\eta$ and $\sigma$, and in $t$ after time integration). A short computation reveals that the stationary points (in $\eta$ and $\sigma$) of the phase functions are given by
\begin{align*}
 \partial_\eta \phi_i = \partial_\sigma \phi_i = 0 \quad \Leftrightarrow \quad (\eta, \sigma) = (\eta_i, \sigma_i), \quad 1 \leq i \leq 4,
\end{align*}
where 
\begin{align*}
 (\eta_1, \sigma_1) &= \Bigl( \frac{\xi}{3}, \frac{\xi}{3} \Bigr), \\
 (\eta_2, \sigma_2) &= \bigl( -\xi, \xi \bigr), \\
 (\eta_3, \sigma_3) &= \bigl( \xi, \xi \bigr), \\
 (\eta_4, \sigma_4) &= \Bigl( \frac{\xi}{3}, \frac{\xi}{3} \Bigr).
\end{align*}
Moreover, we calculate that
\begin{align*}
 \phi_1(\xi, \eta_1, \sigma_1) &= - \jxi + 3 \jap{ \textstyle{\frac{\xi}{3}} } = \calO \bigl( \jxi^{-1} \bigr), \\
 \phi_2(\xi, \eta_2, \sigma_2) &= 0, \\
 \phi_3(\xi, \eta_3, \sigma_3) &= - 2 \jxi, \\
 \phi_4(\xi, \eta_4, \sigma_4) &= - \jxi - 3 \jap{ \textstyle{\frac{\xi}{3}} },
\end{align*}
and that
\begin{align*}
 \det \Hess_{\eta, \sigma} \, \phi_1(\xi, \eta_1, \sigma_1) &= 3 \jap{\textstyle{\frac{\xi}{3}}}^{-6}, &\sign \Hess_{\eta, \sigma} \, \phi_1(\xi, \eta_1, \sigma_1) &= 2, \\
 \det \Hess_{\eta, \sigma} \, \phi_2(\xi, \eta_2, \sigma_2) &= - \jxi^{-6}, &\sign \Hess_{\eta, \sigma} \, \phi_2(\xi, \eta_2, \sigma_2) &= 0, \\
 \det \Hess_{\eta, \sigma} \, \phi_3(\xi, \eta_3, \sigma_3) &= - \jxi^{-6}, &\sign \Hess_{\eta, \sigma} \, \phi_3(\xi, \eta_3, \sigma_3) &= 0, \\
 \det \Hess_{\eta, \sigma} \, \phi_4(\xi, \eta_4, \sigma_4) &= 3 \jap{\textstyle{\frac{\xi}{3}}}^{-6}, &\sign \Hess_{\eta, \sigma} \, \phi_4(\xi, \eta_4, \sigma_4) &= -2. 
\end{align*}
The stationary phase analysis of the oscillatory integrals $I$--$IV$ on the right-hand side of~\eqref{equ:ft_const_cubic_decom_osc_integrals} proceeds in exactly the same manner for each term.
We therefore only provide below the details for the crucial term $II$, which governs the long-time behavior of the solution $v(t)$ since it does not exhibit additional time oscillations. The treatment of the other terms is left to the reader. The final outcome is that there exists a small constant $0 < \nu \ll 1$ such that uniformly for all $1 \leq t \leq T$ we have 
\begin{equation*}
 \begin{aligned}
  I &= \frac{1}{t} \frac{i}{2 \sqrt{3}} \jxi^{\hf} \jap{{\textstyle \frac{\xi}{3}}}^3 e^{it(-\jxi + 3 \jap{\frac{\xi}{3}})} \hat{f}(t, {\textstyle \frac{\xi}{3}})^3 + \calO \biggl( \frac{N(T)^3}{t^{1+\nu}} \biggr), \\
  II &= \frac{3}{t} \jxi^{\frac{7}{2}} \bigl| \hat{f}(t, \xi) \bigr|^2 \hat{f}(t,\xi) + \calO \biggl( \frac{N(T)^3}{t^{1+\nu}} \biggr), \\
  III &= \frac{3}{t} \jxi^{\frac{7}{2}} e^{-2 i t \jxi} \bigl| \hat{f}(t, -\xi) \bigr|^2 \bar{\hat{f}}(t,-\xi) + \calO \biggl( \frac{N(T)^3}{t^{1+\nu}} \biggr), \\
  IV &= - \frac{1}{t} \frac{i}{2 \sqrt{3}} \jap{\xi}^{\hf} \jap{{\textstyle \frac{\xi}{3}}}^3 e^{-it(\jxi + 3 \jap{\frac{\xi}{3}})} \hat{\bar{f}}(t, {\textstyle \frac{\xi}{3}})^3 + \calO \biggl( \frac{N(T)^3}{t^{1+\nu}} \biggr).
 \end{aligned}
\end{equation*}

\medskip 

\noindent \underline{\it Stationary phase analysis of the oscillatory integral II:}
We consider the case where $|\xi| \sim 2^j$ for some $j \gg 1$, noting that the analysis for $|\xi| \lesssim 1$ is analogous, but does not require a refined treatment of the smaller frequencies. 
Let $\{ \psi_\ell \}_{\ell \geq 0}$ be a smooth partition of unity so that $\sum_{\ell \geq 0} \psi_\ell(\eta) = 1$ for all $\eta \in \bbR$
and with the property that $\psi_\ell$ is supported on $\{ |\eta| \sim 2^\ell \}$ for $\ell \geq 1$ and on $\{ |\eta| \lesssim 1 \}$ for $\ell = 0$.
Then we decompose
\begin{align*}
 &\jxi^{\hf} \iint e^{i t \phi_2(\xi, \eta, \sigma)} \hat{f}(t, \xi-\eta-\sigma) \hat{\bar{f}}(t, \eta) \hat{f}(t, \sigma) \, \ud \eta \, \ud \sigma \\
 &\quad \quad = \sum_{k, \ell \geq 0} \underbrace{ \jxi^{\hf} \iint e^{i t \phi_2(\xi, \eta, \sigma)} \hat{f}(t, \xi-\eta-\sigma) \hat{\bar{f}}(t, \eta) \hat{f}(t, \sigma) \psi_k(\eta) \psi_\ell(\sigma) \, \ud \eta \, \ud \sigma }_{J_{k\ell}} \\
 &\quad \quad = \underbrace{ \sum_{k \geq j+10} \sum_{0 \leq \ell \leq k-5} J_{k\ell} }_{\calJ^{(1)}} + \underbrace{ \sum_{\ell \geq j+10} \sum_{0 \leq k \leq \ell-5} J_{k\ell} }_{\calJ^{(2)}} + \underbrace{ \sum_{k \geq j+10} \sum_{|\ell-k| < 5} J_{k\ell} }_{\calJ^{(3)}} + \underbrace{ \sum_{0 \leq k, \ell \lesssim j} J_{k \ell} }_{\calJ^{(4)}}.  
\end{align*}
The only stationary point $(\eta, \sigma) = (-\xi, \xi)$ of the phase $\phi_2(\xi, \eta, \sigma)$ is contained in the region $|\eta|, |\sigma| \lesssim 2^j$. Correspondingly, the terms $\calJ^{(1)}$, $\calJ^{(2)}$, and $\calJ^{(3)}$ can just be estimated by integrating by parts either in $\eta$ or in $\sigma$, while the term $\calJ^{(4)}$ requires a stationary phase analysis.
Below we use the short-hand notation $\widehat{f_{\sim k}}(t,\eta)$ to indicate localization of $\hat{f}(t,\eta)$ to frequencies $|\eta| \sim 2^k$ for $k \geq 1$ and to $|\eta| \lesssim 1$ for $k = 0$. 

\medskip 

\noindent {\it Contribution of the term $\calJ^{(1)}$:} 
On the supports of the integrands of the terms $J_{k\ell}$ in the sum $\calJ^{(1)}$, we have $|\eta| \sim 2^k \gg 2^j \sim |\xi|$ and $|\sigma| \sim 2^\ell \ll 2^k \sim |\eta|$. Thus, there holds $|\xi - \eta - \sigma| \sim 2^k$ and we may write 
\begin{align*}
 \calJ^{(1)} &= \sum_{k \geq j+10} \sum_{0 \leq \ell \leq k-5} \jxi^{\hf} \iint e^{i t \phi_2(\xi, \eta, \sigma)} \widehat{f_{\sim k}}(t, \xi-\eta-\sigma) \widehat{\bar{f}_{\sim k}}(t, \eta) \widehat{f_{\sim \ell}}(t, \sigma) \psi_k(\eta) \psi_\ell(\sigma) \, \ud \eta \, \ud \sigma.
\end{align*}
Moreover, observe that in view of the identity
\begin{equation*}
 \frac{1}{\partial_\sigma \phi_2} = \frac{\jap{\sigma}\jap{\xi-\eta-\sigma} ( \sigma \jap{\xi-\eta-\sigma} + (\xi-\eta-\sigma) \jap{\sigma})}{\sigma^2 - (\xi-\eta-\sigma)^2}
\end{equation*}
we have on the supports of the integrands of $J_{k\ell}$ that
\begin{equation} \label{equ:IIJ2_bound_phase_inverse}
 \Bigl| \partial_\eta^{m_1} \partial_\sigma^{m_2} \frac{1}{\partial_\sigma \phi_2(\xi, \eta, \sigma)} \Bigr| \lesssim 2^{+2\ell} 2^{-(m_1+m_2) \ell} \quad \text{ for } 0 \leq \ell \leq k-5, \, k \geq j+10,
\end{equation}
for all integers $0 \leq m_1, m_2 \leq 10$. We therefore integrate by parts in $\sigma$ to obtain that
\begin{align*}
 \calJ^{(1)} &= \calJ^{(1)}_{(a)} + \calJ^{(1)}_{(b)} + \calJ^{(1)}_{(c)},
\end{align*}
where
\begin{align*}
 \calJ^{(1)}_{(a)} &:= - \sum_{k \geq j+10} \sum_{0 \leq \ell \leq k-5} \jxi^{\hf} \iint e^{i t \phi_2(\xi, \eta, \sigma)} \frac{1}{it \partial_\sigma \phi_2} \partial_\sigma \widehat{f_{\sim k}}(t, \xi-\eta-\sigma) \widehat{\bar{f}_{\sim k}}(t, \eta) \widehat{f_{\sim \ell}}(t, \sigma) \psi_k(\eta) \psi_\ell(\sigma) \, \ud \eta \, \ud \sigma, \\
 \calJ^{(1)}_{(b)} &:= - \sum_{k \geq j+10} \sum_{0 \leq \ell \leq k-5} \jxi^{\hf} \iint e^{i t \phi_2(\xi, \eta, \sigma)} \frac{1}{it \partial_\sigma \phi_2} \widehat{f_{\sim k}}(t, \xi-\eta-\sigma) \widehat{\bar{f}_{\sim k}}(t, \eta) \partial_\sigma \widehat{f_{\sim \ell}}(t, \sigma) \psi_k(\eta) \psi_\ell(\sigma) \, \ud \eta \, \ud \sigma, \\
 \calJ^{(1)}_{(c)} &:= - \sum_{k \geq j+10} \sum_{0 \leq \ell \leq k-5} \jxi^{\hf} \iint e^{i t \phi_2(\xi, \eta, \sigma)} \partial_\sigma \Bigl[  \frac{1}{it \partial_\sigma \phi_2} \psi_k(\eta) \psi_\ell(\sigma) \Bigr] \widehat{f_{\sim k}}(t, \xi-\eta-\sigma) \widehat{\bar{f}_{\sim k}}(t, \eta) \widehat{f_{\sim \ell}}(t, \sigma) \, \ud \eta \, \ud \sigma.
\end{align*}
In view of~\eqref{equ:IIJ2_bound_phase_inverse} we have 
\begin{align*}
 \biggl\| \iint e^{ix\eta} e^{iy\sigma} \frac{1}{\partial_\sigma \phi_2} \psi_k(\eta) \psi_\ell(\sigma) \, \ud \eta \, \ud \sigma \biggr\|_{L^1_{x,y}(\bbR\times\bbR)} \lesssim 2^k 2^{\ell} \quad \text{ for } 0 \leq \ell \leq k-5, \, k \geq j+10.
\end{align*}
Hence, by Lemma~\ref{lem:pseudoproduct_op_bound} we may bound 
\begin{align*}
 \bigl|\calJ^{(1)}_{(a)}\bigr| + \bigl|\calJ^{(1)}_{(b)}\bigr| &\lesssim \frac{2^{\frac{1}{2}j}}{t} \sum_{k \geq j+10} \sum_{0 \leq \ell \leq k-5} 2^k 2^{\ell} \Bigl( 2^{-2k} \bigl\| \jn L v(t) \bigr\|_{L^2_x} 2^{-2k} \bigl\| \jn^2 v(t) \bigr\|_{L^2_x} \|v_{\sim \ell}(t)\|_{L^\infty_x} \\
 &\qquad \qquad \qquad \qquad \qquad \qquad \qquad + 2^{-2k} \bigl\|\jn^2 v(t)\bigr\|_{L^2_x} \|v_{\sim k}(t)\|_{L^\infty_x} 2^{-2\ell} \bigl\| \jn L v(t) \bigr\|_{L^2_x} \Bigr) \\
 &\lesssim \frac{2^{\frac{1}{2}j}}{t} \sum_{k \geq j+10} \sum_{0 \leq \ell \leq k-5} 2^{-k} 2^{-\ell} \bigl\| \jn L v(t) \bigr\|_{L^2_x} \bigl\| \jn^2 v(t) \bigr\|_{L^2_x} \|v(t)\|_{L^\infty_x} \\
 &\lesssim \frac{N(T)^3}{t^{\frac{3}{2}-2\delta}} 2^{-\frac{1}{2} j} \lesssim \frac{N(T)^3}{t^{\frac{3}{2}-2\delta}}.
\end{align*}
Analogously, we obtain that
\begin{align*}
 \bigl|\calJ^{(1)}_{(c)}\bigr| \lesssim \frac{N(T)^3}{t^{\frac{3}{2}-2\delta}}.
\end{align*}

\medskip 

\noindent {\it Contribution of the term $\calJ^{(2)}$:} 
On the supports of the integrands of $J_{k\ell}$ in $\calJ^{(2)}$, we have $|\sigma| \sim 2^\ell \gg 2^j$ and $|\eta| \sim 2^k \ll 2^\ell \sim |\sigma|$. Hence, there holds $|\xi-\eta-\sigma| \sim |\sigma| \sim 2^\ell$ and on the supports of the integrands we have 
\begin{equation*} 
 \Bigl| \partial_\eta^{m_1} \partial_\sigma^{m_2} \frac{1}{\partial_\eta \phi_2(\xi, \eta, \sigma)} \Bigr| \lesssim 2^{+2k} 2^{-(m_1+m_2) k} \quad \text{for } 0 \leq k \leq \ell - 5, \, \ell \geq j+10,
\end{equation*}
for all integers $0 \leq m_1, m_2 \leq 10$.
We integrate by parts in $\eta$ and then proceed in the same manner as for the sum $\calJ^{(1)}$ to get $|\calJ^{(2)}| \lesssim N(T)^3 t^{-(\frac{3}{2}-2\delta)}$.

\medskip 

\noindent {\it Contribution of the term $\calJ^{(3)}$:} 
On the supports of the integrands of $J_{k\ell}$ in $\calJ^{(3)}$, we have $|\eta| \sim |\sigma| \sim 2^k \gg 2^j$, which means that $|\xi-\eta-\sigma|$ can possibly become small. For this reason we additionally decompose 
\begin{equation*}
 \hat{f}(t, \xi-\eta-\sigma) = \sum_{0 \leq n \lesssim k} \widehat{f_{\sim n}}(t, \xi - \eta - \sigma).
\end{equation*}
In this case $\partial_\eta \phi_2$ cannot vanish and satisfies suitable bounds. Indeed, if $\eta$ and $\xi-\eta-\sigma$ have the same sign, we find that
\begin{align*}
 \bigl| \partial_\eta \phi_2(\xi, \eta, \sigma) \bigr| = \biggl| -\frac{\xi-\eta-\sigma}{\jap{\xi-\eta-\sigma}} - \frac{\eta}{\jap{\eta}} \biggr| \geq \frac{|\eta|}{\jap{\eta}} \gtrsim 1.
\end{align*}
Instead, if $\eta$ and $\xi-\eta-\sigma$ have opposite signs, we have 
\begin{align*}
 \bigl| \eta - (\xi-\eta-\sigma) \bigr| \geq |\eta| \sim 2^k.
\end{align*}
Since $|\xi -\sigma| \sim 2^k$, it follows from 
\begin{align*}
 \frac{1}{\partial_\eta \phi_2} &=  - \frac{\jap{\eta}\jap{\xi-\eta-\sigma} (\eta \jap{\xi-\eta-\sigma} - (\xi-\eta-\sigma) \jap{\eta})}{ (\xi-\sigma) (\eta-(\xi-\eta-\sigma)) } 
\end{align*}
that overall we have 
\begin{equation} \label{equ:IIJ3_bound_phase_inverse}
 \Bigl| \partial_\eta^{m_1} \partial_\sigma^{m_2} \frac{1}{\partial_\eta \phi_2(\xi, \eta, \sigma)} \Bigr| \lesssim 2^{+2n} 2^{-(m_1+m_2) n} \quad \text{ for } 0 \leq n \leq k+5, \, k \geq j+10, \, |\ell-k| < 5,
\end{equation}
for all integers $0 \leq m_1, m_2 \leq 10$.
We may therefore integrate by parts in $\eta$ to find that 
\begin{align*}
 \calJ^{(3)} &= \calJ^{(3)}_{(a)} + \calJ^{(3)}_{(b)} + \calJ^{(3)}_{(c)},
\end{align*}
where
\begin{align*}
 \calJ^{(3)}_{(a)} &:= - \sum_{\substack{k \geq j+10 \\ |\ell-k| < 5}} \sum_{0 \leq n \lesssim k} \jxi^{\hf} \iint e^{i t \phi_2(\xi, \eta, \sigma)} \frac{1}{it \partial_\eta \phi_2} \partial_\eta \widehat{f_{\sim n}}(t, \xi-\eta-\sigma) \widehat{\bar{f}_{\sim k}}(t, \eta) \widehat{f_{\sim \ell}}(t, \sigma) \psi_k(\eta) \psi_\ell(\sigma) \, \ud \eta \, \ud \sigma, \\
 \calJ^{(3)}_{(b)} &:= - \sum_{\substack{k \geq j+10 \\ |\ell-k| < 5}} \sum_{0 \leq n \lesssim k} \jxi^{\hf} \iint e^{i t \phi_2(\xi, \eta, \sigma)} \frac{1}{it \partial_\eta \phi_2} \widehat{f_{\sim n}}(t, \xi-\eta-\sigma) \partial_\eta \widehat{\bar{f}_{\sim k}}(t, \eta)  \widehat{f_{\sim \ell}}(t, \sigma) \psi_k(\eta) \psi_\ell(\sigma) \, \ud \eta \, \ud \sigma, \\
 \calJ^{(3)}_{(c)} &:= - \sum_{\substack{k \geq j+10 \\ |\ell-k| < 5}} \sum_{0 \leq n \lesssim k} \jxi^{\hf} \iint e^{i t \phi_2(\xi, \eta, \sigma)} \partial_\eta \Bigl[  \frac{1}{it \partial_\eta \phi_2} \psi_k(\eta) \psi_\ell(\sigma) \Bigr] \widehat{f_{\sim n}}(t, \xi-\eta-\sigma) \widehat{\bar{f}_{\sim k}}(t, \eta) \widehat{f_{\sim \ell}}(t, \sigma) \, \ud \eta \, \ud \sigma.
\end{align*}
In view of~\eqref{equ:IIJ3_bound_phase_inverse} we have 
\begin{align*}
 \biggl\| \iint e^{ix\eta} e^{iy\sigma} \frac{1}{\partial_\eta \phi_2} \psi_k(\eta) \psi_\ell(\sigma) \, \ud \eta \, \ud \sigma \biggr\|_{L^1_{x,y}(\bbR\times\bbR)} \lesssim 2^k 2^{\ell} \quad \text{ for } 0 \leq n \leq k+5, \, k \geq j+10, \, |\ell-k| < 5.
\end{align*}
Hence, by Lemma~\ref{lem:pseudoproduct_op_bound} we may bound 
\begin{align*}
 \bigl|\calJ^{(3)}_{(a)}\bigr| + \bigl|\calJ^{(3)}_{(b)}\bigr| &\lesssim \frac{2^{\frac{1}{2}j}}{t} \sum_{\substack{k \geq j+10 \\ |\ell-k| < 5}} \sum_{0 \leq n \lesssim k} 2^k 2^{\ell} \Bigl( 2^{-2n} \bigl\| \jn L v(t) \bigr\|_{L^2_x} 2^{-2k} \bigl\| \jn^2 v(t) \bigr\|_{L^2_x} \|v_{\sim \ell}(t)\|_{L^\infty_x} \\
 &\qquad \qquad \qquad \qquad \qquad \qquad \qquad + 2^{-2n} \bigl\|\jn^2 v(t)\bigr\|_{L^2_x} 2^{-2k} \bigl\| \jn L v(t) \bigr\|_{L^2_x} \|v_{\sim \ell}(t)\|_{L^\infty_x} \Bigr).
\end{align*}
Finally, using that $\|v_{\sim \ell}(t)\|_{L^\infty_x} \lesssim 2^{-\frac{3}{4} \ell} \bigl\| \jn^2 v(t) \bigr\|_{L^2_x}^\hf \|v(t)\|_{L^\infty_x}^\hf$ for any $\ell \geq 0$,
we arrive at the estimate 
\begin{align*}
 \bigl|\calJ^{(1)}_{(a)}\bigr| + \bigl|\calJ^{(1)}_{(b)}\bigr| &\lesssim \frac{2^{\frac{1}{2}j}}{t} \sum_{\substack{k \geq j+10 \\ |\ell-k| < 5}} \sum_{0 \leq n \lesssim k} 2^k 2^\ell 2^{-2n} 2^{-2k} 2^{-\frac{3}{4} \ell} \bigl\| \jn L v(t) \bigr\|_{L^2_x} \bigl\| \jn^2 v(t) \bigr\|_{L^2_x}^{\thf} \|v(t)\|_{L^\infty_x}^\hf \\
 &\lesssim \frac{N(T)^3}{t^{\frac{5}{4}-\frac{5}{2}\delta}} 2^{-\frac{1}{4} j} \lesssim \frac{N(T)^3}{t^{\frac{5}{4}-\frac{5}{2}\delta}}.
\end{align*}
In the same manner, we derive that $|\calJ^{(1)}_{(c)}| \lesssim N(T)^3 t^{-(\frac{5}{4} - \frac{5}{2}\delta)}$.

\medskip

\noindent {\it Contribution of the term $\calJ^{(4)}$:} 
Here the strategy is to cut out a sufficiently small neighborhood around the stationary point $(\eta, \sigma) = (-\xi, \xi)$, where we have a suitable lower bound on the determinant of the Hessian of the phase $\phi_2$ to apply the stationary phase Lemma~\ref{lem:stationary_phase}. Outside that neighborhood we can again just integrate by parts in $\eta$ or $\sigma$. 
To this end we introduce smooth bump functions $\chi_c, \chi_{8c} \in C^\infty$ for some sufficiently small absolute constant $0 < c \ll 1$ such that $\chi_c(\zeta) = 1$ for $|\zeta-1| \leq c$ and $\chi_c(\zeta) = 0$ for $|\zeta-1| \geq 2c$, respectively such that $\chi_{8c}(\zeta) = 1$ for $|\zeta-1| \leq 8c$ and $\chi_{8c}(\zeta) = 0$ for $|\zeta-1| \geq 16 c$. Then we decompose
\begin{align*}
 \calJ^{(4)} &= \jxi^{\hf} \iint e^{i t \phi_2(\xi, \eta, \sigma)} \widehat{f_{\sim j}}(t, \xi-\eta-\sigma) \widehat{\bar{f}_{\sim j}}(t, \eta) \widehat{f_{\sim j}}(t, \sigma) \chi_c\Bigl(\frac{\sigma}{\xi}\Bigr) \chi_{8c}\Bigl(\frac{\eta}{-\xi}\Bigr) \, \ud \eta \, \ud \sigma \\
 &\quad + \sum_{0 \leq k, \ell \lesssim j} \jxi^{\hf} \iint e^{i t \phi_2(\xi, \eta, \sigma)} \hat{f}(t, \xi-\eta-\sigma) \hat{\bar{f}}(t, \eta) \hat{f}(t, \sigma) \Bigl( 1 - \chi_c\Bigl(\frac{\sigma}{\xi}\Bigr) \Bigr) \psi_k(\eta) \psi_\ell(\sigma) \, \ud \eta \, \ud \sigma \\
 &\quad + \sum_{0 \leq k, \ell \lesssim j} \jxi^{\hf} \iint e^{i t \phi_2(\xi, \eta, \sigma)} \hat{f}(t, \xi-\eta-\sigma) \hat{\bar{f}}(t, \eta) \hat{f}(t, \sigma) \chi_c\Bigl(\frac{\sigma}{\xi}\Bigr) \Bigl( 1 - \chi_{8c}\Bigl(\frac{\eta}{-\xi}\Bigr) \Bigr)  \psi_k(\eta) \psi_\ell(\sigma) \, \ud \eta \, \ud \sigma \\
 &\equiv \calJ^{(4)}_{(a)} + \calJ^{(4)}_{(b)} + \calJ^{(4)}_{(c)}.
\end{align*}

\noindent {\it Contribution of the term $\calJ^{(4)}_{(a)}$}:
On the support of $\chi_c (\frac{\sigma}{\xi} ) \chi_{8c} (\frac{\eta}{-\xi} )$ we have a suitable lower bound on the determinant of the Hessian of the phase function given by
\begin{equation*}
 \bigl| \det \Hess \, \phi_2(\xi, \eta, \sigma) \bigr| \gtrsim \jap{\xi}^{-6} \gtrsim 2^{-6j}, \quad \text{ for } |\sigma - \xi| \leq c |\xi|, \, |\eta + \xi| \leq 8 c |\xi|, \, |\xi| \sim 2^j.
\end{equation*}
Changing variables to $(\xi', \eta', \sigma') := 2^{-j} (\xi, \eta, \sigma)$, we may write
\begin{equation*}
 \calJ^{(4)}_{(a)} = 2^{2j} \jxi^\hf \iint e^{i 2^{3j} t \psi(\xi', \eta', \sigma')} F(t, \eta', \sigma') \chi_c\Bigl(\frac{\sigma'}{\xi'}\Bigr) \chi_{8c}\Bigl(\frac{\eta'}{-\xi'}\Bigr) \, \ud \eta' \, \ud \sigma' 
\end{equation*}
with 
\begin{align*}
 \psi(\xi', \eta', \sigma') &:= 2^{-3j} \phi_2(2^j \xi', 2^j \eta', 2^j \sigma'), \\
 F(t, \eta', \sigma') &:= \widehat{f_{\sim j}}(t, 2^j(\xi' - \eta' - \sigma')) \widehat{\bar{f}_{\sim j}}(t, 2^j \eta') \widehat{f_{\sim j}}(t, 2^j \sigma').
\end{align*}
Correspondingly, on the support of $\chi_c (\frac{\sigma'}{\xi'} ) \chi_{8c} (\frac{\eta'}{-\xi'} )$ we have the lower bound
\begin{equation*}
 \bigl| \det \Hess \, \psi(\xi', \eta', \sigma') \bigr| = 2^{-2j} \bigl| \det \Hess \, \phi_2(\xi, \eta, \sigma) \bigr| \gtrsim 2^{-8j}.
\end{equation*}
Applying Lemma~\ref{lem:stationary_phase} with $\Delta = 2^{-2j} \jxi^{-6}$, $\lambda = 2^{3j} t$, and $\mu = 2^{-8j}$, 
we obtain for any $0 < \alpha \leq 1$ that
\begin{align*}
 \calJ^{(4)}_{(a)} = \frac{2\pi}{t} \jxi^{\hf} \jxi^3 \bigl| \hat{f}(t,\xi) \bigr|^2 \hat{f}(t,\xi) + 2^{2j} \jxi^{\hf} \calO \left( \frac{ \bigl\| \langle (x,y) \rangle^{2 \alpha} \widehat{F}(t) \bigr\|_{L^1_{x,y}}}{ (2^{-8j})^{\hf + 2\alpha} (2^{3j} t)^{1+\alpha}} \right).
\end{align*}
Now observe that
\begin{align*}
 \widehat{F}(t,x,y) = 2^{-3j} \int e^{-iz\xi'} f_{\sim j}(t, 2^{-j} z) \bar{f}_{\sim j}(t, 2^{-j}(z-x)) f_{\sim j}(t, 2^{-j}(z-y)) \, \ud z,
\end{align*}
which implies
\begin{equation} \label{equ:nonresonant_ODE_deriv_profileb1}
 \bigl\| \widehat{F}(t) \bigr\|_{L^1_{x,y}} \lesssim \|f_{\sim j}(t)\|_{L^1_x}^3 \quad \text{ and } \quad \bigl\| |(x,y)|^{2\alpha} \widehat{F}(t) \bigr\|_{L^1_{x,y}} \lesssim 2^{2\alpha j} \|f_{\sim j}(t)\|_{L^1_x}^2 \bigl\| |x|^{2\alpha} f_{\sim j}(t) \bigr\|_{L^1_x}.
\end{equation}
From Proposition~\ref{prop:slow_growth_H2v} and Proposition~\ref{prop:slow_growth_H1L} we obtain the following bounds on the profiles 
\begin{equation} \label{equ:nonresonant_ODE_deriv_profileb2}
\begin{aligned}
 \|f_{\sim j}(t)\|_{L^1_x} \lesssim 2^{-2j} \bigl\| \jn^2 f(t) \bigr\|_{L^1_x} &\lesssim 2^{-2j} \bigl\| \jn^2 f(t) \bigr\|_{L^2_x}^\hf \bigl\| x \jn^2 f(t) \bigr\|_{L^2_x}^{\hf} \\
 &\lesssim 2^{-2j} \bigl( \bigl\| \jn^2 v(t) \bigr\|_{L^2_x} + \bigl\| \jn L v(t) \bigr\|_{L^2_x} \bigr) \\
 &\lesssim 2^{-2j} N(T) \jt^{+\delta}
\end{aligned}
\end{equation}
and for $0 \leq \alpha < \frac{1}{4}$, 
\begin{equation} \label{equ:nonresonant_ODE_deriv_profileb3}
 \begin{aligned}
  \bigl\| |x|^{2\alpha} f_{\sim j}(t) \bigr\|_{L^1_x} \lesssim \bigl\| \jx f_{\sim j}(t) \bigr\|_{L^2_x} &\lesssim \|f_{\sim j}(t)\|_{L^2_x} + \|x f_{\sim j}(t)\|_{L^2_x} \\
  &\lesssim \|f(t)\|_{L^2_x} + \|x f(t)\|_{L^2_x} + 2^{-j} \|f(t)\|_{L^2_x} \\
  &\lesssim \bigl\| \jn^2 v(t) \bigr\|_{L^2_x} + \bigl\| \jn L v(t) \bigr\|_{L^2_x} \\
  &\lesssim N(T) \jt^{+\delta}.  
 \end{aligned}
\end{equation}
Combining \eqref{equ:nonresonant_ODE_deriv_profileb1}--\eqref{equ:nonresonant_ODE_deriv_profileb3} yields for $0 \leq \alpha < \frac{1}{4}$, 
\begin{align*}
 \bigl\| \langle (x,y) \rangle^{2 \alpha} \widehat{F}(t) \bigr\|_{L^1_{x,y}} \lesssim 2^{-(4-2\alpha)j} N(T)^3 \jt^{+3\delta}.
\end{align*}
For sufficiently small $0 < \alpha \ll 1$, it follows that
\begin{align*}
 \Bigl\| \calJ^{(4)}_{(a)} - \frac{2\pi}{t} \jxi^{\hf} \jxi^3 \bigl| \hat{f}(t,\xi) \bigr|^2 \hat{f}(t,\xi) \Bigr\|_{L^\infty_\xi} &\lesssim 2^{(\frac{7}{2} + 13 \alpha) j} 2^{-(4-2\alpha)j} \frac{N(T)^3}{t^{1+\alpha-3\delta}} \lesssim \frac{N(T)^3}{t^{1+\alpha-3\delta}}.
\end{align*}

\medskip 

\noindent {\it Contribution of the term $\calJ^{(4)}_{(b)}$}:
In this case, $\partial_\eta \phi_2$ cannot vanish and satisfies suitable bounds. We additionally decompose 
\begin{equation*}
 \hat{f}(t, \xi-\eta-\sigma) = \sum_{0 \leq n \lesssim j} \widehat{f_{\sim n}}(t, \xi - \eta - \sigma).
\end{equation*}
If $\eta$ and $\xi-\eta-\sigma$ have the same sign, we find that
\begin{align*}
 \bigl| \partial_\eta \phi_2(\xi, \eta, \sigma) \bigr| = \biggl| - \frac{\xi-\eta-\sigma}{\jap{\xi-\eta-\sigma}} - \frac{\eta}{\jap{\eta}} \biggr| \geq \max \biggl\{ \frac{|\xi-\eta-\sigma|}{\jap{\xi-\eta-\sigma}}, \frac{\eta}{\jap{\eta}} \biggr\} \gtrsim 1.
\end{align*}
Instead, if $\eta$ and $\xi-\eta-\sigma$ have opposite signs, we have 
\begin{align*}
 \bigl| \eta - (\xi-\eta-\sigma) \bigr| \geq \max \bigl\{ |\eta|, |\xi-\eta-\sigma| \bigr\} \gtrsim \max \, \{ 2^k, 2^n \}.
\end{align*}
Since $|\xi -\sigma| \geq c |\xi| \gtrsim 2^j$ in view of the cut-off $(1-\chi_c(\frac{\sigma}{\xi}))$, it follows from 
\begin{align*}
 \frac{1}{\partial_\eta \phi_2} &=  - \frac{\jap{\eta}\jap{\xi-\eta-\sigma} (\eta \jap{\xi-\eta-\sigma} - (\xi-\eta-\sigma) \jap{\eta})}{ (\xi-\sigma) (\eta-(\xi-\eta-\sigma)) } 
\end{align*}
that overall we have in this case 
\begin{equation} \label{equ:J4b_bound_phase_inverse}
 \Bigl| \partial_\eta^{m_1} \partial_\sigma^{m_2} \frac{1}{\partial_\eta \phi_2(\xi, \eta, \sigma)} \Bigr| \lesssim 2^{2k} 2^{2n} 2^{-j} 2^{-\max\{k, n\}} 2^{-(m_1+m_2) \min\{k, n\}}
\end{equation}
for all integers $0 \leq m_1, m_2 \leq 10$.
We may therefore integrate by parts in $\eta$ to obtain that 
\begin{align*}
 \calJ^{(4)}_{(b)} = \calJ^{(4)}_{(b1)} + \calJ^{(4)}_{(b2)} + \calJ^{(4)}_{(b3)} 
\end{align*}
where 
\begin{align*}
 \calJ^{(4)}_{(b1)} &= - \sum_{0 \leq k, \ell, n \lesssim j} \jxi^{\hf} \iint e^{i t \phi_2(\xi, \eta, \sigma)} \frac{1}{i t \partial_\eta \phi_2} \partial_\eta \widehat{f_{\sim n}}(t, \xi-\eta-\sigma) \widehat{\bar{f}_{\sim k}}(t, \eta) \widehat{f_{\sim \ell}}(t, \sigma) \Bigl( 1 - \chi_c\Bigl(\frac{\sigma}{\xi}\Bigr) \Bigr) \psi_k(\eta) \psi_\ell(\sigma) \, \ud \eta \, \ud \sigma, \\
 \calJ^{(4)}_{(b2)} &= - \sum_{0 \leq k, \ell, n \lesssim j} \jxi^{\hf} \iint e^{i t \phi_2(\xi, \eta, \sigma)} \frac{1}{i t \partial_\eta \phi_2} \widehat{f_{\sim n}}(t, \xi-\eta-\sigma) \partial_\eta \widehat{\bar{f}_{\sim k}}(t, \eta) \widehat{f_{\sim \ell}}(t, \sigma) \Bigl( 1 - \chi_c\Bigl(\frac{\sigma}{\xi}\Bigr) \Bigr) \psi_k(\eta) \psi_\ell(\sigma) \, \ud \eta \, \ud \sigma, \\
 \calJ^{(4)}_{(b3)} &= - \sum_{0 \leq k, \ell, n \lesssim j} \jxi^{\hf} \iint e^{i t \phi_2(\xi, \eta, \sigma)} \partial_\eta \Bigl[ \frac{1}{i t \partial_\eta \phi_2} \Bigl( 1 - \chi_c\Bigl(\frac{\sigma}{\xi}\Bigr) \Bigr) \psi_k(\eta) \psi_\ell(\sigma) \Bigr] \widehat{f_{\sim n}}(t, \xi-\eta-\sigma) \widehat{\bar{f}_{\sim k}}(t, \eta) \widehat{f_{\sim \ell}}(t, \sigma) \, \ud \eta \, \ud \sigma.
\end{align*}
In view of~\eqref{equ:J4b_bound_phase_inverse} we have for $0 \leq k, \ell, n \lesssim j$ that
\begin{align*}
 \biggl\| \iint e^{ix\eta} e^{iy\sigma} \frac{1}{\partial_\eta \phi_2} \Bigl( 1 - \chi_c\Bigl(\frac{\sigma}{\xi}\Bigr) \Bigr) \psi_k(\eta) \psi_\ell(\sigma) \, \ud \eta \, \ud \sigma \biggr\|_{L^1_{x,y}(\bbR\times\bbR)} \lesssim 2^{3k} 2^\ell 2^{2n} 2^{-j} 2^{-\max\{k, n\}} 2^{-2 \min\{k,n\}}.
\end{align*}
Hence, by Lemma~\ref{lem:pseudoproduct_op_bound} we may bound 
\begin{align*}
 \bigl|\calJ^{(4)}_{(b1)}\bigr| + \bigl|\calJ^{(4)}_{(b2)}\bigr| &\lesssim \frac{2^{\frac{1}{2}j}}{t} \sum_{0 \leq k, \ell, n \lesssim j} 2^{3k} 2^\ell 2^{2n} 2^{-j} 2^{-\max\{k, n\}} 2^{-2 \min\{k,n\}} \times \\
 &\qquad \qquad \qquad \qquad \qquad \qquad \times 2^{-2n} 2^{-2k} \bigl\| \jn L v(t) \bigr\|_{L^2_x} \bigl\| \jn^2 v(t) \bigr\|_{L^2_x} \|v_{\sim \ell}(t)\|_{L^\infty_x}.
\end{align*}
Finally, using that $\|v_{\sim \ell}(t)\|_{L^\infty_x} \lesssim 2^{-\frac{3}{4} \ell} \bigl\| \jn^2 v(t) \bigr\|_{L^2_x}^\hf \|v(t)\|_{L^\infty_x}^\hf$ for any $\ell \geq 0$,
we arrive at the estimate 
\begin{align*}
 \bigl|\calJ^{(4)}_{(b1)}\bigr| + \bigl|\calJ^{(4)}_{(b2)}\bigr| &\lesssim \frac{2^{-\frac{1}{2}j}}{t} \sum_{0 \leq k, \ell, n \lesssim j} 2^{k} 2^{\frac{1}{4} \ell} 2^{-\max\{k, n\}} 2^{-2 \min\{k,n\}} \bigl\| \jn L v(t) \bigr\|_{L^2_x} \bigl\| \jn^2 v(t) \bigr\|_{L^2_x}^{\thf} \|v(t)\|_{L^\infty_x}^\hf \\
 &\lesssim \frac{N(T)^3}{t^{\frac{5}{4}-\frac{5}{2}\delta}} 2^{-\frac12 j} 2^{\frac{1}{4} j} 2^{(0+)j} \lesssim \frac{N(T)^3}{t^{\frac{5}{4}-\frac{5}{2}\delta}}.
\end{align*}
Analogously, we obtain that
\begin{align*}
 \bigl|\calJ^{(4)}_{(b3)}\bigr| \lesssim \frac{N(T)^3}{t^{\frac{5}{4}-\frac{5}{2}\delta}}.
\end{align*}

\medskip 

\noindent {\it Contribution of the term $\calJ^{(4)}_{(c)}$}:
In this regime $\partial_\sigma \phi_2$ cannot vanish and satisfies suitable bounds.
As before we additionally decompose 
\begin{equation*}
 \hat{f}(t, \xi-\eta-\sigma) = \sum_{0 \leq n \lesssim j} \widehat{f_{\sim n}}(t, \xi - \eta - \sigma).
\end{equation*}
If $\sigma$ and $\xi-\eta-\sigma$ have opposite signs, we easily obtain that
\begin{align*}
 \bigl| \partial_\sigma \phi_2(\xi, \eta, \sigma) \bigr| = \biggl| - \frac{\xi-\eta-\sigma}{\jap{\xi-\eta-\sigma}} + \frac{\sigma}{\jap{\sigma}} \biggr| \geq \max \biggl\{ \frac{|\xi-\eta-\sigma|}{\jap{\xi-\eta-\sigma}}, \frac{\sigma}{\jap{\sigma}} \biggr\} \gtrsim 1.
\end{align*}
Consider now the case where $\sigma$ and $\xi-\eta-\sigma$ have the same sign. Without loss of generality assume that $\xi > 0$, and thus $\sigma > 0$ in view of the cut-off $\chi_c(\frac{\sigma}{\xi})$. Then we must have $\xi - \eta - \sigma > 0$, and since $|\sigma| \sim 2^j$ due to the cut-off $\chi_c(\frac{\sigma}{\xi})$, it follows that
\begin{equation*}
 \xi - \eta > \sigma \gtrsim 2^j.
\end{equation*}
Moreover, the cut-offs $\chi_c(\frac{\sigma}{\xi})$ and $(1-\chi_{8c}(\frac{\eta}{-\xi}))$ enforce that $|\sigma - \xi| \leq 2c|\xi|$ and $|\eta+\xi| \geq 8c|\xi|$, which gives 
\begin{align*}
 |\sigma - (\xi-\eta-\sigma)| = |\eta + \xi + 2(\sigma-\xi)| \geq |\eta+\xi| - 2|\sigma-\xi| \geq 4c|\xi| \gtrsim 2^j.
\end{align*}
Hence, in the case where $\sigma$ and $\xi-\eta-\sigma$ have the same sign, we can conclude from the identity
\begin{align*}
 \frac{1}{\partial_\sigma \phi_2} = \frac{\jap{\sigma}\jap{\xi-\eta-\sigma} ( \sigma \jap{\xi-\eta-\sigma} + (\xi-\eta-\sigma) \jap{\sigma})}{(\xi-\eta) (\sigma - (\xi-\eta-\sigma))}
\end{align*}
that overall in this regime we have 
\begin{equation} \label{equ:J4c_bound_phase_inverse}
 \Bigl| \partial_\eta^{m_1} \partial_\sigma^{m_2} \frac{1}{\partial_\sigma \phi_2(\xi, \eta, \sigma)} \Bigr| \lesssim 2^{2n} 2^{-(m_1+m_2) n}
\end{equation}
for all integers $0 \leq m_1, m_2 \leq 10$.
Then we integrate by parts in $\sigma$ to obtain that
\begin{align*}
 \calJ^{(4)}_{(c)} &= \calJ^{(4)}_{(c1)} + \calJ^{(4)}_{(c2)} + \calJ^{(4)}_{(c3)},
\end{align*}
where
\begin{align*}
 \calJ^{(4)}_{(c1)} &:= - \sum_{0 \leq k, n \lesssim j} \sum_{\ell \sim j} \jxi^{\hf} \iint e^{i t \phi_2(\xi, \eta, \sigma)} \frac{1}{it \partial_\sigma \phi_2} \partial_\sigma \widehat{f_{\sim n}}(t, \xi-\eta-\sigma) \widehat{\bar{f}_{\sim k}}(t, \eta) \widehat{f_{\sim \ell}}(t, \sigma) \times \\ 
 &\qquad \qquad \qquad \qquad \qquad \qquad \qquad \qquad \qquad \qquad \qquad \qquad \times \chi_c\Bigl(\frac{\sigma}{\xi}\Bigr) \Bigl( 1 - \chi_{8c} \Bigl(\frac{\eta}{-\xi}\Bigr) \Bigr) \psi_k(\eta) \psi_\ell(\sigma) \, \ud \eta \, \ud \sigma, \\
 \calJ^{(4)}_{(c2)} &:= - \sum_{0 \leq k, n \lesssim j} \sum_{\ell \sim j} \jxi^{\hf} \iint e^{i t \phi_2(\xi, \eta, \sigma)} \frac{1}{it \partial_\sigma \phi_2} \widehat{f_{\sim n}}(t, \xi-\eta-\sigma) \widehat{\bar{f}_{\sim k}}(t, \eta) \partial_\sigma \widehat{f_{\sim \ell}}(t, \sigma) \times \\ 
 &\qquad \qquad \qquad \qquad \qquad \qquad \qquad \qquad \qquad \qquad \qquad \qquad \times \chi_c\Bigl(\frac{\sigma}{\xi}\Bigr) \Bigl( 1 - \chi_{8c} \Bigl(\frac{\eta}{-\xi}\Bigr) \Bigr) \psi_k(\eta) \psi_\ell(\sigma) \, \ud \eta \, \ud \sigma, \\
 \calJ^{(4)}_{(c3)} &:= - \sum_{0 \leq k, n \lesssim j} \sum_{\ell \sim j} \jxi^{\hf} \iint e^{i t \phi_2(\xi, \eta, \sigma)} \partial_\sigma \Bigl[  \frac{1}{it \partial_\sigma \phi_2} \chi_c\Bigl(\frac{\sigma}{\xi}\Bigr) \Bigl( 1 - \chi_{8c} \Bigl(\frac{\eta}{-\xi}\Bigr) \Bigr) \psi_k(\eta) \psi_\ell(\sigma) \Bigr] \times \\
 &\qquad \qquad \qquad \qquad \qquad \qquad \qquad \qquad \qquad \qquad \qquad \qquad \times \widehat{f_{\sim n}}(t, \xi-\eta-\sigma) \widehat{\bar{f}_{\sim k}}(t, \eta) \widehat{f_{\sim \ell}}(t, \sigma) \, \ud \eta \, \ud \sigma.
\end{align*}
In view of~\eqref{equ:J4c_bound_phase_inverse} we have for $0 \leq k, n \lesssim j$ and $\ell \sim j$ that
\begin{align*}
 \biggl\| \iint e^{ix\eta} e^{iy\sigma} \frac{1}{\partial_\sigma \phi_2} \chi_c\Bigl(\frac{\sigma}{\xi}\Bigr) \Bigl( 1 - \chi_{8c} \Bigl(\frac{\eta}{-\xi}\Bigr) \Bigr) \psi_k(\eta) \psi_\ell(\sigma) \, \ud \eta \, \ud \sigma \biggr\|_{L^1_{x,y}(\bbR\times\bbR)} \lesssim 2^k 2^j.
\end{align*}
Hence, by Lemma~\ref{lem:pseudoproduct_op_bound} we may bound 
\begin{align*}
 \bigl|\calJ^{(4)}_{(c1)}\bigr| + \bigl|\calJ^{(4)}_{(c2)}\bigr| &\lesssim \frac{2^{\frac{1}{2}j}}{t} \sum_{0 \leq k, n \lesssim j} 2^k 2^j 2^{-2n} 2^{-2j} \bigl\| \jn L v(t) \bigr\|_{L^2_x} \bigl\| \jn^2 v(t) \bigr\|_{L^2_x} \|v_{\sim k}(t)\|_{L^\infty_x}.
\end{align*}
Using again that $\|v_{\sim k}(t)\|_{L^\infty_x} \lesssim 2^{-\frac{3}{4} k} \bigl\| \jn^2 v(t) \bigr\|_{L^2_x}^\hf \|v(t)\|_{L^\infty_x}^\hf$ for any $k \geq 0$,
we arrive at the desired estimate 
\begin{align*}
 \bigl|\calJ^{(4)}_{(c1)}\bigr| + \bigl|\calJ^{(4)}_{(c2)}\bigr| &\lesssim \frac{2^{-\frac{1}{2}j}}{t} \sum_{0 \leq k, n \lesssim j} 2^{\frac{1}{4} k} 2^{-2n} \bigl\| \jn L v(t) \bigr\|_{L^2_x} \bigl\| \jn^2 v(t) \bigr\|_{L^2_x}^{\thf} \|v(t)\|_{L^\infty_x}^\hf \\
 &\lesssim \frac{N(T)^3}{t^{\frac{5}{4}-\frac{5}{2}\delta}} 2^{-\frac14 j} \lesssim \frac{N(T)^3}{t^{\frac{5}{4}-\frac{5}{2}\delta}}.
\end{align*}
Analogously, we derive that $|\calJ^{(4)}_{(c3)}| \lesssim N(T)^3 t^{-(\frac{5}{4}-\frac{5}{2}\delta)}$.

\medskip 

Putting all of the above estimates together, we find that 
\begin{align*}
 II = \frac{3}{t} \jxi^{\frac{7}{2}} \bigl| \hat{f}(t,\xi) \bigr|^2 \hat{f}(t,\xi) + \calO \biggl( \frac{N(T)^3}{t^{1+\nu}} \biggr)
\end{align*}
for $\nu = \min \, \{ \frac{1}{2} - 2 \delta, \frac{1}{4} - \frac{5}{2} \delta,  \alpha - 3 \delta \}$ with $0 < \delta \ll \alpha \ll 1$ sufficiently small. This finishes the stationary phase analysis of the oscillatory integral $II$ and thus concludes the proof of Lemma~\ref{lem:ode_weighted_profile}.
\end{proof}

\subsection{Proof of Theorem~\ref{thm:nonresonant}}

After the preparations in the previous subsections it is now an easy task to infer the asymptotic behavior of the solution $v(t)$ to~\eqref{equ:first_order_kg} and complete the proof of Theorem~\ref{thm:nonresonant}.

\begin{proof}[Proof of Theorem~\ref{thm:nonresonant}]

 By time reversal symmetry it suffices to consider positive times.
 For sufficiently small initial data we can propagate the bounds on the norms of the solution $v(t)$ in the bootstrap quantity $N(T)$ for short times. We may therefore assume that $N(1) \lesssim \varepsilon$, and turn to proving global-in-time a priori bounds for the solution $v(t)$ to~\eqref{equ:first_order_kg}.

 Let $T \geq 1$. Then we conclude from the asymptotics of the Klein-Gordon propagator stated in Lemma~\ref{lem:KG_propagator_asymptotics} as well as from the a priori bounds established in Proposition~\ref{prop:slow_growth_H2v}, Proposition~\ref{prop:slow_growth_H1L}, and Proposition~\ref{prop:Linfty_bound_profile} that
 \begin{equation} \label{equ:nonresonant_apriori_Linfty}
 \begin{aligned}
  \sup_{1 \leq t \leq T} \, t^{\frac{1}{2}} \|v(t)\|_{L^\infty_x} &\lesssim \sup_{1 \leq t \leq T} \, \bigl\| \jap{\xi}^{\frac{3}{2}} \hat{f}(t,\xi) \bigr\|_{L^\infty_\xi} + \sup_{1 \leq t \leq T} \, \frac{1}{t^{\frac{1}{8}}} \Bigl( \bigl\| (\jn L v)(t) \bigr\|_{L^2_x} + \bigl\| (\jn^2 v)(t) \bigr\|_{L^2_x} \Bigr) \\
  &\lesssim \bigl\| \jxi^{\thf} \hat{f}(1,\xi) \bigr\|_{L^\infty_\xi} + \|v_0\|_{H^2_x} + \| x v_0 \|_{H^2_x} + \|v_0\|_{H^1_x}^2 + N(T)^2 \\
  &\lesssim \varepsilon + N(T)^2.
 \end{aligned}
 \end{equation}
 Combining the estimate~\eqref{equ:nonresonant_apriori_Linfty} with the a priori bounds from Proposition~\ref{prop:slow_growth_H2v}, Proposition~\ref{prop:slow_growth_xv}, Proposition~\ref{prop:slow_growth_H1L}, and Proposition~\ref{prop:Linfty_bound_profile}, we conclude that
 \begin{equation*}
  N(T) \lesssim \varepsilon + N(T)^2. 
 \end{equation*}
 A standard continuity argument now yields that there exists an absolute constant $\varepsilon_0 > 0$ such that if $0 < \varepsilon \leq \varepsilon_0$, then we obtain the global a priori bound
 \begin{align*}
  &\sup_{t \geq 0} \, \biggl\{ \jt^{\frac{1}{2}} \|v(t)\|_{L^\infty_x} + \jt^{-\delta} \bigl\| \jn^2 v(t) \bigr\|_{L^2_x} + \jt^{-\delta} \bigl\| \jn L v(t) \bigr\|_{L^2_x} 
  + \jt^{-1-\delta} \|x v(t)\|_{L^2_x} + \bigl\| \jap{\xi}^{\frac{3}{2}} \hat{f}(t,\xi) \bigr\|_{L^\infty_\xi} \biggr\} \lesssim \varepsilon.
 \end{align*}
 The latter includes the sharp decay estimate~\eqref{equ:thm_sharp_decay} asserted in Theorem~\ref{thm:nonresonant}.
 Moreover, we observe that a standard by-product of the proof of the a priori bound~\eqref{equ:Linfty_bound_profile} for the profile in Proposition~\ref{prop:Linfty_bound_profile} is the existence of a limit profile $\widehat{W} \in L^\infty_\xi$ such that 
 \begin{equation*}
  \Bigl\| \jxi^\thf \hat{f}(t,\xi) - \widehat{W}(\xi) e^{-i \frac{3 \beta_0}{2} \jxi^{-1} |\widehat{W}(\xi)|^2 \log(t)} \Bigr\|_{L^\infty_\xi} \lesssim \frac{\varepsilon^2}{t^{\nu}}, \quad t \geq 1,
 \end{equation*}
 for the small constant $0 < \nu \ll 1$ from the statement of Lemma~\ref{lem:ode_weighted_profile}.
 Then the asserted asymptotics~\eqref{equ:thm_asymptotics} of the solution $v(t)$ to~\eqref{equ:first_order_kg} follow from Lemma~\ref{lem:KG_propagator_asymptotics}.
\end{proof}

\bibliographystyle{amsplain}
\bibliography{references_v2}

\end{document}